\documentclass{amsart}
\usepackage{dsfont}
\usepackage{graphicx}
\usepackage[usenames,dvipsnames,svgnames,table]{xcolor}
\usepackage[utf8]{inputenc}
\usepackage[T1]{fontenc}
\usepackage{ tipa }
\usepackage{textgreek, bm}
\usepackage{color}
\usepackage{mathrsfs, enumitem}  
\usepackage{stmaryrd}  
\usepackage{amsmath,amsxtra,amsthm,amssymb,xr,enumerate,fullpage}
\usepackage[all]{xy}
\usepackage[bbgreekl]{mathbbol}
\usepackage{bbm}
\usepackage{tikz}
\newcommand*\circled[1]{\tikz[baseline=(char.base)]{
            \node[shape=circle,draw,inner sep=2pt] (char) {#1};}}

\makeatletter
\newcommand{\mylabel}[2]{#2\def\@currentlabel{#2}\label{#1}}
\makeatother
\newcommand{\refsymbolone}{{\ensuremath\circled{1}}}
\newcommand{\refsymboltwo}{{\ensuremath\circled{2}}}
\newcommand{\refsymbolthree}{{\ensuremath\circled{3}}}
\newcommand{\refsymbolfour}{{\ensuremath\circled{4}}}

\DeclareMathSymbol{\invques}{\mathord}{operators}{`>}
\DeclareUnicodeCharacter{00BF}{\tmquestiondown}
\DeclareRobustCommand{\tmquestiondown}{%
  \ifmmode\invques\else\textquestiondown\fi
}
\usepackage{verbatim}

\newtheorem{theorem}{Theorem}[section]
\newtheorem{lemma}[theorem]{Lemma}

\newtheorem{proposition}[theorem]{Proposition}
\newtheorem{corollary}[theorem]{Corollary}
\newtheorem{defn}[theorem]{Definition}

\newtheorem{remark}[theorem]{Remark}

\newtheorem{thmx}{Theorem}

\newcommand{\pr}{\mathrm{pr}}

\newcommand{\Gal}{\operatorname{Gal}}
\newcommand{\Fil}{\operatorname{Fil}}

\newcommand{\DD}{\mathbb{D}}

\newcommand{\QQ}{\mathbb{Q}}
\newcommand{\Qp}{\mathbb{Q}_p}
\newcommand{\Zp}{\mathbb{Z}_p}
\newcommand{\ZZ}{\mathbb{Z}}
\renewcommand{\AA}{\mathbb{A}}

\newcommand{\bbG}{\mathbb{G}}

\newcommand{\ord}{\mathrm{ord}}

\newcommand{\cH}{\mathcal{H}}
\newcommand{\cO}{\mathcal{O}}
\newcommand{\cX}{\mathcal{X}}

\newcommand{\Iw}{\mathrm{Iw}}

\newcommand{\GL}{\mathrm{GL}}

\newcommand{\cyc}{\textup{cyc}}

\newcommand{\Hom}{\mathrm{Hom}}

\newcommand{\LL}{\Lambda}
\newcommand{\TT}{\mathbb{T}}

\newcommand{\f}{\textup{\bf f}}

\newcommand{\lra}{\longrightarrow}

\newcommand{\res}{\textup{res}}
\newcommand{\TSym}{\textup{TSym}}

\newcommand{\cF}{\mathcal{F}}

\newcommand{\Tw}{\mathrm{Tw}}

\newcommand{\mom}{\mathrm{mom}}

\newcommand{\cC}{\mathcal{C}}

\newcommand{\cW}{\mathcal{W}}

\newcommand{\bz}{\mathbf{z}}

\newcommand{\DdagrigA}{\mathbf D^\dagger_{\mathrm{rig},A}}
\newcommand{\DdagrigE}{\mathbf D^\dagger_{\mathrm{rig},E}}
\newcommand{\bD}{\mathbf{D}}

\newcommand{\CR}{\mathscr{R}}

\newcommand{\Spm}{\mathrm{Spm}}
\newcommand{\Dst}{{\mathbf D}_{\mathrm{st}}}
\newcommand{\Dc}{{\mathbf D}_{\mathrm{cris}}}
\newcommand{\DCc}{{\mathscr D}_{\mathrm{cris}}}
\newcommand{\Exp}{\mathrm{Exp}}
\newcommand{\CH}{\mathscr H}
\newcommand{\Log}{\mathfrak{L}}
\newcommand{\spe}{\mathrm{sp}}

\newcommand{\umT}{{\mathscr T}}
\newcommand{\uLambda}{\underline\Lambda}
\newcommand{\LLU}{{O}^\circ_U}
\newcommand{\LLUp}{{O}_U}
\newcommand{\sheafLLU}{{\mathscr O}^\circ_U} 
\newcommand{\w}{w}
\newcommand{\M}{M}
\newcommand{\bP}{\mathscr P}
\newcommand{\bV}{\mathscr V}
\newcommand{\cusp}{\textrm{cusp}}

\usepackage{hyperref}

\begin{document}

\title{ {I}\lowercase{nterpolation of} {B}\lowercase{eilinson}--{K}\lowercase{ato elements and} \lowercase{$p$-adic} {$L$}-\lowercase{functions}}

\begin{abstract}
Our objective in this series of two articles, of which the present article is the first, is to give a Perrin-Riou-style construction of $p$-adic $L$-functions (of Bella\"iche and Stevens) over the eigencurve. As the first ingredient, we interpolate the Beilinson--Kato elements over the eigencurve (including the neighborhoods of $\theta$-critical points). Along the way, we prove \'etale variants of Bella\"iche's results describing the local properties of the eigencurve. We also develop the local framework to construct and establish the interpolative properties of these $p$-adic $L$-functions away from $\theta$-critical points.
\end{abstract}

\author{Denis Benois}
\address{Denis Benois\newline Institut de Math\'ematiques, Universit\'e de Bordeaux  \\ 351, Cours de la Lib\'eration 33405  \\ Talence, France
}
\email{denis.benois@math.u-bordeaux1.fr}

\author{K\^az\i m B\"uy\"ukboduk}
\address{K\^az\i m B\"uy\"ukboduk\newline UCD School of Mathematics and Statistics\\ University College Dublin\\ Ireland}
\email{kazim.buyukboduk@ucd.ie}

\dedicatory{ To Bernadette Perrin-Riou on the occasion of her 65th birthday, with admiration.}

\subjclass[2010]{11F11, 11F67 (primary); 11R23 (secondary)}
\keywords{Eigencurve, $\theta$-criticality, triangulations, Beilinson--Kato elements, $p$-adic $L$-functions}

\maketitle
\tableofcontents
\section{Introduction}

Let $p\ge 5$ be a prime number and let us fix forever an embedding $\iota_\infty: \overline{\QQ}\to \mathbb{C}$ as well as an isomorphism $\iota:  \mathbb{C} \xrightarrow{\sim}  \mathbb{C}_p$. Let us also put $\iota_p:=\iota\circ \iota_\infty$. Let $f=\sum_{n}a_n(f)q^n \in S_{k+2}(\Gamma_1(N)\cap \Gamma_0(p),\varepsilon)$ be a cuspidal eigenform (for all Hecke operators $\{T_\ell\}_{\ell\nmid N}$ and $\{U_\ell, \langle \ell\rangle\}_{\ell \mid Np}$) of weight $k+2\geq 2$ with $p\nmid N$. When $\ord_p\left(\iota_p(a_p(f))\right)<k+1$, Amice--V\'elu in \cite{AmiceVelu} and Vi\v{s}ik in \cite{Vishik} have given a construction of a $p$-adic $L$-function $L_p(f,s)$, which is characterized with the property that it interpolates the critical values of Hecke $L$-functions attached to (twists) of $f$. 

The analogous result in the extreme case when $\ord_p\left(\iota_p(a_p(f))\right)=k+1$ (in which case we say that $f$ has critical slope) was established by Pollack--Stevens in \cite{PollackStevensJLMS} and Bella\"iche~\cite{bellaiche2012}. Note that if $\ord_p\left(\iota_p(a_p(f))\right)=k+1$  then $f$ is necessarily $p$-old unless $k=0$.  It is worthwhile to note that the $p$-adic $L$-functions of Pollack--Stevens and Bella\"iche are not characterized in terms of their interpolation property, but rather via the properties of the $f$-isotypic Hecke eigensubspace of the space of modular symbols. The work of Pollack--Stevens assumes in addition that $f$ is not in the image of the $p$-adic $\theta$-operator $\theta^{k+1}:=(q(d/dq))^{k+1}$ on the space of overconvergent modular forms of weight $-k$ (i.e., $f$ is not $\theta$-critical); Bella\"iche's work removes this restriction.

Both constructions in \cite{PollackStevensJLMS} and \cite{bellaiche2012} take place in Betti cohomology. Our objective in this series of two articles, of which the present article is the first, is to recover these results in the context of $p$-adic (\'etale) cohomology. More precisely, we shall recast the results of Bella\"iche and Pollack--Stevens in terms of the Beilinson--Kato elements and the triangulation over the Coleman--Mazur--Buzzard eigencurve. Along the way, we interpolate  the Beilinson--Kato elements over a neighborhood on the eigencurve (including the neighborhoods of $\theta$-critical points).


\subsection{Set up}
\label{subsec_setup}
We put $\Gamma_p=\Gamma_1(N)\cap \Gamma_0(p)$ and let $f_0^{\alpha_0} \in S_{k_0+2}(\Gamma_p)$ be a $p$-stabilized cuspidal eigenform, where $U_pf_0^{\alpha_0}=\alpha_0 f_0^{\alpha_0}$ and $N$ is coprime to $p$. We let $f_0$ be the newform associated to $f_0^{\alpha_0}$ (note that it may happen that $f_0=f_0^{\alpha_0}$, in which case $f_0^{\alpha_0}$ is not of critical slope).  We fix a real number $\nu\geq v_p(\alpha_0)$, where $v_p(\cdot)$ is the $p$-adic valuation on $\overline{\QQ}_p$ normalized so that $v_p(p)=1$.

We shall call $\Hom_{\rm cts}(\ZZ_p^\times,\bbG_m)$ the weight space, which we think of as a rigid analytic space over $\QQ_p$. Let $\mathcal{W}$ be a nice affinoid neighborhood (in the sense of \cite{bellaiche2012}, Definition 3.5)  about $k_0$  of the weight space, which is adapted to slope $\nu$ (in the sense of \cite{bellaiche2012}, \S3.2.4). We will adjust our choice of $\mathcal W$ on shrinking it as necessary for our arguments.

Let $\cC$ be the Coleman--Mazur--Buzzard eigencurve and let $x_0\in \cC$ be the point that corresponds to $f_0^{\alpha_0}$. We let $\cC_{\mathcal{W},\nu}\subset \cC$ denote an open affinoid subspace of the eigencurve that lies over $\mathcal W$ and $U_p$ acts by slope at most $\nu$. By shrinking $\mathcal W$ as necessary, there is a unique  connected component $\mathcal{X} \subset \cC_{\mathcal{W},\nu}$ that contains $x_0$, which is an affinoid neighborhood of $x_0$.

For any $E$-valued point $x\in \mathcal{X}(E)$ (where $E$ is a sufficiently large extension of $\QQ_p$, which contains the image of the Hecke field of $f_0^{\alpha_0}$ under the fixed isomorphism $\iota$) of classical weight $w(x) \in \ZZ_{\geq 0}$ in the irreducible component $\mathcal{X} \subset \cC$, we let $f_x \in S_{w(x)+2}(\Gamma)$ denote the corresponding $p$-stabilized eigenform. We let $V_{f_x}'$ denote Deligne's representation attached to $f_x$; see \S\ref{subsubsec_621} for its precise description. There is a natural free $\cO_{\mathcal X}$-module $V_\mathcal{X}'$ of rank $2$, which is equipped with an $\cO_{\mathcal X}$-linear continuous $G_\QQ$-action such that $V_{\mathcal X}'\otimes_{x}E\xrightarrow{\sim}V_{f_x}'$ (c.f. \eqref{eqn_defn_prop_andreattaiovitastevens2015_Prop318} and \eqref{eqn: specialization of big modular representation} below).

As in \cite[\S5]{kato04}, let $\xi$ denote either the symbol $a(B)$ with $a,B\in\ZZ$ and $B\geqslant 1$ or an element of $\mathrm{SL}_2(\ZZ)$.  For each positive integer $m$ coprime to $p$, we set $S:=\begin{cases} \text{primes} (mBp),
& \textrm{if $\xi=a(B)$}\\
\text{primes} (mNp),
& \textrm{if $\xi\in \mathrm{SL}_2(\ZZ)$}
\end{cases}$\,,
where ${\rm primes}(M)$ stands for the set of prime divisors of $M$. Let $(c,d)$ be a pair of positive integers satisfying the following conditions:
$ (cd,6)=1=(d,N)$ and $\text{prime} (cd) \cap S=\emptyset$. For each integer $r$ and natural number $n$, we let 
$${}_{c,d}{\mathrm{BK}}_{N,mp^n} (f_x,j, r, \xi)\in H^1(\QQ(\zeta_{mp^n}),V_{f_x}'(2-r))$$
denote the Beilinson--Kato element associated to the eigenform $f_x$, given as in Theorem~\ref{thm_control_thm_interpolation_BK}(iv) (see also \S\ref{subsec:Beilinson--Kato} for details).

For any abelian group $G$, let us denote by $X(G)$ its character group. We also put $\LL(G):=\ZZ_p[[G]]$ for its completed group ring. We shall denote by $\LL(G)^\iota$ for the free $\LL(G)$-module of rank one on which $G$ acts via the character $g\stackrel{\iota}{\mapsto} g^{-1}\in \LL(G)^\times$.

For each positive integer $m$, let us put $\Gamma_{\QQ (\zeta_m)}:=\Gal (\QQ (\zeta_{mp^\infty})/\QQ (\zeta_m))$ and set $\Gamma:=\Gamma_{\QQ}$. We denote by $\chi$ the $p$-adic cyclotomic character. We put $\mathscr{H}_r(\Gamma)$ for the image of tempered distributions of order $\nu$ under the Amice transform and define $\mathscr{H}(\Gamma):=\varinjlim \mathscr{H}_r(\Gamma)$. We also set $\mathscr{H}_?(\Gamma):=\mathscr{H}(\Gamma)\,\widehat{\otimes}\, (?)$, where $(?)=E,\cO_{\cW},\cO_{\mathcal X}$.

Fix a generator $\varepsilon$ of $\ZZ_p(1)$. 


\subsection{Main results}
\label{subsec_intro_main_results} 

We will briefly overview of the results in this article. In a nutshell, our work has two threads: The first concerns the interpolation of Beilinson--Kato elements along the eigencurve $\cC$, the second concerns $p$-local aspects, such as the (properties of the) triangulation over $\cC$ and the formalism of Perrin-Riou exponential maps. The first part of Theorem~\ref{thmA} below belongs to the first thread and corresponds to Theorem~\ref{thm_control_thm_interpolation_BK} below. The second part corresponds to Theorem~\ref{thm_interpolative_properties_bis} in the main body of our article and dwells on the second thread, granted the first. 

\begin{thmx}
\label{thmA}
\item[i)] For each $m\geq 1$ and coprime to $p$, there exists a cohomology class 
$${}_{c,d}{\mathbb{BK}}^{[\mathcal X]}_{N,m} (j, \xi) \in H^1(\QQ(\zeta_m), V^\prime_{\mathcal X}\widehat{\otimes}\,
\Lambda (\Gamma_{\QQ (\zeta_m)})^\iota (1))$$
which interpolates the Beilinson--Kato classes ${}_{c,d}{\mathrm{BK}}_{N,mp^n} (f_x,j, r, \xi)$ as $x \in \mathcal{X}^{\rm cl}(E)$ and integers $r,n$ vary.

\item[ii)] { Suppose that $\mathcal{X}$ is \'etale over $\mathcal{W}$.} We let $\mathbb{BK}_{N}^{[\mathcal X]}(j,\xi)\in   H^1 \left (\ZZ [1/S],V^\prime_{\mathcal X} \widehat{\otimes}\,
\Lambda (\Gamma)^\iota (1) \right )$ denote the partially normalized Beilinson--Kato element, given as in Definition~\ref{defn: normalized BK elements} and Proposition~\ref{prop_partial_normalization_of_BK_elements} below. Let $L_{p}^{\pm}(\mathcal X;j,\xi) \in \mathscr{H}_{\cO_{\mathcal X}}(\Gamma)$ denote the images of the $\pm$-parts of the class $\mathbb{BK}_{N}^{[\mathcal X]}(j,\xi)$ under the Perrin-Riou dual exponential map  $($see Definition~\ref{def_two_var_padicL_function_bis}$)$. On shrinking the neighborhood $\mathcal X$ of $x_0$ as neccesary, the following hold true. 
\begin{itemize}
\item[a)] There exist { $(j_+,\xi_+)$ and $(j_-,\xi_-)$} such that { $L_{p}^{+}(\mathcal X;j_+, \xi_+, x_0)$ and
$L_{p}^{-}(\mathcal X;j_-, \xi_-, x_0)$} are nonzero elements of $\CH_E(\Gamma)$. 
\item[b)] Assume that { $(j_\pm,\xi_\pm)$} satisfy the conditions in (a). Then for each 
$x \in \mathcal{X}^{\rm cl}(E)$ such that $v_p(\alpha (x))<w(x)+1$, the $p$-adic $L$-functions { $L_{p}^{\pm}({\mathcal X} ;j_\pm, \xi_\pm, x)$} agree with the Manin--Vi{\v{s}}ik  $p$-adic $L$-functions attached to $f_x$, up to multiplication by  $D^\pm\mathcal{E}_N(x)$ where $D^\pm\in E^\times$ and $\mathcal{E}_N(x)$ is the product of bad Euler factors, given as in \eqref{eqn: E(x)-factor}. 

\item[c)] Assume that $v_p(\alpha_0)=k_0+1,$ but $f$ is not $\theta$-critical. Then  { $L_{p}^{\pm}({\mathcal X};j_\pm, \xi_\pm, x_0)$} agree with the one-variable $p$-adic $L$-functions of  Pollack--Stevens \cite{PollackStevensJLMS} up to multiplication by $D^\pm\mathcal{E}_N(x)$, where $D^\pm\in E^\times$.

\item[d)] Let $L_p({\mathcal X},\Phi^{\pm}) $ denote the two-variable $p$-adic $L$-functions of Bella\"{\i}che and Stevens associated to modular symbols $\Phi^{\pm}\in \mathrm{Symb}_{\Gamma_p}^\pm(\mathcal X)$  $($c.f. \cite{bellaiche2012}, Theorem~3$)$. Then there exist  functions $u^\pm (x)\in O_{\mathcal X}^\times$
such that 
\[
{ L_{p}^{\pm}({\mathcal X};j_\pm, \xi_\pm)=u^\pm(x){\mathcal E}_N(x)L_p({\mathcal X},\Phi^{\pm})\,.}
\]
\end{itemize}
\end{thmx}

See Definition~\ref{defn_big_BK_lambdaadicsheaf}(iii) where the class ${}_{c,d}{\mathbb{BK}}^{[\mathcal X]}_{N,m} (j, \xi) $ is introduced and see Definition~\ref{defn: normalized BK elements} where we introduce its partial normalization (as well as Appendix~\ref{sec_appendix_integrality} where we establish the regularity properties of this normalization). We call the $p$-adic $L$-function $L_{p}^{\pm}(\mathcal X;j,\xi) \in \mathscr{H}_{\cO_{\mathcal X}}(\Gamma)$ given as in Definition~\ref{def_two_var_padicL_function_bis}(i) (which is denoted by $L_{p,\eta}^{\pm}(\f;j, \xi)$ in the main body of our text) the arithmetic $p$-adic $L$-function. 

\begin{remark} 
\label{remark_intro_hansen}
Results similar to Theorem~\ref{thmA} are proved in the recent paper of Wang~\cite{Wang2021} via different methods. Some parts of this result has been also announced in the independent preprints of Hansen~\cite{HansenBKEigencurve} and Ochiai~\cite{OchiaiIMC}. Our approach to establish Theorem~\ref{thmA}(i) (which applies even when $\mathcal{X}$ is not \'etale over $\cW$; this is crucial for our results in the $\theta$-critical case) is a synthesis of the techniques of \cite{LZ1} and \cite{HansenBKEigencurve}, relying on the overconvergent \'etale cohomology of Andreatta--Iovita--Stevens~\cite{andreattaiovitastevens2015} and Kings' theory of $\LL$-adic sheaves developed in \cite{Kings2015}. Along the way, we verify (relying on the results of Ash--Stevens and Pollack--Stevens) that the construction of the local pieces (including neighborhoods of $\theta$-critical points, where $\mathcal X$ is no longer \'etale over $\mathcal W$) of the cuspidal eigencurve using different types of distribution spaces over the weight space, or compactly vs non-compactly supported cohomology produce the same end result. See \S\ref{subsec_Bellaiche_revised}, \S\ref{subsubsec_wideopen_vs_bellaiche} and \S\ref{Overconvergent etale sheaves} for further details (most particularly, Lemma~\ref{lemma: shrinking wide open neighborhood}, Proposition~\ref{claim}(ii) and Proposition~\ref{prop_andreattaiovitastevens2015_Prop318}) concerning this technical point, which we believe is of independent interest.  
\end{remark}

\begin{remark}
We note that in the particular case when $f_0=f^{\alpha_0}_0$ is a newform of level $\Gamma_0(Np)$ and weight $k_0+2$ with $\alpha_0=a_p(f_0)=p^{k_0/2}$, the conclusions of Theorem~\ref{thm_interpolative_properties_bis} play a crucial role in \cite{benoisbuyukboduk}. \end{remark}

\subsubsection{$\theta$-critical case} 
\label{subsubsec_intro_theta_critical}We conclude \S\ref{subsec_intro_main_results} with a brief summary of our results in our companion article~\cite{BB_CK1_B}, where we focus on the $\theta$-critical scenario (i.e. in the situation when $\mathcal X$ is \emph{not} \'etale over $\cW$) and give an \'etale construction of Bella\"iche's $p$-adic $L$-function. 

The first key ingredient in \cite{BB_CK1_B} is Theorem~\ref{thmA}(i), namely the construction of a big Beilinson--Kato class about a $\theta$-critical point on the eigencurve. The local aspects turn out to be significantly more challenging when $\mathcal{X}$ fails to be \'etale over $\cW$. To prove that the Perrin-Riou style $p$-adic $L$-functions $L_p^{\pm}({\mathcal X};j,\xi)$, defined in an analogous way, has the required properties, we introduce in op. cit. a new local argument (called the ``eigenspace transition via differentiation'')  and prove the following results (which we state in vague form to avoid digression and refer readers to  \cite{BB_CK1_B} for details):

\begin{thmx}
\label{thmB} Suppose that the ramification index of $\mathcal X$ over $\mathcal W$ at $x_0$ equals\footnote{In an unpublished note (see ~\cite{CMLBellaiche} Proposition 1), Bella\"iche explains that a conjecture of Jannsen combined with Greenberg's conjecture (that ``locally split implies CM'') and the important of result of Breuil-Emerton (``$\theta$-critical implies locally split'') would yield $e=2$. We are grateful to R. Pollack for bringing Bella\"iche's note to our attention.} to $2$. As before, we let $L_{p}^{\pm}(\mathcal X;j,\xi) \in \mathscr{H}_{\cO_{\mathcal X}}(\Gamma)$ denote the images of the $\pm$-parts of the class $\mathbb{BK}_{N}^{[\mathcal X]}(j,\xi)$ under the Perrin-Riou dual exponential map. There exists two pairs { $(j_+,\xi_+)$ and $(j_-,\xi_-)$} with the following properties:
\item[i)] We have  { $L_p^{\pm}(x_0; j_\pm, \xi_\pm,\rho\chi^r)=0$}
for all integers $1\leqslant r\leqslant k_0+1$ and characters $\rho \in X(\Gamma)$ of finite order.

\item[ii)] We define the improved arithmetic $p$-adic $L$-functions at the critical point $x_0$ on setting
$$L_p^{[1],\pm}(x_0; j_\pm, \xi_\pm):=\frac{\partial}{\partial X}\, L_{p}^{\pm}(X;j_\pm,\xi_\pm)\Big{\vert}_{X=0} \in \mathscr{H}_E(\Gamma)$$
Here, as in \cite[\S4.4]{bellaiche2012}, we denote by $X$ a fixed choice of a uniformizer of $\mathcal{X}$ about $x_0$ and consider the $p$-adic $L$-functions $L_{p}^{\pm}(\mathcal X;j,\xi)$ in the neighborhood $\mathcal X$ of $x_0$ as the functions $L_{p}^{\pm}(X;j,\xi)$ with $X$ in a neighborhood of $0$. { Then  the improved arithmetic $p$-adic $L$-functions verify the same interpolation property that Bella\"iche's improved $p$-adic $L$-functions do, up to the bad Euler factors $\mathcal{E}_N(x_0)$ and constants that depend only on the choices of Shimura periods.}
\end{thmx}

This theorem supplies us with a new construction of Bella\"iche's $p$-adic $L$-functions (with Euler factors at primes dividing the tame conductor removed). One of the consequences of the \'etale construction of the $p$-adic $L$-functions is the leading term formulae for these $p$-adic $L$-functions. These will be explored in the sequels to the present article.


\subsubsection{Layout} We close our introduction reviewing the layout of our article. 
 
 After a very general preparation in \S\ref{sec_Bellaiche_axiomatized} (where we axiomatise various constructions in \cite{bellaiche2012}, \S4.3), we give a general overview of triangulations in \S\ref{sec_triangulations}. In \S\ref{subsec_PR_map_general}, we also define the Perrin-Riou exponential map, which is one of the crucial inputs defining the ``arithmetic'' $p$-adic $L$-functions. 

We then introduce Perrin-Riou-style (abstract) two-variable $p$-adic $L$-functions in \S\ref{sec_2varpadicLabstract} and study their interpolative properties. These results are later applied in \S\ref{sec_geom_critical_padicL} in the context of the Coleman--Mazur--Buzzard eigencurve and with the Beilinson--Kato element ${}_{c,d}{\mathbb{BK}}^{[\mathcal X]}_{N,m} (j, \xi)$ of Theorem~\ref{thmA}.

 In \S\ref{sec_bellaiche_eigencurve}, we review Bella\"iche's results in \cite{bellaiche2012} on the local description of the eigencurve and prove variants (in \S\ref{subsec_Bellaiche_revised}) of these results involving slightly different local systems as coefficients and non-compactly supported cohomology (at a level of generality that covers also the neighborhoods of $\theta$-critical point). We utilize these in \S\ref{sec_interpolateBK_elements} (together with the work of Andreatta--Iovita--Stevens~\cite{andreattaiovitastevens2015}) to deduce the required properties for the Galois representation $V_{\mathcal X}'$, where the interpolated Beilinson--Kato elements take coefficients in. The variants of Bella\"iche's construction that we discuss  in \S\ref{subsec_Bellaiche_revised} allow us (through Proposition~\ref{prop_andreattaiovitastevens2015_Prop318}) to establish the properties of $V_{\mathcal X}$ as an $\cO_{\mathcal X}$-module, where the images of Perrin-Riou exponential maps take coefficients in.

In \S\ref{subsec_BK_families}, we introduce the interpolated Beilinson--Kato elements (which are denoted by ${}_{c,d}{\mathbb{BK}}^{[\mathcal X]}_{N,m} (j, \xi)$ in Theorem~\ref{thmA}(i)) as part of Definition~\ref{defn_big_BK_lambdaadicsheaf}(iii) (see also \S\ref{subsec_BK_big_normalized} where we introduce their normalized versions). Our construction builds primarily on the ideas in \cite{Kings2015, KLZ2, LZ1}. As we have noted in Remark~\ref{remark_intro_hansen}, our approach in this portions is, in some sense, a synthesis of the techniques of \cite{LZ1} and \cite{HansenBKEigencurve} (which we crucially enhance to apply also about $\theta$-critical points). In \S\ref{sec_geom_critical_padicL}, we apply the general results in \S\ref{sec_2varpadicLabstract} to define the ``arithmetic'' $p$-adic $L$-functions and study their interpolation properties, proving our Theorem~\ref{thmA}(ii).

\subsection{Acknowledgements}
The first named author (D.B.) wishes  to thank the second author (K.B.)  for his invitation to University College Dublin in April 2019 and Bo\v gazi\c ci University of Istanbul in December 2019. This work  was started during these visits. D.B.  was also partially supported  by the  Agence National de Recherche (grant ANR-18-CE40-0029) in the framework of the ANR-FNR project ``Galois representations, automorphic forms and their $L$-functions''. We thank the anonymous referees for carefully reading our work and their very helpful comments, which guided us towards many technical and stylistic improvements to the earlier versions of our article.


\section{Linear Algebra} 
 \label{sec_Bellaiche_axiomatized}

 In \S\ref{sec_Bellaiche_axiomatized}, we  fix some notation and conventions from linear algebra, which will be used in the remainder of the paper (as well as in the companion paper \cite{BB_CK1_B}). We also  axiomatize various constructions in \cite[\S4.3]{bellaiche2012}.

\subsection{Tate algebras}
\label{subsection:generalities about W and X}

Let $E$ be a finite extension of $\Qp.$  Fix an integer $e\geqslant 1$
and denote by $R$ the  Tate algebra $R=E\left <Y/p^{re}\right >,$
where $r\geqslant 0$ is some fixed integer. Let $A=R[X]/(X^e-Y).$
Then $A=E\left <X/p^r\right >.$ 

Set $\mathcal W=\Spm (R)$ and $\mathcal X=\Spm (A).$ We  will consider $\mathcal W$ as a weight space in the following sense.   
Fix an integer $k_0\geqslant 2$ and denote by $D(k_0,p^{re-1})=k_0+p^{re-1}\Zp$ the closed disk with center $k_0$ and radius $1/p^{re-1}.$ We  identify each $y\in D(k_0,p^{re-1})$ with the point of  $\mathcal W$ that corresponds to the maximal ideal 
\begin{equation}
\nonumber
\mathfrak m_y= \left ((1+Y)- (1+p)^{y-k_0} \right ).
\end{equation}
We have a canonical morphism 
\begin{equation}
\label{eqn: the weight map}
w\,:\,\mathcal X\rightarrow \mathcal W,
\end{equation}
which we will call the weight map. 
If $x\in \mathcal X,$ we say that $y=\w (x)$ is the weight of $x.$
Let $x_0\in \mathcal X$ denote the unique point such that $\w(x_0)=k_0.$
Set
\[
\mathcal X^{\mathrm{cl}}=\{x\in \mathcal X  \mid \w(x)\in \ZZ, w (x) \geqslant 2\}. 
\]

\subsection{Generalized eigenspaces} 
\subsubsection{}
Let $\M$ be an $A$-module. For any $x\in \mathcal X(E)$ and $y\in \mathcal W (E)$ define
\begin{equation}
\nonumber
\begin{aligned}
&\M_{x}=\M\otimes_{A} A/\mathfrak m_x ,\\
&\M_{y}=\M\otimes_{R} R/\mathfrak m_{y}.
\end{aligned}
\end{equation}
We denote by
\[
\pi_x\,:\, \M_{w (x)} \lra \M_{x}
\]
the canonical projection. 

Consider $X$ as a function on $\mathcal X$ and denote by $X(x)$ the value 
of $X$ at $x\in \mathcal X (E).$ For any endomorphism $f$ of  $\M,$  set 
$\M [f]=\ker (f).$

\begin{defn} 
We denote by $\M [x]=\M_{\w(x)}[X-X(x)]$, the submodule of $\M_{\w(x)}$ annihilated by $X-X(x)\in A$. We also put $\M [[x]]:=\underset{n\geqslant 1}\cup \M_{\w(x)}[(X-X(x))^n]$ and call it the generalized eigenspace associated to $x$.
\end{defn}

\subsubsection{}
Suppose that $\M\cong A^d$ is a free $A$-module of rank $d$, so that we also have $M\cong \left (R[X]/(X^e-Y)\right )^d$. Let us put $P_x(X):=\dfrac{X^{e}-X(x)^e}{X-X(x)} \in E[X]$. Then, 
\begin{align*}
\M_{\w (x)}&= \left(E[X]\Big{/}(X^{e}-X(x)^e)E[X]\right)^d ,\\
\M [x]&= \left(P_x(X)E[X]\Big{/}(X^{e}-X(x)^e)E[X]\right)^d , \\
\M_x&= \left ( E[X]/(X-X(x)) \right )^d,
\end{align*}
and the multiplication by $P_x(X)$ gives an isomorphism 
\begin{equation}
\nonumber 
\M_x\xrightarrow[\times P_x(X)]{\sim} \M [x].
\end{equation}

We consider separately the following  cases:

\begin{itemize} 
\item[a)] When $x\neq x_0$ (so $\w (x)\neq k_0$), then the polynomial $X^e-X(x)^e$ is separable. In this scenario, the $A$-module $\M_{\w (x)}$ is semi-simple and $\M[x]=P_x(X)\M_{\w (x)}\subset M_{\w (x)}$ is an $A$-direct summand. The restriction of the natural surjection
$
\pi_x: \M_{\w(x)}\lra \M_x
$
to $\M [x]$ is an isomorphism, which we shall denote with the same symbol unless there is a chance of confusion. Also, $\M [[x]]=\M [x]$.

\item[b)]When $x=x_0$, we have $X(x_0)=0$ and $P_{x_0}(X)=X^{e-1}$. In this case, $\M_{k_0}= (E[X]/X^eE[X])^{d}$ as an $A$-module and 
\begin{equation}
\nonumber
\begin{aligned}
&\M_{k_0}= \M [[x_0]] \stackrel{\sim}{\lra} (E[X]/X^eE[X])^{d},\\
&\M [x_0]=\M_{k_0}[X]=X^{e-1}\M_{k_0},\\
&M_{x_0} \stackrel{\sim}{\lra} E^d. 
\end{aligned}
\end{equation}
If $e\geqslant 2,$ the restriction $\pi_{x_0}{\vert_{_{\M[x_0]}}}: \M [x_0]\to \M_{x_0}$ of the surjection $\pi_{x_0}$ is evidently the zero map. 

Observe also that in either of the cases \textup{(a)} or  \textup{(b)} above, $\M[x]$ can be identified with an $A$-equivariant image of $\M_{\w(x)}$ under the multiplication-by-$P_{x}(X)$ map. 
\end{itemize}

\subsection{Specializations}
\label{subsec_specializations}
Let $\M$ be an $A$-module, and let $x\in \mathcal X(E).$ 
For any $m\in \M,$ we denote by $m_x\in \M_x$ the image 
of $m$ under the specialization  map $M \rightarrow M_x.$ 
On the other hand,  $\M_A=\M\otimes_R A$ has the $A$-module structure via the $A$-action on the second factor. For any $x\in \mathcal X(E)$, the specialization
of $\M_A$ at $x$ is
\[
\left (\M_A\right )_{x} :=(\M\otimes_R A)\otimes_{A,x}E=
\M\otimes_{R,\w (x)}E=\M_{\w (x)}.
\]

\begin{defn}
\label{defn_PhiGamma_Big} 
For an $A$-module $\M$ and $x\in \mathcal X(E)$, we let ${\rm sp}_x$ denote the  specialization map $\M_A\rightarrow \left (\M_A\right )_{x}.$
\end{defn}

We need the following lemma. 

\begin{lemma}
\label{lemma_bellaiche_4_14_general}
In the notation of Definition~\ref{defn_PhiGamma_Big}, suppose that $\M$ is a free $A$-module of rank one. Let $m$ denote a generator of $\M$ and put $\Phi:=\sum_{i=0}^{e-1}X^i m\otimes X^{e-1-i}\in \M\otimes_R A$. Then 
\item[i)] We have $(X\otimes 1) \Phi= (1\otimes X)\Phi$.
\item[ii)] The element ${\rm sp}_x(\Phi)\in \M_{\w (x)}$ generates the $E$-vector space $\M [x]$.
\item[iii)] We have $\pi_x\circ{\rm sp}_x(\Phi)=X(x)^{e-1}m_x$.
\end{lemma}
\begin{proof}
The first part is the abstract version of \cite[Lemma 4.13]{bellaiche2012} and the second part is that of \cite[Proposition 4.14]{bellaiche2012}. Final assertion follows from a direct calculation.
\end{proof}
\begin{remark}\label{rem_sp_x_when_e_equals1}
Suppose in this remark that $e=1$, so that $A=R$. In this scenario, we have $M_A=M$ and $M_{w(x)}\xrightarrow[\pi_x]{\sim}M_x$. Moreover, the specialization map ${\rm sp}_x$ is simply the canonical projection $M\to M_x$.
\end{remark}

\subsection{Duality} 
\label{subsec:duality}
Let $\M$ and $\M'$ be two free $A$-modules of finite rank equipped with an $R$-linear pairing
\begin{equation}
\nonumber
(\,,\,)\,:\, \M'\otimes_R \M \rightarrow R.
\end{equation}
Assume that this pairing satisfies the following property:

\begin{itemize}[leftmargin=*]
\item[\mylabel{item_P1}{\bf Adj)}]  For every $m'\in \M' $ and $m\in \M $, we have 
$(Xm',m)=(m',Xm)$.
\end{itemize}

For any $x\in \mathcal X(E),$ we denote by 
\[
(\,,\,)_{\w (x)}\,:\, \M_{\w(x)}'\otimes_E \M_{\w(x)} \rightarrow E
\]
the specialization of $(\,,\,)$ at $\w (x).$

\begin{lemma} 
\label{lemma:factorization of pairing}
There exists a unique pairing 
\[
(\,,\,)_{x} \,:\, \M'_{x}\otimes_E \M[x]\rightarrow E 
\]
such that the restriction of  $(\,,\,)_{\w(x)}\ $ to  the subspace 
$\M_{\w (x)}'\otimes_E \M[x]$ factors as 
\begin{equation}
\label{eqn_prop_221_atempt_1_0}
\begin{aligned}
\xymatrix@R=.4cm{\M'_{\w (x)}\otimes_E \M [x]\ar[dr]_{\pi_{x}\otimes {\rm id}\,\,
}\ar[rr]^(.6){(\,,\,)_{\w (x)}}&&E\\
& \M'_{x}\otimes_E \M[x]\ar[ur]_{(\,,\,)_{x}}}
\end{aligned}
\end{equation}
\end{lemma}
\begin{proof}
If $x\neq x_0,$ the vector space $\M'_{x}$ is a direct summand
of $\M'_{\w (x)}$, and in this case, $(\,,\,)_{x}$ coincides with the restriction
of $(\,,\,)_{\w(x)}$ on $\M'_{x}\otimes_E \M [x].$

If $x=x_0,$ then $\M'_{x_0}=\M'/X\M'$ and 
$\M [x_0]=X^{e-1}\M_{k_0}.$ The property \ref{item_P1} shows that 
$(\,,\,)_{k_0}$ is trivial on $X\M'_{k_0}\otimes  \M [x_0],$ and therefore
it factorizes uniquely through $\M'_{x_0}\otimes_E \M[x_0].$
\end{proof}


\section{Exponential maps and triangulations }
\label{sec_triangulations}


\subsection{Cohomology of $(\varphi,\Gamma)$-modules}
\subsubsection{}
\label{subsubsection:notation phi-Gamma modules}

We set $K_n=\Qp(\zeta_{p^n}),$ $K_\infty=\underset{n=0}{\overset{\infty}\cup}K_n$
and $\Gamma=\Gal (\Qp (\zeta_{p^\infty})/\Qp)$. For any $n\geqslant 1$ we set $\Gamma_n=\Gal (K_\infty/K_n),$ and put $G_n=\Gal (K_n/\Qp)$. Let us fix topological generators $\gamma_n\in\Gamma_n$ 
such that $\gamma_{n}^p=\gamma_{n+1}$ for all $n\geqslant 1$ and $\gamma_{0}^{p-1}=
\gamma_1$.
For any group $G$ and left $G$-module $M$ we denote by $M^{\iota}$
the right $G$-module whose underlying group is $M$ and on which $G$ acts by $m\cdot g=g^{-1}m.$ 

If $G$ is a finite abelian group, we denote by $X(G)$
its   group of characters. If $E$ is a fixed field such that $\rho \in X(E)$ takes values in $E,$ we denote by 
\[
e_\rho=\frac{1}{\vert G\vert} \underset{g\in G}\sum \rho^{-1}(g) g
\] 
the corresponding indempotent of $E[G].$ For any $E[G]$-module $M$ we denote 
by $M^{(\rho)}=e_\rho M$ its $\rho$-isotypic component.  For any map $f\,:\, N\rightarrow M$ we denote by $f^{(\rho)}$ 
the compositum  
$$f^{(\rho)}: N\xrightarrow{f} M\xrightarrow{[e_\rho]} M^{(\rho)}.$$

\subsubsection{}
In this section, we review the construction of the Bloch--Kato exponential map
for crystalline $(\varphi,\Gamma)$-modules. 
Let $E$ be a finite extension of $\Qp.$ For each $n\geqslant 0,$ we denote by $\CR_E$ the Robba ring 
of formal power series $f(\pi)=\underset{m\in \ZZ}{\sum} a_m\pi^m$
coverging on some annulus of the form $r(f)\leqslant \vert \pi \vert_p <1.$
Recall that $\CR_E^+:=\CR_E\cap E[[\pi]]$ is the ring of formal power series
converging on the open unit disk. We equip $\CR_E$ with the usual $E$-linear actions of 
the Frobenius operator $\varphi$ and the cyclotomic Galois group $\Gamma$ given by
\begin{equation}
\nonumber
\begin{aligned}
&\varphi (\pi)=(1+\pi)^p-1,\\
&\gamma (\pi)=(1+\pi)^{\chi (\gamma)}-1,\qquad \gamma \in \Gamma.
\end{aligned}
\end{equation}
Let $\psi$ denote the left inverse of $\varphi$ defined by
\[
\psi (f (\pi))=\frac{1}{p}\varphi^{-1}\left (\underset{\zeta^p=1}\sum f(\zeta (1+\pi)-1)
\right ).
\]
Then $\left (\CR^+_E \right )^{\psi=0}$ is a $\CH (\Gamma)$-module of rank one, generated by 
$1+\pi.$ The differential operator $\partial =(1+\pi) \displaystyle \frac{d}{d\pi}$
is a bijection  from  $\left (\CR^+ \right )^{\psi=0}$ to itself.  Furthermore, we have
\begin{equation}
\nonumber
\partial \circ \gamma =\chi (\gamma) \gamma \circ \partial.
\end{equation}
Let $t=\log (1+\pi)$ denote the ``additive generator of $\Zp (1)$''.
Recall that 
\[
\begin{aligned}
\nonumber
&\varphi (t)=pt\quad,  \quad \gamma (t)=\chi (\gamma) t, \qquad\qquad\qquad \gamma\in \Gamma.
\end{aligned}
\]

Let $A=E\left < X/p^r\right >$ be a Tate algebra over $E$ as above.  We denote by 
$\CR_A=A\,\widehat\otimes_E \CR_E$ the Robba ring with coefficients in $A.$ 
The action of $\Gamma$, $\varphi,$ $\psi$ and $\partial $ can be extended to 
$\CR_A$ by linearity.

\subsubsection{} Recall that a $(\varphi,\Gamma )$-module over $\CR_A$ is a finitely generated projective module over $\CR_A$ equipped with commuting semilinear actions of $\varphi$ and $\Gamma$  and satisfying some additional technical properties which we shall not record here (see \cite{BergerColmez2008} and \cite{KPX2014} for details).  The cohomology $H^i(K_n,\bD)$ of $\bD$ over $K_n$ is defined as  the cohomology of the Fontaine--Herr complex
\begin{equation}
\nonumber
C^{\bullet}_{\varphi,\gamma_n}(\bD) \,:\,
0\lra \bD\xrightarrow{d_0} \bD\oplus \bD
\xrightarrow{d_1} \bD \lra 0,
\end{equation}
where $d_0(x):=((\varphi-1)x,(\gamma_n-1)x)$ and $d_1(y,z):=(\gamma_n-1)y-(\varphi-1)z.$

\subsubsection{} For any $(\varphi,\Gamma)$-module $\bD$ over $\CR_A$
we set $\DCc (\bD)=(\bD [1/t])^{\Gamma}.$ Recall that $\DCc (\bD)$ is a 
finitely generated free  $A$-module equipped with an $A$-linear frobenius $\varphi$ and a decreasing filtration $\left (\Fil^i\DCc (\bD) \right )_{i\in \ZZ}.$
In~\cite{Nakamura2014JIMJ}, Nakamura defined 
$A$-linear maps 
\begin{equation}
\nonumber
\exp_{\bD, K_n}\,:\, \DCc (\bD)\otimes_{\Qp}K_n \lra H^1(K_n,\bD),
\qquad n\geqslant 0
\end{equation}
which extends the definition of Bloch--Kato exponential maps
to $(\varphi,\Gamma)$-modules. 

\subsubsection{} Let $V$ be a $p$-adic representation of $G_{\Qp}$ with coefficients in $A.$
The theory of $(\varphi,\Gamma)$-modules associates to $V$ a $(\varphi,\Gamma)$-module $\DdagrigA (V)$ over $\CR_A .$ The functor $\DdagrigA$ is fully faithul and
we have functorial isomorphisms
\begin{equation}
\nonumber
H^i(K_n,V)\simeq H^i(K_n,\DdagrigA (V)).
\end{equation}
If $V$ is crystalline in the sense of \cite{BergerColmez2008}, we have a functorial isomorphism between the ``classical'' filtered Dieudonn\'e module   $\Dc (V)$ associated to $V$ and $\DCc (\DdagrigA (V)).$
Moreover, the diagram
\begin{equation}
\nonumber
\xymatrix{
\Dc (V)\otimes_{\Qp}K_n \ar[d]^{\simeq} \ar[rrr]^{\exp_{V,K_n}}& & &H^1(K_n,V) \ar[d]^{\simeq}\\
\DCc (\DdagrigA (V))\otimes_{\Qp}K_n \ar[rrr]^{\exp_{\DdagrigA (V),K_n}}& & &H^1(K_n,\DdagrigA (V)),
}
\end{equation}
where the upper row is the Bloch--Kato exponential map,  commutes.

\subsubsection{}  
\label{subsubsec:Iwasawa cohomology}
We refer the reader to  \cite[Section 4.4]{KPX2014}  and  \cite{jaycyclotmotives} for the proofs of the results reviewed in this subsection and for further details. 
Let $\CH_E$ denote the algebra of
 formal power series $f(z)=\underset{j=0}{\overset{\infty}\sum} a_jz^j$ with coefficients in $E$ that converges on the open unit disk. We put
\begin{equation}
\nonumber
\begin{aligned}
&\CH_E(\Gamma_1):=\left \{ f(\gamma_1-1) \mid f\in \CH_E\right \},
&&
\CH_E (\Gamma):=E[\Delta]\otimes_E \CH_E (\Gamma_1), \\
&\CH_A(\Gamma):=A\,\widehat\otimes_E \CH_E(\Gamma),
\end{aligned}
\end{equation}
{ where $\Delta=\Gal (\QQ_p (\zeta_p)/\QQ_p)$.} Note that $\CH_A(\Gamma)$ contains the Iwasawa algebra
$\Lambda_A :=A\otimes_{\ZZ_p}\ZZ_p[[\Gamma]].$

The Iwasawa cohomology $H^1_{\Iw}(\Qp, \bD)$ of a $(\varphi,\Gamma)$-module 
$\bD$ over $\CR_A$ is defined as the cohomology of the complex 
\begin{equation}
\nonumber
\bD \xrightarrow{\psi-1} \bD
\end{equation}
concentrated in degrees $1$ and $2$. In particular,
$H^1_{\Iw}(\Qp, \bD)=\bD^{\psi=1}.$ We have canonical projections 
\begin{equation}
\nonumber
\pr_n\,:\, H^1_{\Iw}(\Qp, \bD) \rightarrow H^1 (K_n,\bD), \qquad n\geqslant 0.
\end{equation}
If $\bD=\DdagrigA (V)$ is the $(\varphi,\Gamma)$-module associated to a $p$-adic
representation $V$ over $A$, we then  have a functorial isomorphism
\begin{equation}
\label{formula: isomorphism between Iwasawa cohomologies of V and  of the (phi-Gamma)-module}
 \CH_A(\Gamma)\otimes_{\Lambda_A}  H^1_{\Iw}(\Qp,V)\simeq H^1_{\Iw}(\Qp, \DdagrigA (V))\,,
\end{equation}
where $H^1_{\Iw}(\Qp,V)$ denotes the classical Iwasawa cohomology with coefficients 
in $V$.
The isomorphism  \eqref{formula: isomorphism between Iwasawa cohomologies of V and  of the (phi-Gamma)-module} composed with the projections $\pr_n$ coincide with the natural morphisms
\[
\CH_A(\Gamma)\otimes H^1_{\Iw}(\Qp,V) \lra H^1(K_n,V) 
\]
induced by the Iwasawa theoretic  descent maps $H^1_{\Iw}(\Qp,V) \rightarrow 
H^1(K_n,V).$ 

For each integer $m$, the cyclotomic twist 
\begin{equation}
\nonumber
\Tw_m \,:\, H^1_{\Iw}(\Qp,V) \rightarrow  H^1_{\Iw}(\Qp,V(m)),
\qquad 
\Tw_m(x)=x\otimes \varepsilon^{\otimes m}
\end{equation} 
can be extended to a map
\begin{equation}
\nonumber
\begin{aligned}\Tw_m \,:\,\CH_A(\Gamma)\otimes_{\Lambda_A} H^1_{\Iw}(\Qp,V) &\lra  
\CH_A(\Gamma)\otimes_{\Lambda_A} H^1_{\Iw}(\Qp,V(m))
\\
f(\gamma-1)\otimes x&\longmapsto f(\chi^m (\gamma)\gamma-1) \otimes \Tw_m(x)\,.
\end{aligned}
\end{equation}


\subsection{The Perrin-Riou exponential map}
\label{subsec_PR_map_general}
 
\subsubsection{}
\label{subsubsec_subsec_PR_map_general_1} In this subsection, we shall review fragments of Perrin-Riou's theory of large exponential maps. Define the operators
\begin{equation}
\nonumber
\begin{aligned}
&\nabla =\frac{\log (\gamma_1)}{\log \chi (\gamma_1)}, \qquad \gamma_1\in \Gamma_1,
\\
&\ell_j=j-\nabla, \qquad j\in \ZZ.
\end{aligned}
\end{equation}
Note that $\nabla$ does not depend on the choice of $\gamma_1\in \Gamma_1.$
It is easy to check by induction that
\begin{equation}
\nonumber
\underset{j=0}{\overset{h-1}\prod} \ell_j= (-1)^ht^h\partial^h \qquad
\text{\rm (on $\CR_E$)}.
\end{equation}

Let $V$ be a $p$-adic crystalline representation of $G_{\Qp}$ with coefficients in $E$ satisfying  the following condition:
\begin{itemize}
 \item[\mylabel{item_LE1}{\bf LE)}]
$\Dc (V)^{\varphi=p^i}=0$ for all $i\in \ZZ.$
\end{itemize}
Note that this assumption can be relaxed (see, for example, \cite{perrinriou94, BenoisJIMJ2014}), but it  simplifies the presentation. In particular, it implies that $H^0(K_\infty,V)=0$ and therefore also that $H^1_{\Iw}(\Qp, V)$ is torsion free over the Iwasawa algebra (c.f. \cite{perrinriou94}, Lemme~3.4.3 and Proposition~3.2.1).

Let us set
\begin{equation}
\nonumber 
\mathfrak D (V)=\left (\CR_E^+\right )^{\psi=0}
\otimes_E \Dc (V).
\end{equation}

For any $\alpha \in \mathfrak D (V),$
the equation 
\begin{equation}
\nonumber
(1-\varphi)(F)=\alpha
\end{equation}
has a unique solution $F\in \CR_E^+\otimes_E \Dc (V).$
We define the maps 
\begin{equation}
\nonumber
\Xi_{V,n}\,:\,\mathfrak D (V) \rightarrow \Dc (V)\otimes_{\Qp}K_n, \qquad n\geqslant 0
\end{equation}
on setting
\begin{equation}
\nonumber
\Xi_{V,n}(\alpha)= \begin{cases}
 p^{-n}(\mathrm{id} \otimes \varphi )^{-n}(F)(\zeta_{p^n}-1)
& \text{if $n\geqslant 1$},\\
\displaystyle\left (\frac{1-p^{-1}\varphi^{-1}}{1-\varphi}\right ) \alpha (0)
& \text{if $n=0.$}
\end{cases}
\end{equation}

The following is the main result of \cite{perrinriou94}. 

\begin{theorem}[Perrin-Riou] 
\label{thm:Perrin-Riou}
Let $V$ be a crystalline representation which satisfies the condition \ref{item_LE1} above. Then for integers $h\geqslant 1$ such that 
$\Fil^{-h}\Dc (V)=\Dc (V)$, there exists
a $\CH_E (\Gamma)$-homomorphism 
\begin{equation}
\nonumber
\Exp_{V,h} \,:\, \mathfrak D (V) \lra \CH_E (\Gamma)\otimes_{\Lambda_E} H^1_{\Iw}(\Qp,V)
\end{equation}
satisfying the following properties:

\item[i)] For all $n\geqslant 0$ the following diagram commutes:
\begin{equation}
\nonumber
\xymatrix{
\mathfrak D (V) \ar[rr]^-{\Exp_{V,h}} 
\ar[d]^{\Xi_{V,n}}& & \CH_E (\Gamma)\otimes_{\Lambda_E} H^1_{\Iw}(\Qp,V)
\ar[d]^{\pr_n}\\
\Dc (V)\otimes K_n \ar[rr]^{(h-1)!\exp_{V,K_n}} && H^1(K_n,V).
}
\end{equation}

\item[ii)] Let us denote by $e_{-1}:=\varepsilon^{-1}\otimes t$ the canonical generator of $\Dc (\Qp (-1))$. Then 
\begin{equation}
\nonumber
\Exp_{V(1),h+1}=-\Tw_1 \circ \Exp_{V,h} \circ (\partial \otimes e_{-1}).
\end{equation}

\item[iii)] We have
\begin{equation}
\nonumber 
\Exp_{V,h+1}=\ell_{h}\Exp_{V,h}.
\end{equation}
\end{theorem}
Let us put 
\begin{equation}
\label{definition of finite level  projection of Perrin-Riou Exp }
\Exp_{V,h,n}=\pr_n\circ \Exp_{V,h}\,\,:\,\,
\mathfrak D (V)  \lra H^1(K_n, V).
\end{equation}

\subsubsection{}
\label{subsubsection: isotypic components}
 
 Set $G_n=\Gal (K_n/K).$  Without loss of generality, we can assume
that the characters of $G_n$ take values in $E.$
 Recall that Shapiro's lemma gives an isomorphism of $E[G_n]$-modules
\[
H^1(K_n, V)\stackrel{\sim}{\lra} H^1(\Qp, V\otimes_E E[G_n]^{\iota}).
\]
On taking the $\rho$-isotypic components, we obtain isomorphisms
\begin{equation}
\nonumber
H^1(K_n, V)^{(\rho)}  \stackrel{\sim}{\lra}       H^1(\Qp, V(\rho^{-1})),\qquad\qquad \rho \in X(G_n).
\end{equation} 
We shall make use of the following elementary lemma giving the $\rho$-isotypic component 
of the map $\Xi_{V,n}.$ 

\begin{lemma} 
\label{lemma from BenoisBerger2008}
Assume that $\rho \in X(G_n)$ is a primitive character. Then
for any $\alpha =f(\pi)\otimes d \in \mathfrak D (V)$
\begin{equation}
\nonumber
\Xi_{V,n}^{(\rho)}(\alpha)=
\begin{cases} e_{\rho}(f(\zeta_{p^n}-1))\otimes p^{-n}\varphi^{-n}(d)
& \textrm{if $n\geqslant 1$},\\
(1-p^{-1}\varphi^{-1})(1-\varphi)^{-1} (d)
 &\textrm{if $n=0.$}
 \end{cases}
 \end{equation}
\end{lemma}
\begin{proof} See \cite[Lemma~4.10]{BenoisBerger2008}.
\end{proof}

\subsubsection{}
\label{subsubsec_Berger_PR}
We recall the explicit construction of the Perrin-Riou exponential that was 
discovered by Berger in \cite{BergerExpKatoVol}.
Let, as before, $F\in \CR_E^+\otimes \Dc (V)$ denote the solution of the equation
$
(1-\varphi)\,F=\alpha .
$
Define 
\[
\Omega_{V,h} (\alpha) \,:=\,-\frac{\log \chi (\gamma_1)}{p}\,\ell_{h-1}\ell_{h-2} \cdots \ell_0 (F (\pi)).
\]
It is not difficult  to see that $\Omega_{V,h} (\alpha) \in
\DdagrigE (V)^{\psi=1}$ and an explicit computation shows that
$\Omega_{V,h}$ satisfies  properties i-iii) in Theorem~\ref{thm:Perrin-Riou}.  

See also \cite{Nakamura2014JIMJ} for a generalization of this approach to 
de Rham representations.

\subsubsection{}
In the remainder of \S\ref{subsec_PR_map_general}, we shall review the construction of the Perrin-Riou exponential maps for families of  $(\varphi,\Gamma)$-modules of rank one. 

For any continuous character 
$\delta \,:\,\Qp^\times \rightarrow A^\times$ we denote by $\bD_\delta=\CR_Ae_\delta$
the $(\varphi,\Gamma)$-module of rank one over the relative Robba ring $\CR_A$, generated by $e_\delta$
and such that
\begin{equation}
\nonumber
\varphi (e_\delta)=\delta (p)e_\delta \quad, \quad 
\gamma (e_\delta)= \delta (\chi (\gamma)) e_\delta \qquad \qquad \gamma \in \Gamma.
\end{equation}
Set $\bD_\delta^+:=\CR_A^+e_\delta.$
We say that $\bD_\delta$ is of Hodge--Tate weight $-m\in \ZZ$ (sic!) if 
\begin{equation}
\nonumber 
\delta (u)=u^{m}, \qquad\qquad \forall u\in \Zp^\times.
\end{equation}
If $\bD_\delta$ is of Hodge--Tate weight $-m,$ then  $\DCc (\bD_\delta)$ is the free $A$-module of rank one generated 
by $d_\delta=t^{-m}e_\delta.$ It has the unique filtration break at $-m,$ and $\varphi$ acts on $\DCc (\bD_\delta)$
as the multiplication by $p^{-m}\delta (p)$ map.

\subsubsection{}\label{subsubsec_PR_exp_map}  Let $\bD_\delta$ be a $(\varphi,\Gamma)$-module of rank one and Hodge--Tate weight $-m\in \ZZ$. The direct  analogue of the condition \ref{item_LE1} for families of $(\varphi,\Gamma)$-modules  is the following condition: 
\begin{itemize}
 \item[\mylabel{item_LEstar}{\bf LE*)}]
For any $i\in \ZZ$,\, $1-p^i\delta (p) \in A$ does not vanish  on
 $\mathcal X=\Spm (A)$\,.
\end{itemize}
For any point $x\in \Spm (A),$ let $\bD_{\delta,x}$ denote the  specialization 
of $\bD_\delta$ at $x.$ To facilitate a comparison with \ref{item_LE1}, we remark that the condition \ref{item_LEstar} implies that 
\[
\DCc (\bD_{\delta,x})^{\varphi=p^i}=0, \qquad \forall  x\in \Spm (A),\quad  \forall i \in \ZZ\,.
\]

Let us put $\mathfrak D (\bD_\delta)=\CR_A^{\psi=0}\otimes_A \DCc (\bD_\delta).$
We shall explain the construction of a family of maps (which we will call Perrin-Riou exponential maps)
\begin{equation}
\nonumber
\Exp_{\bD_\delta, h}\,:\,\mathfrak D (\bD_\delta) \lra H^1_{\Iw}(\Qp, \bD_\delta), \qquad h\geqslant m
\end{equation}
modelled on the discussion in \S\ref{subsubsec_Berger_PR}. Let us set $\alpha (\pi):=f(\pi)\otimes d_\delta \in \mathfrak D (\bD_\delta)
.$ The equation
\begin{equation}
\nonumber
(1-\delta (p)\varphi) F_{m}(\pi)=\partial^{m}f(\pi)
\end{equation}
has a unique solution in $\CR_A^+.$  It is easy to see that 
\[
F_{-m}(\pi)\otimes e_\delta \in \bD_\delta^{\psi=1}=H^1_{\Iw}(\Qp, \bD_\delta).
\] 
We define
\begin{equation}
\nonumber
\Exp_{\bD_\delta, m}(\alpha):=(-1)^{m-1} \frac{\log \chi (\gamma_1)}{p}F_{m}(\pi)
\otimes e_\delta,
\end{equation}
and set
\begin{equation}
\nonumber
\Exp_{\bD_\delta, h}(\alpha)= \left (\underset{j=m}{\overset{h-1} \prod}\ell_j \right )\circ
\Exp_{\bD_\delta, m}(\alpha), \qquad h\geqslant m+1.
\end{equation}

The following result can be extracted from   \cite[Section~4]{Nakamura2017ANT} or proved directly using Berger's arguments.

\begin{proposition}
\label{prop: large exponential map}
\item[i)] For any  $h\geqslant m$ we have 
\[
\Exp_{\bD_\delta, h+1}=\ell_h \Exp_{\bD_\delta, h}.
\]

\item[ii)] For any $h\geqslant m$ the following diagram commutes:
\begin{equation}
\nonumber
\xymatrix{
\mathfrak D (\bD_{\delta \chi})
\ar[d]_{e_1\otimes \partial}
\ar[rr]^-{\Exp_{\bD_{\delta \chi,h+1}}} &&H^1_{\Iw}(\Qp, \bD_{\delta \chi}) \\
\mathfrak D (\bD_\delta)
\ar[rr]^-{\Exp_{\bD_{\delta},h}} &&H^1_{\Iw}(\Qp, \bD_{\delta}) \ar[u]_{-\Tw_1}
}
\end{equation}
 
\item[iii)] Assume that  $ h\geqslant m\geqslant  1.$ For any   $\alpha (\pi)=f(\pi)\otimes d_\delta \in 
\CR_A^{\psi=0}\otimes_A \DCc (\bD_\delta)$, the equation 
\[
(1-\varphi)(F)=\alpha (\pi)
\]
has a unique solution $F\in \DCc (\bD_\delta)\otimes \CR_A^+$ and it verifies
\begin{equation}
\nonumber
\Exp_{\bD_\delta,h}(\alpha)=-\frac{\log \chi (\gamma_1)}{p} \underset{j=0}{\overset{h-1} \prod}\ell_j  (F).
\end{equation}
Moreover, 
\begin{equation}
\nonumber
\underset{j=0}{\overset{h-1} \prod}\ell_j  (F)=
(-1)^{h-1}t^h\partial^h(F)\in \left (t^{h-m} \bD_\delta^+ \right )^{\psi=1}.
\end{equation}

\item[iv)] Under the conditions and notation of iii), let us put 
\[
\Xi_{\bD_\delta,n}(\alpha):=
\begin{cases}\left (\frac{1-p^{-1}\varphi^{-1}}{1-\varphi} \right )\alpha (0)
&\text{ if $n=0$},\\
p^{(m-1)n}\delta (p)^{-n} F(\zeta_{p^n}-1) &\text{ if $n\geqslant 1.$}
\end{cases}
\]
Then the map $\Xi_{\bD_\delta,n}\,:\,\mathfrak D (\bD_\delta)
\rightarrow  \DCc (\bD_{\delta}) \otimes_{\Qp} K_n$ is surjective, and  the diagram
\begin{equation}
\nonumber
\xymatrix{
\mathfrak D (\bD_\delta) \ar[d]^{\Xi_{\bD_\delta,n}}
\ar[rrr]^-{\Exp_{\bD_{\delta},h}} & &&H^1_{\Iw}(\Qp, \bD_{\delta}) \ar[d]^{\pr_n} \\
\DCc (\bD_{\delta}) \otimes_{\Qp} K_n
\ar[rrr]^-{(h-1)!\exp_{\bD_\delta, K_n}}& &&H^1 (K_n,\bD_{\delta}) 
}
\end{equation}
commutes.

\end{proposition}


\subsection{Triangulations in families}
\label{subsec_triangulations}

\subsubsection{}
\label{subsubsec_triangulations_1}

In \S\ref{subsec_triangulations}, we shall work in the setting of \S\ref{sec_Bellaiche_axiomatized} and introduce various objects with which we shall apply the general constructions in \S\ref{sec_Bellaiche_axiomatized}. To that end, let us fix an integer $e\geqslant 1$ and put $A=R[X]/(X^e-Y)$ as before, where $R=E\left <Y/p^{re} \right >$ is a Tate algebra. 
Set ${\mathcal W}=\Spm (R),$  $\mathcal X=\Spm (A)$ and denote by 
$\w\,:\, \mathcal X \rightarrow {\mathcal W}$ the weight map. 
 We fix an integer $k_0\geqslant 2$ and 
identify  $\mathcal W$ with the closed disk $D(k_0,p^{re-1})$ as in 
Section~\ref{subsection:generalities about W and X}. 
Let $x_0\in \mathcal X$ denote the unique point such that $\w (x_0)=k_0.$
We put
\[
\mathcal X^{\mathrm{cl}}=\{x\in \mathcal X  \mid \w (x)\in \ZZ, {\w (x) \geqslant 0}\}. 
\]

We let $V$ denote a free $A$-module of rank $2$ which is endowed with a continuous action of the Galois 
group $G_{\QQ,S}$. 
In accordance with the notation of Section~\ref{sec_Bellaiche_axiomatized},
for any $y\in {\mathcal W}(E)$ and $x\in \mathcal X$, we set $V_{y}=V\otimes_{R,\w} E$\,, $V_{x}=V/\mathfrak m_x V$ and denote by 
$\pi_x\,:\,V_{\w (x)}\rightarrow V_x$ the canonical projection. 
Note that  $V_{k_0}=V/X^{e}V$\,, $V_{x_0}=V/XV$ and $V[x_0]=X^{e-1}V/ X^eV.$

\subsubsection{}
\label{subsubsecKPX}

We shall assume that $V$ verifies following conditions:

\begin{itemize}
 \item[\mylabel{item_C1}{\bf C1)}]
 For each $x\in \mathcal X^{\mathrm{cl}} $ of integer weight $\w (x)\geqslant 0$
the restriction of $V_x$ on the decomposition group at $p$ is semistable
of Hodge--Tate weights  $(0, \w (x)  +1).$

 \item[\mylabel{item_C2}{\bf C2)}]
 There exists $\alpha \in A$ such that for all $x\in \mathcal X^{\mathrm{cl}}$
 the eigenspace  $\Dst (V_x)^{\varphi=\alpha (x)}$ is one dimensional. 

 \item[\mylabel{item_C3}{\bf C3)}]
 For each $x\in \mathcal X^{\mathrm{cl}}-\{x_0\}$ 
\begin{equation}
\nonumber
\Dst (V_x)^{\varphi=\alpha (x)}\cap \Fil^{  \w (x) +1}\Dst (V_x)=0.
\end{equation}
\end{itemize}
On shrinking $\mathcal{X}$ as necessary, it follows from  \cite{KPX2014} that one can construct a unique $(\varphi,\Gamma)$-submodule $\bD\subseteq \DdagrigA (V)$ of rank one with the following properties:

\begin{itemize}
 \item[\mylabel{item_phiGamma1}{$\varphi\Gamma_1$)}]
 $\bD=\bD_\delta$ with $\delta\,:\,\Qp^\times\rightarrow A^\times$
such that $\delta (p)=\alpha (x)$ and $\left.\delta \right \vert_{\Zp^\times}=1.$ 

 \item[\mylabel{item_phiGamma2}{$\varphi\Gamma_2$)}]
 $\DCc (\bD_x)=\Dst (V_x)^{\varphi =\alpha (x)}$ for each $x\in \mathcal X^{\textrm{cl}}$, where $\bD_{x}:=\bD\otimes_{A,x}E$.

 \item[\mylabel{item_phiGamma3}{$\varphi\Gamma_3$)}]
 For each $x\in \mathcal X^{\textrm{cl}}-\{x_0\},$  
the $(\varphi,\Gamma )$-module $\bD_x$ is saturated in $\DdagrigE (V_x).$
\end{itemize}

Let $\bD_{x_0}^{\mathrm{sat}}$ denote the saturation of the specialization 
$\bD_{x_0}$ of the $(\varphi,\Gamma)$-module $\bD$ at $x_0.$ Then, as explained in the final section of \cite{KPX2014}, we have $\bD_{x_0}=t^m\bD_{x_0}^{\mathrm{sat}}$ for some $m\geqslant 0$. We will consider the following two scenarios\footnote{In the context of elliptic modular forms, the condition \ref{item_Sat_A} (resp., \ref{item_Sat_B}) translates into the requirement that the corresponding eigenform is non-$\theta$ critical (resp., is $\theta$-critical) in the sense of Coleman.} :
\begin{itemize}[leftmargin=*]
 \item[\mylabel{item_Sat_A}{${}^{\neg}\Theta$)}] $\Dst (V_{x_0})^{\varphi=\alpha (x_0)}\cap 
\Fil^{ k_0+1}\Dst (V_{x_0})=0.$ 
\end{itemize}
Since $\DCc (\bD_{x_0}^{\mathrm{sat}})=\Dst (V_{x_0})^{\varphi=\alpha (x_0)},$ the condition \ref{item_Sat_A} implies that the Hodge--Tate weight of 
$\Dst (V_{x_0})^{\varphi=\alpha (x_0)}$ is $0.$ Since the 
Hodge--Tate weight of $\bD_{x_0}$ is also $0$, we deduce that in this case
$m=0$ and $\bD_{x_0}$ is saturated in $\DdagrigE (V_x).$
\begin{itemize}[leftmargin=*]
 \item[\mylabel{item_Sat_B}{$\Theta$)}] $\Dst (V_{x_0})^{\varphi=\alpha (x_0)}= 
\Fil^{ k_0+1}\Dst (V_{x_0}).$ 
\end{itemize}
Suppose that \ref{item_Sat_B} holds.  Then $ \DCc (\bD_{x_0}^{\mathrm{sat}})   =
\Fil^{ k_0+1}\Dst (V_{x_0}).$ Therefore $\bD_{x_0}^{\mathrm{sat}}$ is of Hodge--Tate weight ${k_0+1}$ and $m={k_0+1}.$ Let $\beta (x_0)$ denote the other eigenvalue 
of $\varphi$ on $\Dst (V_{x_0}).$ By the weak admissibility of $\Dst (V_{x_0})$, we infer that $v_p(\alpha (x_0))\geqslant { k_0+1}$\,, $v_p(\beta (x_0))\geqslant 0$ and that
$v_p(\alpha (x_0))+v_p(\beta (x_0))= k_0+1$. Thence, 
$v_p(\alpha (x_0))=  k_0+1$ and $v_p(\beta (x_0))=0,$ and the eigenspaces 
$\Dst (V_{x_0})^{\varphi=\alpha (x_0)}$ and $\Dst (V_{x_0})^{\varphi=\beta (x_0)}$ 
are weakly admissible. They are therefore admissible and the restriction of
$V_{x_0}$ to the decomposition group at $p$ decomposes into a direct sum 
\begin{equation}
\label{eqn_p_local_decomposition_theta}
V_{x_0}=V_{x_0}^{(\alpha)}\oplus V_{x_0}^{(\beta)}
\end{equation}
of two one-dimensional $G_{\QQ_p}$-representations. 
Let us $\bD^{(\alpha)}=\DdagrigE (V_{x_0}^{(\alpha)})$
and $\bD^{(\beta)}=\DdagrigE (V_{x_0}^{(\beta)}).$
Then
\begin{equation}
\nonumber
\bD^{(\alpha)}=\CR_E e^{(\alpha)}\quad, \quad \bD^{(\beta)}=\CR_E e^{(\beta)},
\end{equation}
where
\begin{equation}
\nonumber
\begin{aligned}
&\gamma (e^{(\alpha)})=\chi^{{ -1-k_0}}(\gamma)e^{(\alpha)},\qquad
&&\varphi (e^{(\alpha)})= \alpha (x_0)p^{ -1-k_0}e^{(\alpha)}, &&&\\
&\gamma (e^{(\beta)})=e^{(\beta)},\qquad
&&\varphi (e^{(\beta)})= \beta (x_0)e^{(\beta)}, &&& \qquad\qquad\gamma\in \Gamma.
\end{aligned}
\end{equation}
We have 
\begin{equation}
\label{formula: formulas for bD_x_0 and bD_x_0^sat}
\begin{aligned}
& \Dc (V_{x_0}^{(\alpha)})=\Dst (V_{x_0})^{\varphi=\alpha (x_0)}, \\
&\Dc (V_{x_0}^{(\beta)})=\Dst (V_{x_0})^{\varphi=\beta (x_0)}, \\
& \bD^{(\alpha)}=\bD_{x_0}^{\mathrm{sat}}=t^{ -(k_0+1)}\bD_{x_0}.
\end{aligned}
\end{equation}
We remark that $\DCc (\bD_{x_0})$ and $\Dc (V_{x_0}^{(\alpha)})$ are isomorphic as
$\varphi$-modules but not as filtered modules: they have  Hodge--Tate weights
$0$ and $ k_0+1$, respectively.

\section{Two variable $p$-adic $L$-functions: the abstract definitions}
\label{sec_2varpadicLabstract}
Suppose that we are given another free $A$-module $V'$ of rank two which is equipped with a continuous $G_{\QQ,S}$-action, together with a Galois equivariant  $R$-linear pairing 
\begin{equation}
\label{eqn_the_pairing_abstract}
(\,,\,):  V' \otimes V\lra R
\end{equation}
satisfying the condition \ref{item_P1} of Section~\ref{sec_Bellaiche_axiomatized}\,: 
\begin{center}
For every  $v'\in V'$  and $v\in V$, we have 
 $(Xv',v)=(v',Xv)$. 
\end{center}

\subsection{Duality (bis)} 
\label{subsubsec_twovarprelim_objects}

We have a canonical $\CH_R(\Gamma)$-bilinear pairing
\begin{equation}
\label{eqn_PR_pairing_Iw}
\left <\,\,,\,\,\right > \,:\,\left(\CH_R(\Gamma)\otimes_R H^1_{\Iw}(\Qp,{ V'} (1))\right)
\otimes  \left(\CH_R(\Gamma)\otimes_R H^1_{\Iw}(\Qp,V)^\iota\right) \longrightarrow \CH_R(\Gamma),
\end{equation}
which has the following explicit description. Let 
\[
\left (\,\,,\,\,\right )_n\,:\, H^1(K_n,{ V'}(1))\times H^1(K_n,V) \longrightarrow R, \qquad n\geqslant 0
\]
denote the  cup-product pairings induced by  (\ref{eqn_the_pairing_abstract}). 
Recall from  \S\ref{subsubsec:Iwasawa cohomology} the projection maps
\begin{align*}\pr_n\,:\, \left(\CH_R(\Gamma)\otimes_R H^1_{\Iw}(\Qp,V)\right) \lra
H^1(K_n,V),&\\
{\pr'_n}\,:\, \left(\CH_R(\Gamma)\otimes_R 
H^1_{\Iw}(\Qp,{ V'} (1))\right) \lra
H^1(K_n, { V'}(1)),&
\qquad n\geqslant 0\,.
\end{align*}

We then have 
\begin{equation}
\label{formula: Perrin-Riou pairing}
\left <x,y\right >\equiv \underset{\tau\in \Gamma/\Gamma_n}\sum 
\left (\tau^{-1}{ \pr'_n}(x)\,,\,\pr_n (y)\right )_n \tau \quad \mod{(\gamma_n-1)}
\end{equation}
(see \cite{perrinriou94}, Section~3.6).  In particular, for any finite character $\rho \in X(\Gamma)$ of conductor $p^n$, we have
\begin{equation}
\nonumber
\rho \left (\left <x,y^\iota\right >\right )= \underset{\tau\in \Gamma/\Gamma_n}\sum 
\left (\tau^{-1}{\pr'_n}(x),\pr_n (y)\right )_n \rho (\tau)=
\left (e_\rho(x), e_\rho(y^\iota) \right )_{\rho,0}\,,
\end{equation}
where $\left ( \, , \,\right )_{\rho,0}$ stands for the cup-product pairing
$$H^1(\Qp, { V'}(\chi\rho^{-1}))\otimes_R H^1(\Qp, V(\rho)) \xrightarrow{\left ( \, , \,\right )_{\rho,0}} R.$$
Note that the pairing \eqref{eqn_PR_pairing_Iw} can be recast entirely in terms of the Iwasawa cohomology of associated  $(\varphi,\Gamma)$-modules; c.f. \cite[Section~4.2]{KPX2014}.

\subsubsection{}
We apply the formalism of Section~\ref{subsec:duality} to our situation. 
 Equip the tensor product
\[
A\otimes_R \left (\CH_R(\Gamma)\otimes_R 
H^1_{\Iw}(\Qp, V )\right )\simeq\CH_A(\Gamma)\otimes_R 
H^1_{\Iw}(\Qp, V )   
\]
 with the action of $A$ through the first factor.
We extend the pairing \eqref{eqn_PR_pairing_Iw} by linearity to the pairing
\begin{equation}
\label{eqn_PR_pairing_Iw_A}
\left <\,\,,\,\,\right >_A \,:\,\left(\CH_R(\Gamma)\otimes_R 
H^1_{\Iw}(\Qp,{ V'} (1))\right)
\otimes_{\CH_R(\Gamma)}  \left(\CH_A(\Gamma)\otimes_R H^1_{\Iw}(\Qp,V)^\iota\right) \longrightarrow \CH_A(\Gamma),
\end{equation}
which is $\CH_R(\Gamma)$-linear (respectively, $\CH_A(\Gamma)$-linear) with respect to the first (respectively second) factor.

For any $x\in \mathcal{X}(E)$, the functoriality of the cup-products gives rise to the following commutative diagram:
\begin{equation}
\label{eqn_PR_pairing_Iw_A_diagram_specialize}
\begin{aligned}
\xymatrix{
\left(\CH_R(\Gamma)\otimes_R H^1_{\Iw}(\Qp,{V'}(1))\right)
\ar[d]^{\w (x)} 
&\otimes_{\CH_R(\Gamma)} & \left(\CH_A(\Gamma) \otimes_R H^1_{\Iw}(\Qp,V)^\iota\right)\ar[d]^{{\rm sp}_x}\ar[r]^(.7){\left <\,\,,\,\,\right >_A}&\CH_A(\Gamma)\ar[d]^{x}\\
\CH_E(\Gamma)\otimes_E H^1_{\Iw}(\Qp,{V'_{\w(x)}}(1))  
&\otimes_{\CH_E(\Gamma)}&  \left(\CH_E(\Gamma) \otimes_E H^1_{\Iw}(\Qp,V_{{ \w(x)}})^\iota\right)\ar[r]^(.7){\left <\,\,,\,\,\right >_{{ \w(x)}}}&  \CH_E(\Gamma).
}
\end{aligned}
\end{equation}
where the left and the right vertical maps are the specializations at $\w (x)$ and 
$x$ respectively, and the middle   vertical map is the specialization  defined in 
Section~\ref{subsec:duality}.

By Lemma~\ref{lemma:factorization of pairing}, 
for any $x\in \mathcal X(E)$ the pairing $(\,,\,)$ induces a pairing 
\[
(\,,\,)_x\,:\, V_{x}\otimes_E V[x]  \lra  E.
\]
The pairing in the bottom row of \eqref{eqn_PR_pairing_Iw_A_diagram_specialize} therefore factors as 
\begin{equation}
\label{formula:factorization of Iwasawa pairing}
\begin{aligned}
\xymatrix{
\CH_E(\Gamma)\otimes_E H^1_{\Iw}(\Qp,{ V'_{ \w(x) }(1))}\ar[d]^x  &\otimes&  \left(\CH_E(\Gamma) \otimes_E H^1_{\Iw}(\Qp,V[x])^\iota\right)\ar[r]^(.72){\left <\,\,,\,\,\right >_{{ \w (x)}}}&  \CH_E(\Gamma)\\
\CH_E(\Gamma)\otimes_E H^1_{\Iw}(\Qp,{ V'_x}(1))  &\otimes&  \left(\CH_E(\Gamma) \otimes_E H^1_{\Iw}(\Qp,V[x])^\iota\right)\ar[r]^(.72){\left <\,\,,\,\,\right >_{x}}\ar@{=}[u]&  \CH_E(\Gamma)\,. \ar@{=}[u]}
\end{aligned}
\end{equation}

\subsection{Perrin-Riou-style two variable $p$-adic $L$-functions}
 Recall from \S\ref{subsubsec_PR_exp_map} the Perrin-Riou exponential  
\begin{equation}
\nonumber
\Exp_{\bD,h}\,:\,\mathfrak D(V)
\lra  H^1_{\Iw}(\Qp, \bD)\,.
\end{equation}
Note that we have a canonical injection
\begin{equation}
\nonumber
 H^1_{\Iw}(\Qp, \bD) \lra \CH_R(\Gamma)\otimes_R H^1_{\Iw}(\Qp,V)= \CH_R(\Gamma)\otimes_R H^1_{\Iw}(\Qp,V).
\end{equation}
Let us fix an $R$-module generator  $\eta \in \DCc (\bD)$.

Let us set $\mathfrak D(V)_A=\mathfrak D(V)\otimes_R A$. We have a canonical isomorphism
\[
\mathfrak D(V)_A\simeq
\DCc (\bD)\otimes_R  \left (\CR_A^+\right )^{\psi=0}
 \]
 The Perrin-Riou exponential map $\Exp_{\bD,h}$ can be extended by linearity to an $A$-linear map
\begin{equation}
\nonumber
\Exp_{\bD,h}^{A}\,:\,\mathfrak D(V)_A 
\xrightarrow{\Exp_{\bD,h}\otimes {\rm id}_A}  H^1_{\Iw}(\Qp, \bD)\otimes_R A\,.
\end{equation}
Note that we have a canonical injection
\begin{equation}
\nonumber
 H^1_{\Iw}(\Qp, \bD)\otimes_R A \lra A\otimes_R\left(
\CH_R(\Gamma)\otimes_R H^1_{\Iw}(\Qp,V)\right)= \CH_A(\Gamma)\otimes_R
H^1_{\Iw}(\Qp,V).
\end{equation}

Fix an $A$-module generator  $\eta \in \DCc (\bD)$  and set
\begin{equation}
\nonumber
\begin{aligned}
\label{eqn_bbeta_def}&\bbeta:=\sum_{i=0}^{e-1}X^{i}\eta\otimes X^{e-1-i}
\in \DCc (\bD) \otimes_RA,\\
&\widetilde \bbeta:=  \bbeta \otimes (1+\pi)
\in \mathfrak D(V)_A.
\end{aligned}
\end{equation}
We remark that  for each $x\in\mathcal X(E)$, the specialization $\bbeta_x:=\spe_x (\bbeta)$ is a generator of $\DCc (\bD)[x]$ by Lemma~\ref{lemma_bellaiche_4_14_general}. Let us put $\widetilde \bbeta_x:=\bbeta_x\otimes (1+\pi)$.

\begin{defn}
\label{defn_fat_eta_etatilde}

\item[i)] \label{defn_a_linear_PR_Log}  For each $ h\geqslant 0$ and cohomology class $\bz\in H^1_{\Iw}(\Qp, {  V'(1)})$,  we define
\begin{equation}
\nonumber
\begin{aligned}
&\Log^A_{\bD,\bbeta,1-h}\,:\,H^1_\Iw (\Qp, { V'(1)})\lra \CH_A(\Gamma)\,,\\
&\Log^A_{\bD,\bbeta,1-h} (\bz):=\left <\bz, c \circ \Exp^A_{\bD,h}(\widetilde\bbeta)^\iota \right >_{A}.
\end{aligned}
\end{equation}

\item[ii)] We define the two-variable $p$-adic $L$-function associated to $\bz$ on setting
\[
L_{p,\bbeta}(\bz):=\Log^A_{\bD,\bbeta,1}(\bz) \in \CH_A(\Gamma).
\]
For each $x\in \mathcal X^{\mathrm{cl}}(E)$, we similarly define the one-variable $p$-adic $L$-function 
\[
L_{p,\bbeta_x}(\bz_x):=\left <\bz_x, c \circ \Exp_{\bD[x],0}(\widetilde\bbeta_x)^\iota \right >_x \in \CH_E(\Gamma).
\]

\item[iii)] For any $x\in \mathcal X(E)$, finite character 
$\rho \in X(\Gamma)$, we denote by
\[
L_{p,\bbeta}(\bz,x,{ \rho}):={\rho }\circ x\circ L_{p,\bbeta}(\bz,X).
\]
the value of $L_{p,\bbeta}(\bz)$ at $(x, { \rho }).$
\end{defn}

\begin{proposition}
\label{prop_bellaiche_formal_step_1}
Suppose $e=1$. For any $x\in \mathcal X^{\mathrm{cl}}(E)$ we have
\begin{equation}
\nonumber
L_{p,\bbeta}(\bz)(x)=L_{p,\bbeta_x}(\bz_{x}).
\end{equation}
\end{proposition}
\begin{proof} 
This is an immediate consequence of the diagram \eqref{eqn_PR_pairing_Iw_A_diagram_specialize}, combined with the discussion in Remark~\ref{rem_sp_x_when_e_equals1}.
\end{proof}


\section{Local description of the eigencurve}
\label{sec_bellaiche_eigencurve}
In Section~\ref{sec_bellaiche_eigencurve}, we review Bella\"iche's results in \cite{bellaiche2012} on the local description of the eigencurve (most particularly, about a $\theta$-critical eigenform). We also follow his exposition very closely here, particularly of Sections 2.1.1, 2.1.3, 3.1, 3.2 and 3.4 in op. cit.; and rely also on the notation set therein for the most part (e.g., unless we declare otherwise).

\subsection{Modular symbols}
\label{subsec_bellaiche_eigencurve}

\subsubsection{} 
\label{subsubsec_bellaiche_eigencurve_1} Let  $\mathcal{WS}=\textrm{Hom} (\Zp^\times, \mathbb G_m)$ denote the weight space which  we consider as a rigid analytic space over a finite
extension $E$ of $\Qp$. If $y\in \mathcal{WS}(E),$ we denote by $\kappa_y\,:\, \Zp^\times \rightarrow E^\times$  the associated character. We consider $\ZZ$ as a subset of  $\mathcal{WS}(E)$ identifying $k\in \ZZ$ with the character $u\mapsto u^k .$
Set 
\[
\mathcal W^*=\left \{y\in \mathcal{WS} \mid v_p(\kappa_y(a)^{p-1}-1)>
\frac{1}{p-1} \quad \forall a\in \Zp^\times \right \}.
\]
Note that $\ZZ \subset \mathcal W^*.$ 
If $U\subset \mathcal W^*$ is  either an affinoid disk
or a wide open disk  we will write $O_U^\circ$ for  the ring of analytic functions on $U$ that are bounded by $1$ and set $O_U=O_U^\circ [1/p].$

 We denote by $\kappa_U  \,:\, \Zp^\times \rightarrow O_{U}^{\circ,\times}$  the universal weight  character characterized by the property 
\[
\kappa_y= y\circ \kappa, \qquad
\forall y\in U(E).
\]
We remark that our assumption that $U\subset \mathcal W^*$ is to guarantee that $\kappa_U$ lands in $O_{U}^{\circ,\times}$.

\subsubsection{} 

Let $\mathbb A (U)^\circ$ (resp., $\mathbb A'(U)^\circ$) denote the 
space of   functions $f\,:\,\Zp^\times \times \Zp \rightarrow O_U^\circ$ (resp., $f\,:\,p\Zp \times \Zp^\times  \rightarrow O_U^\circ$)
satisfying the following properties: 
\begin{itemize}

\item[$\bullet$]{} $f$ is homogeneous of weight $\kappa_U,$ namely
\[
f (ax,ay)= \kappa_U (a) f(x,y), \qquad \forall a\in \Zp^\times ;
\]
\item[$\bullet$]{} The one variable function 
$ f(1,z)$ (resp., $ f(pz,1)$  is analytic, namely it can be written in the form
\[
\underset{m=0}{\overset{\infty}\sum} c_m z^m, \qquad c_m \in O_V^\circ 
\]
where $c_m$ goes to zero when $m\to \infty.$ Note that in the scenario when $U$ is wide open, we work with the $\mathfrak m_U$-adic topology. 
\end{itemize}

Set $\DD(U)^\circ :=\Hom_{O_U^\circ, \mathrm{cont}}(\mathbb A(U)^\circ, O_U^\circ),$ 
 $\mathbb A(U)=\mathbb A(U)^\circ  [1/p]$ and 
 $\DD(U):=\DD(U)^\circ[1/p]$. Similarly, we set 
 $\DD'(U)^\circ :=\Hom_{O_U^\circ, \mathrm{cont}}(\mathbb A'(U)^\circ, O_U^\circ),$ 
 $\mathbb A'(U)=\mathbb A'(U)^\circ  [1/p]$ and 
 $\DD'(U):=\DD'(U)^\circ[1/p]$.
 
Let
$$
\Sigma_0(p):=\left\{\left(\begin{array}{cc}a& b\\ c&d \end{array}\right)\in M_{2\times 2}(\ZZ_p): p\nmid a, p\mid c, ad-bc\neq 0 \right\}
$$ 
and  
 $$
\Sigma'_0(p):=\left\{\left(\begin{array}{cc}a& b\\ c&d \end{array}\right)\in M_{2\times 2}(\ZZ_p): p\nmid d, p\mid c, ad-bc\neq 0 \right\}.
$$ 
The monoid $\Sigma_0(p)$ acts naturally on $\Zp^\times \times \Zp$
and therefore on $\mathbb A(U)$ and $\DD (U).$ Similarly, 
the monoid $\Sigma'_0(p)$ acts  on $\Zp \times \Zp^\times$
and therefore on $\mathbb A'(U)$ and $\DD' (U).$

 If $\mathcal W=\Spm (R)$ is an affinoid, our $\DD(\mathcal W)$
 corresponds to the space of distributions $\mathbf D (R)[0]$ in \cite{bellaiche2012}.  If $U$ is an open wide disk, our $\DD(U)$
 and $\DD'(U)$ correspond to $D_{U,0}(T_0)$ and $D_{U,0}(T_0')$   of \cite{LZ1}.  If $U$ reduces to one point $y\in \mathcal{WS}(E),$ we write $\DD_y(E)$ (resp., $\DD'_y(E)$) instead of $\DD(U)$  (resp., $\DD^\prime (U)$) and ${\mathbb A}_y(E)$ (resp., ${\mathbb A}^\prime_y(E)$) instead of $\mathbb A (U)$.

\subsubsection{}
If $\mathcal W\subset \mathcal{WS}$ is an affinoid disk we will also work with the space of overconvergent modular symbols  $\DD^\dagger (\mathcal W)$ which 
corresponds to the space $\bD^\dagger [0](R)$ in \cite{bellaiche2012}.
To be more precise, we let $\mathbb A[r](\mathcal W)$ denote the space of functions $f\,:\, \Zp \rightarrow O_{\mathcal W}$ that are analytic on the closed disk of radius $r$ and let $\DD [r] (\mathcal W) =\mathrm{Hom}_{O_{\mathcal W},\textrm{cpt}} (\mathbb A[r](\mathcal W), O_{\mathcal W})$ denote the space of compact continuous linear maps $\mathbb A[r](\mathcal W) \rightarrow  O_{\mathcal W}.$ 
In particular, $\mathbb A[1](\mathcal W)=\mathbb A(\mathcal W)$ 
and $\DD [1](\mathcal W)=\DD (\mathcal W).$ 
For any $r' >r$ the natural map
$\mathbb A[r'](\mathcal W) \rightarrow \mathbb A[r](\mathcal W)$ induces 
a map $\DD [r] (\mathcal W) \rightarrow \DD [r'] (\mathcal W),$ and one defines  
\[
\DD^\dagger (\mathcal W)=\varprojlim_{r>0} \DD [r] (\mathcal W).
\]
If $\mathcal W\subset U,$ where $U$ is a wide open, then we have 
the restriction map $\mathbb A_U \rightarrow \mathbb A (\mathcal W),$
which induces an isomorphism of $O_{\mathcal W}$-modules
\[
\DD  (\mathcal W) \simeq \DD (U)_{\mathcal W},
\qquad \DD (U)_{\mathcal W}:=\DD (U) \widehat\otimes_{O_U}O_{\mathcal W}
\]
(c.f. \cite{andreattaiovitastevens2015}, Lemma~3.8). In particular, we have an
injection 

\begin{equation}
\label{eqn_DdaggerWtoDUsubW}
\DD^\dagger (\mathcal W) \lra \DD (U)_{\mathcal W} \,
.
\end{equation}

For any $y\in \mathcal{W}(E)$,  the natural map 
\begin{equation}\label{eqn_bellaiche_lemma3_6}
    \DD^\dagger (\mathcal W)\otimes_{\cO_{\mathcal W},y}E\lra \DD^\dagger_y(E)
\end{equation}
is a  $\Sigma_0(p)$-equivariant isomorphism of $E$-Banach spaces
(c.f. \cite{bellaiche2012}, Lemma~3.6).

We refer the reader to \cite{andreattaiovitastevens2015}, \cite{bellaiche2012}
and   \cite[Section 4]{LZ1} for further 
details pertaining to these objects.

\subsubsection{}
Put $\Gamma_p=\Gamma_1(N)\cap \Gamma_0(p)\subset {\rm SL}_2(\ZZ)$; note that this congruence subgroup is denoted by $\Gamma$ in \cite{bellaiche2012}. We denote by $H^*(\Gamma_p,-)$ (resp., $H^*_c(\Gamma_p,-)$) the cohomology of $\Gamma_p$ (resp., the cohomology with compact support).
Note that $H^0_c(\Gamma_p, -)$ and $H^2(\Gamma_p, -)$ vanish since $\Gamma_p$ is the fundamental group of an open curve.

\begin{lemma} 
\label{lemma: properties of specialization}
Let 
$\mathcal W \subset \mathcal{W}^*$ be  an affinoid disk. Then:
\item[i)] One has
\[
H^2_c(\Gamma_p, \mathbb D^\dagger (\mathcal W))\,{\cong}\,
\begin{cases}
\nonumber 
\{0\} &\text{if $0\notin \mathcal W(E)$}, \\
E  & \text {if $0\in \mathcal W(E)$}.
\end{cases}
\]

\item[ii)] For any $y\in \mathcal W (E),$ the specialization at $y$ gives rise to the short exact sequence 
\begin{equation}
\label{eqn: specialization exact sequence for D}
0 \lra \mathbb D^\dagger (\mathcal W) \lra
\mathbb D^\dagger (\mathcal W) \stackrel{y}{\lra}
 \mathbb D_y^\dagger (E)  \lra 0
 \end{equation}

\item[iii)]  The exact sequence \eqref{eqn: specialization exact sequence for D}
 induces an isomorphism
\begin{equation}
\label{eqn: specialization for non-compact first cohomology}
H^1(\Gamma_p, \mathbb D^\dagger (\mathcal W))\otimes_{O_{\mathcal W}, y}E
\simeq H^1(\Gamma_p, \mathbb D_y^\dagger (E))
\end{equation}
and an injection
\begin{equation}
\label{eqn: specialization for compact first cohomology}
H^1_c(\Gamma_p, \mathbb D^\dagger (\mathcal W))\otimes_{O_{\mathcal W}, y}E
\hookrightarrow H^1_c(\Gamma_p, \mathbb D_y^\dagger (E))
\end{equation}

\item[iv)] If $y\neq 0,$ the map \eqref{eqn: specialization for compact first cohomology}
is an isomorphism. If $y=0,$ then  ${\rm coker}\eqref{eqn: specialization for compact first cohomology}$ is isomorphic to $E.$

\item[v)] The statements i-iv) remain true if $\mathbb D^\dagger (\mathcal W)$
is replaced by $\mathbb D (U),$ where $U\subset \mathcal{W}^*$ is a wide open.
 \end{lemma}
 \begin{proof} 
  The part i) is  proved in \cite[Lemma~3.9]{bellaiche2012}
  and the same arguments compute $H^2_c$ with coefficients in $\DD (U).$
 
 The exactness of \eqref{eqn: specialization exact sequence for D} in part ii)  is proved in 
 \cite[Lemma~3.6]{bellaiche2012}, whereas its analogue for $\DD (U)$ in \cite[Proposition~3.11]{andreattaiovitastevens2015}.
 
 The isomorphism \eqref{eqn: specialization for non-compact first cohomology} in part iii)  follows from the long exact sequence associated to the specialization exact sequence and the vanishing
 of $H^2(\Gamma_p,  \mathbb D^\dagger (\mathcal W)).$ 
The inclusion \eqref{eqn: specialization for compact first cohomology} for the cohomology with compact support and the point iv) 
follow formally from i) (c.f. \cite{bellaiche2012}, Theorem~3.10), and
the proofs of  these statements for cohomology with coefficients in $\DD (U)$
are the same.

 \end{proof}

 \subsubsection{}
 \label{subsubsec_wideopen_vs_bellaiche}
  Consider the spaces of modular symbols
\[
{\rm Symb}_{\Gamma_p}(\DD):=
\mathrm{Hom}_{\Gamma_p} (\mathrm{Div}^0({\mathbb P}^1(\QQ)),\DD),
\qquad 
\DD= \DD (U), \DD^\dagger_y,  \DD^\dagger (\mathcal W),
\]
where $\mathrm{Div}^0({\mathbb P}^1(\QQ))$ is the group of divisors
of degree $0$ on ${\mathbb P}^1(\QQ).$
Recall   the functorial Hecke equivariant isomorphism of Ash and Stevens~\cite[Proposition~4.2]{ashstevens1986}
\begin{equation}
\label{eqn_AshStevesComparison}
{\rm Symb}_{\Gamma_p}(\DD )\stackrel{\sim}{\lra} H^1_c(\Gamma_p, \DD).
\end{equation}

For  a positive real number $\nu$, we  denote by 
$ {\rm Symb}_{\Gamma_p}(\mathbb{D})^{\leq \nu}$
 the space of modular symbols  of slope bounded by $\nu$ and by 
$$ {\rm Symb}_{\Gamma_p}^{\pm}(\mathbb{D})^{\leq \nu}\subset {\rm Symb}_{\Gamma_p}(\DD )^{\leq \nu}$$ 
 the $\pm1$-eigenspace for the involution $\iota$ given by the action of the matrix $\left(\begin{array}{cr} 1 &0 \\ 0&-1 \end{array}\right)$
(see for example, \cite[Section 3.2.4]{bellaiche2012}).

 \begin{lemma} 
\label{lemma: shrinking wide open neighborhood}
Let $U$ be a wide open and $\mathcal W\subset U$
an affinoid disk. 
Then we have  canonical isomorphisms
\[
\begin{aligned}
\mathrm{Symb}_{\Gamma_p}(\DD (U)_{\mathcal W}) &\simeq \mathrm{Symb}_{\Gamma_p}(\DD (U))_{\mathcal W}\\
H^1(\Gamma_p, \DD (U)_{\mathcal W}) &\simeq H^1(\Gamma_p, \DD (U))_{\mathcal W}
\end{aligned}
\]
where $\DD (U)_{\mathcal W}=\DD (U) \widehat\otimes_{O_U}O_{\mathcal W}$
and  $\mathrm{Symb}_{\Gamma_p}(\DD (U))_{\mathcal W}=
 \mathrm{Symb}_{\Gamma_p}(\DD (U))\widehat\otimes_{O_U}O_{\mathcal W}.$
\end{lemma}
\begin{proof}
Only in this proof, $A$ shall denote a general ring. For any complex $C^\bullet$ of $A$-modules over 
a ring $A$ and a morphism of rings $A\rightarrow B,$  one has the spectral  sequence
\[
E_2^{ij}=\mathrm{Tor}_{-i}^A(H^j(C^\bullet), B )\, \implies \, H^{i+j}(C^\bullet\otimes_A B)\,.
\]
Let us choose $C^\bullet$ as the complex computing the cohomology with compact  support with coefficients in $\DD (U)^\circ ,$ $A=O_U^\circ$ and $B=O_{\mathcal W}^\circ /p^n$.  Then our spectral sequence induces an exact sequence
\begin{equation}
\label{eqn:shrinking lemma exact sequence}
0\lra  H^1_c(\Gamma_p, \DD (U)^\circ )\otimes_{O_U^\circ} O^\circ_{\mathcal W}/p^n
 \lra 
H^1_c(\Gamma_p, \DD (U)^\circ\otimes_{O_U^\circ} O^\circ_{\mathcal W}/p^n)
\lra 
\mathrm{Tor}_{1}^{O_U^\circ} (H^2_c (\Gamma_p, \DD (U)^\circ ) , O_{\mathcal W}^\circ /p^n ).
\end{equation}
It follows from  Lemma~\ref{lemma: properties of specialization}(i) (combined with (v) of the same lemma) that $H^2_c(\Gamma_p, \DD(U)^{\circ})$ is a finitely generated $O_E$-module and therefore, there exists a natural number $N$ such that for any $n\geqslant 1$, the $O_E$-module 
\[
\mathrm{Tor}_{1}^{O_U^\circ} (H^2_c (\Gamma_p, \DD (U)^\circ ) , O_{\mathcal W}^\circ /p^n )
\]
is annihilated by $p^N.$ On passing to projective limits in \eqref{eqn:shrinking lemma exact sequence} and  using \cite[Lemma~3.13]{andreattaiovitastevens2015}, we infer that 
\[
0\lra 
 H^1_c(\Gamma_p, \DD (U)^\circ )\widehat \otimes_{O_U^\circ} O^\circ_{\mathcal W}
 \lra 
H^1_c(\Gamma_p, \DD (U)^\circ\widehat \otimes_{O_U^\circ} O^\circ_{\mathcal W})
\lra  C
\]
 where $C$ is a finitely generated module that is annihilated by $p^N$. This concludes the proof of the first isomorphism. The proof of the 
 second isomorphism is analogous (an even simpler because $H^2 (\Gamma_p, \DD (U)^\circ )$ vanishes).
\end{proof}

\subsubsection{}
For a non-negative integer $k$, we let $\bP_k(E)\subset E[Z]$ denote the space of polynomials of degree less or equal to $k$, which is equipped with a left ${\rm GL}_2(\ZZ_p)$-action; see \cite[3.2.5]{bellaiche2012} for a description of this action. We let $\bV_k(E)$ denote the $E$-linear dual of $\bP_k(E)$, which also carries an induced ${\rm GL}_2(\ZZ_p)$-action.

Regarding elements $\mathscr{P}_k(E)$ as analytic functions, we have an induced map 
$$\rho_k^*: \DD^\dagger_k (E) \lra \mathscr{V}_k(E)$$
which is a $\Sigma_0(p)$-equivariant surjection. By Stevens' control theorem  (c.f. \cite{PollackStevensJLMS}, Theorem 5.4) it induces a Hecke equivariant 
isomorphism of $E$-vector spaces 
\begin{equation}
\label{eqn:stevens control thm}
\rho_k^* \,:\, \mathrm{Symb}_{\Gamma_p}(\DD^\dagger_k(E))^{< k+1}
\xrightarrow{\sim} 
\mathrm{Symb}_{\Gamma_p}(\bV_k (E))^{< k+1}
\end{equation} 
We also have the non-compact version of this isomorphism, proved 
by Ash and Stevens \cite[Theorem 1]{AshStevens2008Unpublished} :
\begin{equation}
\label{eqn_AshStevens2008Unpublished_1}
H^1(\Gamma_p, \mathbb{D}^\dagger_k(E))^{< k+1}\xrightarrow{\,\,\sim\,\,} H^1(\Gamma_p, \mathscr{V}_k(E))^{<k+1}.
\end{equation}

\subsection{Local description of the Coleman--Mazur--Buzzard eigencurve}
\label{subsec_BuzzardColemanMazur}

Our main objective in this subsection is to record 
Proposition~\ref{prop_bellaiche_4_11} and    Theorem~\ref{thm_bellaiche} (which is due to Bella\"iche; see the second paragraph in \cite[\S 1.5]{bellaiche2012}).

\subsubsection{}
\label{defn_Bellaiche_2_1_1_3}
We now briefly recall the Coleman--Mazur--Buzzard construction of the eigencurve. We retain the notation from Section~\ref{subsec_bellaiche_eigencurve} and continue following the exposition in \cite[\S2.1]{bellaiche2012}.
Let ${\mathcal H}$ denotes the Hecke algebra generated over $\ZZ$ by the Hecke operators $\{T_\ell\}_{\ell\nmid Np}$, the Atkin--Lehner operator $U_p$ and diamond operators $\{\langle d \rangle\}_{d\in (\ZZ/N\ZZ)^\times}.$ For any affinoid
$\mathcal W$ we set $\mathcal H_{\mathcal W}=\mathcal H \otimes_{\ZZ}O_{\mathcal W}.$ 

We fix a nice affinoid disk $\mathcal W=\Spm (O_{\mathcal W})$ (in the sense of \cite{bellaiche2012}, Definition 3.5) of the weight space $\mathcal{WS}$ adapted to slope $\nu$ in the sense of \cite{bellaiche2012}, \S 3.2.4.
Let us denote by  $M^\dagger(\Gamma_p,\mathcal{W})$
   Coleman's space of overconvergent modular forms  of level $\Gamma_p$ and weight in  $\mathcal{W}$; and let $M^\dagger(\Gamma_p,\mathcal{W})^{\leq \nu}$ denote its $O_{\mathcal W}$-submodule on which $U_p$ acts with slope at most $\nu$. We similarly let $S^\dagger(\Gamma_p,\mathcal{W})$ denote the space of cuspidal overconvergent modular forms of level $\Gamma_p$ and weight in the affinoid disk $\mathcal{W}$; and $S^\dagger(\Gamma_p,\mathcal{W})^{\leq \nu}\subset S^\dagger(\Gamma_p,\mathcal{W})$. 

Let $\mathbb{T}_{\mathcal{W},\nu}$ (resp., $\mathbb{T}^\cusp_{\mathcal{W},\nu}$)
denote the image of $\mathcal H_{\mathcal W}$ in  $ {\rm End}_{O_{\mathcal W}}\left( M^\dagger(\Gamma_p,\mathcal{W}) \right)$ (resp.,  ${\rm End}_{O_{\mathcal W}}\left( S^\dagger(\Gamma_p,\mathcal{W}) \right)$).  
Then   $\mathcal{C}_{\mathcal{W},\nu}:={\rm Spm}(\mathbb{T}_{\mathcal{W},\nu})$
(resp., $\mathcal{C}^\cusp_{\mathcal{W},\nu}:={\rm Spm}(\mathbb{T}^\cusp_{\mathcal{W},\nu})$ is an open affinoid subspace of the Coleman--Mazur--Buzzard eigencurve $\mathcal{C}$  (resp, cuspidal eigencurve $\mathcal{C}^\cusp$) 
that lies over $\mathcal{W}$.


We let $\TT_{\mathcal{W},\nu}^\pm$ denote the 
image of $\mathcal H_{\mathcal W}$ in  $\mathrm{End}_{O_{\mathcal W}} ( {\rm Symb}^{\pm}_{\Gamma_p}(\mathbb{D}(\mathcal W))^{\leq \nu})$.  We then define the open affinoids $\mathcal{C}_{\mathcal{W},\nu}^\pm:={\rm Spm}(\TT_{\mathcal{W},\nu}^\pm)$.
By  \cite[Theorem 3.30]{bellaiche2012}, there exist canonical closed immersions
\begin{equation}
\label{eqn: inclusions for eigencurves}
\mathcal{C}_{\mathcal{W},\nu}^{\cusp} \hookrightarrow \mathcal{C}_{\mathcal{W},\nu}^\pm \hookrightarrow \mathcal{C}_{\mathcal{W},\nu}, \qquad
\mathcal{C}_{\mathcal{W},\nu}=\mathcal{C}_{\mathcal{W},\nu}^+ \cup
\mathcal{C}_{\mathcal{W},\nu}^- .
\end{equation}
 The open affinoids $\mathcal{C}_{\mathcal{W},\nu}^\pm$ admissibly cover the Coleman--Mazur--Buzzard eigencurve $\mathcal{C}$ as $(\mathcal W,\nu)$ varies (see \cite{bellaiche2012}, \S 3).
 
 \subsubsection{} In the remainder of this subsection, we fix a $p$-stabilization of a newform  $f_0$ and denote by $x_0$ the corresponding point
 of $\mathcal C.$  The following proposition will allow us to apply the formalism
 of Sections~\ref{sec_triangulations}-\ref{sec_2varpadicLabstract} in the context of eigencurves. We denote by $\mathcal W$ an affinoid neighborhood
of $k_0=\w(x_0)$ in the weight space, where $O_{\mathcal W}$ is of the form 
$O_{\mathcal W}=E\left <Y/p^m\right >$.

\begin{proposition}[Bella\"iche]
\label{prop_bellaiche_4_11}
Up to shrinking   $\mathcal W$ and enlarging $E$, there exists an affinoid neighborhood $\mathcal X$ of $x_0\in \mathcal{C}_{\mathcal{W},\nu}^{\mathrm{cusp}}$ such that 
\item[a)] $\mathcal X$ is a connected component of $\mathcal C_{\mathcal W,\nu}$;
\item[b)]There exist  integers $r, e\geqslant 1$ and  an element   $a\in O_{\mathcal X}$ such that $O_{\mathcal W}=E\left <Y/p^{re} \right >$  and the map $X \mapsto a$ induces  an isomorphism  
 of $R$-algebras
\[
O_{\mathcal W}[X]/(X^e-Y)\simeq O_{\mathcal X}.
\] 
\end{proposition}
\begin{proof}
By \cite[Proposition 4.11]{bellaiche2012}, there exists an affinoid neighborhood $\mathcal X$ of $\mathcal{C}_{\mathcal{W},\nu}^\pm$ which satisfies 
the condition (b). By \cite[Corollary 2.17]{bellaiche2012}, the eigencurves 
$\mathcal C$ and $\mathcal C^\cusp$ are locally isomorphic at $x_0.$ Together with 
\eqref{eqn: inclusions for eigencurves}, this concludes the proof of (a).
\end{proof}

\subsubsection{}
We set 
\[
 {\rm Symb}_{\Gamma_p}^{\pm}(\mathcal X)^{\leq \nu}:= {\rm Symb}_{\Gamma_p}^\pm(\mathbb{D}^\dagger (\mathcal W))^{\leq \nu}\otimes_{\TT_{\mathcal{W},\nu}^{\pm}}O_{\mathcal X},
\]
where $\mathcal X$ is as in Proposition~\ref{prop_bellaiche_4_11}. 

\begin{theorem}[Bella\"iche]
\label{thm_bellaiche}
\item[i)] The $O_{\mathcal X}$-module $ {\rm Symb}_{\Gamma_p}^{\pm}(\mathcal X)^{ \leq \nu}$ is free of rank one. 
\item[ii)] The $O_{\mathcal X}$-module $ {\rm Symb}_{\Gamma_p}(\mathcal X)^{\leq \nu}$ is free of rank $2$.
\end{theorem}
\begin{proof}
The first assertion follows as a consequence of the discussion in Section 4.2.1 in op. cit.; see also the first paragraph following the proof of \cite[Proposition 4.11]{bellaiche2012}. The second assertion follows from the first one 
and the decomposition 
\[
{\rm Symb}_{\Gamma_p}(\mathcal X)^{\leq \nu}= 
{\rm Symb}^+_{\Gamma_p}(\mathcal X)^{\leq \nu} \oplus 
{\rm Symb}^-_{\Gamma_p}(\mathcal X)^{\leq \nu}.
\]
\end{proof}

\subsection{A variant of Bella\"iche's construction}
\label{subsec_Bellaiche_revised}

\subsubsection{}


Fix an affinoid disk  $\mathcal  W$ and  a wide open affinoid $U$  such that 
\[
\mathcal W \subset U\subset \mathcal W^*
\]
where $\mathcal W^*$ is given as in \S\ref{subsubsec_bellaiche_eigencurve_1}. See  \cite{andreattaiovitastevens2015} as well as \S\ref{subsubsec_bellaiche_eigencurve_1} for an explanation of this assumption. Note that this condition translates to the requirement that  $U$ is $0$-accessible in the sense of \cite[Definition 4.1.1]{LZ1}. We then have a natural (restriction) map $\LLU\to R$.
Recall that $\kappa_U \,:\, \Zp^\times \rightarrow O_U^{\circ, \times}$
denotes the restriction of the universal weight character to $U$. 
In what follows, we will allow ourselves to shrink both $U$ and $\mathcal W$ as necessary.

\subsubsection{}
Recall that  ${\mathcal H}_{\mathcal W}$ denotes the Hecke algebra over $O_{\mathcal W}.$ In this subsection, we work with the spaces $\DD (U)$ and 
$\mathcal A^\prime (U).$ Set 
\begin{equation}
\nonumber
\begin{aligned}
&H^1_c(\Gamma_p,\DD (U))^{\pm,\leq \nu}_{\mathcal W}:=
H^1_c(\Gamma_p,\DD (U))^{\pm,\leq \nu}\otimes_{\LLUp}O_{\mathcal W}\,,\\
&H^1_c(\Gamma_p,\mathbb A^\prime (U))^{\pm,\leq \nu}_{\mathcal W}:=
H^1_c(\Gamma_p, \mathbb A^\prime (U) )^{\pm,\leq \nu}\otimes_{\LLUp}O_{\mathcal W}
\end{aligned}
\end{equation}
and  let
\begin{equation}
\nonumber 
r_1^\pm\,:\, \mathcal H_{\mathcal W} \rightarrow  {\rm End}_{O_{\mathcal W}}
( H^1_c(\Gamma_p,\DD (U))^{\pm,\leq \nu}_{\mathcal W} ) \quad, \quad 
r_2^\pm\,:\, \mathcal H_{\mathcal W} \rightarrow  {\rm End}_{O_{\mathcal W}}
(H^1_c(\Gamma_p,\mathbb A^\prime (U))^{\pm,\leq \nu}_{\mathcal W})
\end{equation}
denote the canonical representations of $\cH_{\mathcal W}$. 
We define the ideals $I^\pm_1=\ker (r_1^\pm)$ and $I^\pm_2=\ker (r_2^\pm)$.

 \begin{lemma}
\label{prop_bellaiche_vs_AIS}
Let $k_0\geqslant 0$ be an integer such that $k_0\in \mathcal W\subset U$. For sufficiently small $\mathcal W$ and $U$ 
the map \eqref{eqn_DdaggerWtoDUsubW} together with  Lemma~\ref{lemma: shrinking wide open neighborhood}
induce   natural Hecke equivariant isomorphisms
 \begin{equation}
\label{eqn_prop_bellaiche_vs_AIS}
\begin{aligned}
H^1_c(\Gamma_p, \mathbb{D}^\dagger (\mathcal W))^{\pm,\leq \nu} &\simeq  H^1_c(\Gamma_p,\DD (U))^{\pm,\leq \nu}_{\mathcal W},\\
H^1(\Gamma_p, \DD^\dagger (\mathcal W))^{\pm,\leq \nu} &\simeq
H^1(\Gamma_p, \DD (U))^{\pm,\leq \nu}_{\mathcal W}
.
\end{aligned}
\end{equation}
In particular, the morphism
$$\TT_{\mathcal W,\nu}^{\pm}  \lra \mathrm{im} (r^\pm_1)$$
induced from \eqref{eqn_prop_bellaiche_vs_AIS} is an isomorphism as well.
\end{lemma}

\begin{proof}
Let us put $M_1:= H^1_c(\Gamma_p, \mathbb{D}^\dagger (\mathcal W))^{\pm,\leq \nu}$ and $M_2:=H^1_c(\Gamma_p,\DD (U))^{\pm,\leq \nu}_{\mathcal W} $
and denote    by  $g\,:\, M_1 \rightarrow M_2$ the map induced from  \eqref{eqn_DdaggerWtoDUsubW} together with  Lemma~\ref{lemma: shrinking wide open neighborhood}. 
It follows from \eqref{eqn_AshStevesComparison} and  general results about the slope decomposition that $M_1$ and $M_2$ are finitely generated free 
$O_{\mathcal W}$-modules.  
By Lemma~\ref{lemma: properties of specialization}, we have a commutative diagram
\[
\xymatrix{
0 \ar[r] &M_1/\mathfrak m_{k_0} M_1 \ar@{.>}[d] \ar[r]  &H^1_c(\Gamma_p,\mathbb{D}^\dagger_{k_0}(E))^{\pm,\leq \nu} \ar[d] \ar[r] & {H^2_c(\Gamma_p, \mathbb D^\dagger  (\mathcal W))^{\pm,\leq \nu}[k_0]} \ar[d] \ar[r]& 0\\
0 \ar [r] & M_2/\mathfrak m_{k_0} M_2 \ar[r] &H^1_c(\Gamma_p,\mathbb{D}_{k_0}(E))^{\pm,\leq \nu}\ar[r] & {H^2_c(\Gamma_p, \mathbb,\mathbb{D} (U))_{\mathcal W}^{\pm,\leq \nu}[k_0]} \ar[r] &0
}
\]
where $\mathfrak m_{k_0} \subset \cO_{\mathcal W}$ is the maximal ideal determined by the weight $k_0$. By \cite[Lemma~5.3]{PollackStevensJLMS}, the middle vertical map of the diagram is an isomorphism.  According to Lemma~\ref{lemma: properties of specialization}, both $H^2_c(\Gamma_p, \mathbb D^\dagger (\mathcal W))^{\pm,\leq \nu}[k_0]$ and $H^2_c(\Gamma_p, \mathbb,\mathbb{D} (U))_{\mathcal W}^{\pm,\leq \nu}[k_0]$ are both either isomorphic to $E$ or else vanish (depending on whether $k_0=0$ or not). Moreover, since the middle vertical arrow is an isomorphism, the right vertical arrow is surjective and therefore an isomorphism as well. We conclude that the left vertical map induced from the vertical isomorphism in the middle is an isomorphism:
\begin{equation}
\label{eqn: isomorphism of M} 
M_1/\mathfrak m_{k_0} M_1 \stackrel{\sim}{\lra} M_2/\mathfrak m_{k_0} M_2\,.
\end{equation}
In particular, $M_1$ and $M_2$ are free $O_{\mathcal W}$-modules with the same rank. Let $G\in O_{\mathcal W}$ denote the determinant of $g$ in some bases of $M_1$ and $M_2$. The isomorphism 
\eqref{eqn: isomorphism of M} shows that 
$G(k_0)\neq 0$ and hence, on shrinking the neighborhood $\mathcal W$ as necessary, $G$ is non-vanishing on $\mathcal{W}$.  This concludes the proof that $g$ is an isomorphism for sufficiently small $\mathcal W$ and $U$,  as we have asserted in the statement of our lemma. 

The second isomorphism in  \eqref{eqn_prop_bellaiche_vs_AIS}
can be established by the same argument, using the fact that the map $H^1(\Gamma_p, \DD^\dagger_{k_0}(E))^{\leq \nu}=H^1(\Gamma_p,\DD_{k_0}(E))^{\leq \nu}$ by \cite[Theorem~5.5.3]{AshStevens2008Unpublished}.
\end{proof}
 
\subsubsection{} Let us set  
\[
\TT_{\mathcal W,\nu}^{\prime, \pm} = \mathrm{im} (r^\pm_2),
\qquad \qquad 
{\mathcal C}_{\mathcal W,\nu}^{\prime,\pm}:=\Spm (\TT_{\mathcal W,\nu}^{\prime, \pm}).
\]
Note then that we have canonical isomorphisms $\TT_{\mathcal W,\nu}^{\pm}\simeq  \cH_{\mathcal W}/I^\pm_1$
and $\TT_{\mathcal W,\nu}^{\prime,\pm}\simeq  \cH_{\mathcal W}/I^\pm_2$.

\begin{theorem} 
\label{thm_eigencurve_variant}
\item[i)] We have $I_1^\pm=I_2^\pm.$ 
In particular, there exist  canonical isomorphisms
$\TT_{\mathcal W,\nu}^{\pm}\simeq \TT_{\mathcal W,\nu}^{\prime,\pm}$ and   
${\mathcal C}_{\mathcal W,\nu}^{\pm} \simeq {\mathcal C}_{\mathcal W,\nu}^{\prime,\pm}$\,.

\item[ii)]  Suppose $x_0\in \Spm ({\mathcal H}_{\mathcal W}/I^\pm)$, where $I^\pm :=I_1^\pm=I_2^\pm.$ For a sufficiently small wide open neighborhood $U$ of $k_0=\w (x_0)$ and an affinoid 
$\mathcal W \subset U$ containing $k_0$ the following two assertions hold true.
\begin{itemize}
\item[a)] $x_0$ belongs to a unique connected component $\mathcal X^{\pm}$ of
$\Spm ({\mathcal H}_{\mathcal W}/I^\pm)$\,. 
\item[b)] Both  $O_{{\mathcal X}^\pm}$-modules 
$H^1_c(\Gamma_p,\DD (U))^{\pm,\leq \nu}_{\mathcal W}\otimes_{ \cH_{\mathcal W}}O_{{\mathcal X}^\pm}$ and  $H^1_c(\Gamma_p,{\mathbb A}^\prime (U))^{\pm,\leq \nu}_{\mathcal W}\otimes_{ \cH_{\mathcal W}}O_{{\mathcal X}^\pm} $ are free of rank one. 
\end{itemize}
\item[iii)] Fix a wide open $U$ as well as an affinoid neighborhood $\mathcal{W}\subset U$ of $x_0$ chosen so as to ensure that the conclusions of Part ii) are verified. 
If $x_0$ is cuspidal, then $x_0\in \mathcal C_{\mathcal W,\nu}^+ \cap \mathcal C_{\mathcal W,\nu}^- $  and $\mathcal X^+=\mathcal X^-$.
\end{theorem}
\begin{proof} 
\item[i)] We first show that for  every integer  $k\in U$ with $k\geq \nu$
there exists a Hecke equivariant\footnote{Note here and below that we are not specifying the covariance or the contravariance of the Hecke action that ensures Hecke compatibilities.} isomorphism
\begin{equation}
\label{formula: theorem about cohomology analytic functions}
 H^1_c(\Gamma_p,\DD_{k}(E))^{\pm,\leq \nu} \stackrel{\sim}{\lra}
  H^1_c(\Gamma_p,\mathbb A_{k}^{\prime} (E))^{\pm,\leq \nu} .
 \end{equation}
For such $k$, the natural injection 
$$H^1_c(\Gamma_p,\mathscr{P}_{k}(E))^{\leq \nu}\lra H^1_c(\Gamma_p,{\mathbb A}_{k}^{\prime} (E))^{\leq \nu}$$
as well as the natural surjection 
$$H^1_c(\Gamma_p,\DD_{k}(E))^{\leq \nu}\lra H^1_c(\Gamma_p,\mathscr{V}_{k}(E))^{\leq \nu}$$
of Hecke modules are isomorphisms thanks to the control theorem of Ash and Stevens (c.f. \cite{BSV2020}, Proposition 4.2.2).

 We are therefore reduced to proving that we have an Hecke equivariant isomorphism
$$  H^1_c(\Gamma_p,\mathscr{V}_{k}(E)) \stackrel{\sim}{\lra} H^1_c(\Gamma_p,\mathscr{P}_{k}(E)) \,.$$ 
This follows from the isomorphism
$$\mathscr{V}_{k}(E)\otimes \det{}^{k}\xrightarrow[(\det^*)^{-1}]{\sim}\Hom(\mathscr{V}_{k}(E),E)=:\mathscr{P}_{k}(E)$$
of $\GL_2(\ZZ_p)$-modules (where $\det$ stands for the determinant character of $\GL_2(\ZZ_p)$), together with the fact that $\det_{\vert_{\Gamma_p}}=1$. Here, $\det^*$ is the isomorphism induced from the perfect pairing 
\begin{align*}
\mathscr{V}_{k}(E)\otimes \mathscr{V}_{k}(E) &\lra E \otimes\det{}^{-k}\\
v_1\otimes v_2&\longmapsto \sum_{j=0}^k {k \choose j}(-1)^jv_1(Z^j)v_2(Z^{k-j})\,;
\end{align*}
see the paragraph following \cite[Proposition 3.2]{BSV2020} as well as Remark 3.3.2 in op. cit. This proves \eqref{formula: theorem about cohomology analytic functions}.

In  view of \cite{BSV2020}, Proposition 4.2.1, it follows from the isomorphisms \eqref{formula: theorem about cohomology analytic functions} that the representations  $r_1$ and $r_2$ verify the conditions of \cite[Proposition 3.7]{Chenevier2005Duke}. This implies that $I^\pm_1=I^\pm_2$ and the remaining assertions are immediate from this fact and Lemma~\ref{prop_bellaiche_vs_AIS}.

\item[ii)]

We will only prove the assertion when the coefficients are $\mathbb A^\prime (U)$, since the proof in the case of $\mathbb D(U)$ is similar. We remark that one may alternatively deduce our claim that the $O_{\mathcal X^\pm}$-module $H^1_c(\Gamma_p,\DD (U))^{\pm,\leq \nu}\otimes_{\cH_{\mathcal W}}O_{\mathcal X^\pm}$ is free of rank one from \cite[Proposition~4.5]{bellaiche2012}, using the isomorphism \eqref{eqn_AshStevesComparison} of Ash--Stevens and Proposition~\ref{prop_bellaiche_vs_AIS}.

We will prove that for $U$ sufficiently small the $O_{\mathcal X^\pm}$-module $H^1_c(\Gamma_p,\mathbb A^\prime (U))^{\pm,\leq \nu}\otimes_{\cH_{\mathcal W}} O_{\mathcal X^\pm}$ is free of rank one.  It is clear that one may choose $U$ so as to ensure that the condition a) holds. Let us fix such $U$. According to  \cite[Proposition 4.2.1]{BSV2020}, the $\TT_{U,\nu}^{\prime,\pm}$-module $H^1_c(\Gamma_p,{\mathbb A}^\prime (U))^{\pm,\leq \nu}_{O_{\mathcal W}}$ is of finite rank as an $O_{\mathcal W}$-module (therefore also as a $\TT_{ O_{\mathcal W},\nu}^{\prime,\pm}$-module). For any $y\in U,$  the long exact sequence $\Gamma_p$-cohomology induced from the short exact sequence of specialization 
at $y$
$$0\lra \mathbb A^\prime (U) {\lra} \mathbb A^\prime (U) \xrightarrow{y} 
{\mathbb A}^\prime_{y}(E)\lra 0$$
shows that 
$$H^1_c(\Gamma_p,{\mathbb A}^\prime (U))[\mathfrak m_y]={\rm im}(H^0_c(\Gamma_p, {\mathbb A}^\prime_{y}(E))\lra H^1_c(\Gamma_p,{\mathbb A}^\prime (U))).
$$
Since   $H^0_c(\Gamma_p, {\mathbb A}^{\prime}_{y}(E))=0,$ this proves that $H^1_c(\Gamma_p,{\mathbb A}^\prime (U))$ is torsion-free
over $\LLUp$ and therefore, the $O_{\mathcal W}$-module $H^1_c(\Gamma_p,{\mathbb A}^\prime (U))\otimes_{\LLUp}O_{\mathcal W}$ is torsion-free as well. We have proved that  the $O_{\mathcal W}$-modules $H^1_c(\Gamma_p,{\mathbb A}^\prime (U))^{\pm, \leq \nu}_{\mathcal W}$ are finitely generated and torsion-free (therefore free, since $O_{\mathcal W}$ is a PID). 

Set $N^\pm=H^1_c(\Gamma_p,{\mathbb A}'(U))^{\pm, \leq \nu}\otimes_{\cH_{\mathcal W}}{O_{\mathcal X^\pm}}.$ 
Now the arguments of Bella\"{\i}che apply verbatim. More precisely, since ${O_{\mathcal X^\pm}}$ are PIDs, \cite[Lemma~4.1]{bellaiche2012} tells us that for any $x\in \mathcal X^{\pm}$, the localization ${N_{(x)}^\pm}$ is a 
free $O_{\mathcal X^\pm, x}$-module of finite rank. Moreover, for any classical 
point $x\neq x_0$ of weight $k\geq \nu$, the isomorphism \eqref{formula: theorem about cohomology analytic functions} together with 
\cite[Lemma~2.8]{bellaiche2012} show  that ${ N_{(x)}^\pm}$ is of rank one over $O_{\mathcal X^\pm, x}.$
By the local constancy of the rank, the same also holds true for ${ N_{(x_0)}^\pm}.$
We conclude that $N^\pm$ is a free $O_{\mathcal X^\pm}$-module of of rank one, when $U$ and $\mathcal W$ sufficiently small. 

\item[iii)] This portion follows directly from \cite[Theorem~3.30]{bellaiche2012}.
\end{proof}

\begin{corollary}
\label{cor_main_Bellaiche_BCM_eigencurve_variant}
Let us fix $\mathcal{W}$ that ensures the validity of the conclusions of Theorem~\ref{thm_eigencurve_variant}(ii) and suppose $x_0$ is cuspidal, so that 
$ \mathcal X^+=\mathcal X^-$ according to  Theorem~\ref{thm_eigencurve_variant}(iii). Put $\mathcal X=\mathcal X^\pm$. Then both $O_{\mathcal X}$-modules $H^1_c(\Gamma_p,\DD (U))^{\leq \nu}_{\mathcal W}\otimes_{ \cH_{\mathcal W}}O_{\mathcal X}$ and $H^1_c(\Gamma_p,{\mathbb A}^\prime (U))^{\leq \nu}_{\mathcal W}\otimes_{ \cH_{\mathcal W}}O_{\mathcal X}$ are free of rank $2$.
\end{corollary}

\begin{proof}
Clear, thanks to Theorem~\ref{thm_eigencurve_variant}. 
\end{proof}

\subsubsection{} 
We record the following  proposition which will shall use in the proof of Proposition~\ref{prop_andreattaiovitastevens2015_Prop318}. 
It is well known to experts, but we prove it for reader's convenience.   

\begin{proposition}
\label{claim}
 Let $x_0$ be a classical cuspidal point on the eigencurve $\mathcal C$
 and $k_0=w (x_0)$. Then 
the following hold true. 
\item[i)] The natural map
\begin{equation}
\label{eqn_AISvsASbisbisbis}
H^1_c(\Gamma_p, \mathscr{V}_{{ k_0}}(E))\otimes_{\cH,{ x_0}}E\lra H^1(\Gamma_p, \mathscr{V}_{{ k_0}}(E))\otimes_{\cH,{ x_0}}E
\end{equation}
is an isomorphism of $E$-vector spaces of dimension $2$.

\item[ii)] There exists an affine neighborhood $\mathcal W$ of ${ k_0}$
and $\nu >0$  
such that for the connected component $\mathcal X\subset \mathcal C_{\mathcal W, \nu}$ of $x_0\in \mathcal C_{\mathcal W, \nu}$  the map
\begin{equation}
\label{eqn_AISvsAS}
H^1_c(\Gamma_p, \mathbb{D}^\dagger(\mathcal W))^{\leq \nu}\otimes_{\cH_{\mathcal W}}O_{\mathcal X} \xrightarrow{\frak{j}\otimes {\rm id}} H^1(\Gamma_p, \mathbb{D}^\dagger(\mathcal W))^{\leq \nu}\otimes_{\cH_{\mathcal W}} O_{\mathcal X}
\end{equation}
induced from  the natural $\mathcal{H}_{\mathcal W}$-equivariant map $H^1_c(\Gamma_p, \mathbb{D}^\dagger (\mathcal W))^{\leq \nu}\xrightarrow{\frak{j}} H^1(\Gamma_p, \mathbb{D}^\dagger (\mathcal W))^{\leq \nu}$, is an isomorphism.
\end{proposition}

\begin{proof}
\item[i)] Let us denote by $\widetilde{Y}$ the Borel--Serre compactification of 
the modular curve $Y=\Gamma_p \setminus \mathbf{H}$ (and $\partial \widetilde{Y}$ its boundary). The sheaf $ \mathscr{V}_{k_0}$ on $Y$ extends to a sheaf on $\widetilde{Y}$. It follows from \cite[Proposition 4.2]{ashstevens1986} that 
$${\rm coker}\eqref{eqn_AISvsASbisbisbis}\hookrightarrow{\rm im} \left(H^1(Y, \mathscr{V}_{k_0})\stackrel{\partial^1_Y}{\lra} H^1(\partial \widetilde{Y}, \mathscr{V}_{k_0})\right)[\mathfrak m_{x_0}]$$ 
and the surjectivity of \eqref{eqn_AISvsASbisbisbis} follows once we verify that
$${\rm im} \left(H^1(Y, \mathscr{V}_{k_0})\xrightarrow{\partial^1_Y} H^1(\partial \widetilde{Y}, \mathscr{V}_{k_0})\right)[\mathfrak m_{x_0}]=0\,.$$ 
This follows from the  properties\footnote{The required vanishing statement follows from \cite[Corollary 4.7]{Harder1975} (see also \cite{Harder1987}, Theorem 1) combined with the strong multiplicity one for $\GL2$, which follows as a consequence of the classification result due to Jacquet--Shalika \cite{JacquetShalika_1,JacquetShalika_2} (c.f. the proof of \cite{FrankeSchwermer}, Proposition 4.1).  In brief terms, Corollary 4.7 of \cite{Harder1975} tells us that the systems of Hecke eigenvalues occurring in the image of the map $\partial^1_Y$ are those of Eisenstein series. By the classification result of Jacquet--Shalika, the set of systems of Hecke eigenvalues of Eisenstein series is disjoint from those of cusp forms. (We are grateful to Fabian Januszewski for indicating these references.)} of the Eisenstein cohomology, since $x_0$ is cuspidal.

The proof that the map $\eqref{eqn_AISvsASbisbisbis}$ is injective is similar, where one instead relies on the vanishing of the $x_0$-isotypic subspace of ${\rm im} (H^0(\partial \widetilde{Y}, \mathscr{V}_{k_0})\rightarrow H^1_c(Y, \mathscr{V}_{k_0}))$\,. 

Our assertion on the dimension of these vector spaces is standard, c.f. \cite[Proposition 3.18]{bellaiche2012}.

\item[ii)] Fix $\nu \geqslant k_0+1.$ Let $\mathcal W$ be an affinoid neighborhood 
of $k_0$ such that the conclusion of  Theorem~\ref{thm_bellaiche} hold for $\mathcal W.$
To simplify notation, set $M_c=H^1_c(\Gamma_p, \mathbb{D}^\dagger(\mathcal W))^{\leq \nu}\otimes_{\cH_{\mathcal W}}\cO_{\mathcal{X}}$ and 
$M =H^1(\Gamma_p, \mathbb{D}^\dagger(\mathcal W))^{\leq \nu}\otimes_{\cH_{\mathcal W}}\cO_{\mathcal{X}}.$
Then the $\cO_{\mathcal X}$-module $M_c$ is free of rank $2.$ The same argument (that Bella\"iche utilizes to prove $M_c$ is free of rank $2$) shows that $M$ is also free of rank $2$ over $\cO_{\mathcal X}.$ Namely,  on shrinking $\mathcal W$ if necessarly, we can assume that 
$M$ is a finitely generated free $O_{\mathcal X}$-module.
For any point $x$ of weight $w(x)$ we have  
\begin{equation}
\label{eqn: computation specialization of M}
M/
\mathfrak m_{x}M
\simeq 
\left (H^1 (\Gamma_p, \mathbb{D}^\dagger (\mathcal W))^{\leq \nu}\otimes_{O_{\mathcal W}, w(x)} E \right )\otimes_{\cH, x} E,
\end{equation}
Together with Lemma~\ref{lemma: properties of specialization} this gives 
\begin{equation}
\label{eqn: computation specialization of M-2}
M/\mathfrak m_{x}M \simeq H^1(\Gamma_p, \DD^\dagger_{w(x)}(E))^{\leq \nu}\otimes_{\cH,x}E.
\end{equation}
The control theorem of Ash--Stevens 
\eqref{eqn_AshStevens2008Unpublished_1} reads
\begin{equation}
\nonumber
H^1(\Gamma_p, \mathbb{D}^\dagger_k(E))^{\leq \nu}\xrightarrow{\,\,\sim\,\,} H^1(\Gamma_p, \mathscr{V}_k(E))^{\leq \nu}
\end{equation}
for each positive integer $k>\nu-1$. Choosing $x\in \mathcal X$
such that $w(x)>\nu-1,$ we deduce that
\[
M/\mathfrak m_{x}M \simeq H^1(\Gamma_p, \mathscr{V}_k(E))^{\leq \nu}\otimes_{\cH,x}E.
\]
Now it follows from the classical Eichler--Shimura isomorphism and multiplicity-one that $M/\mathfrak m_{x}M$ has dimension $2$ over $E$, and  the $\cO_\mathcal{X}$-module $M$ is free of rank $2$ as well. 

Let $\det_{\frak j}\in \cO_\mathcal{X}$ denote the determinant of $\frak j \otimes {\rm id}\,:\, M_c \rightarrow M$ with respect to fixed bases of 
$M_c$ and $M.$ If we knew that $\det_{\frak j}(x_0)\neq 0$, we could shrink  
$\mathcal W$ and the neighborhood $\mathcal X$ of $x_0$) to ensure that $\det_{\frak j}$ is non-vanishing on $\mathcal{X}$ and thereby conclude that $\frak{j}\otimes {\rm id}$ is an isomorphism, as required. We have therefore reduced to proving that $\det_{\frak j}(x_0)\neq 0.$

Recall the isomorphism \eqref{eqn: computation specialization of M-2}  for 
$x=x_0$:
\begin{equation}
M/\mathfrak m_{x_0}M \simeq H^1(\Gamma_p, \DD^\dagger_{k_0}(E))^{\leq \nu}\otimes_{\cH,x_0}E.
\end{equation}
The analogous isomorphism \eqref{eqn: computation specialization of M} for $M_c$ together with Lemma~\ref{lemma: properties of specialization} shows that we have an injection
\begin{equation}
\nonumber
M_c/\mathfrak m_{x_0}M_c \hookrightarrow H^1_c(\Gamma_p, \DD^\dagger_{k_0}(E))^{\leq \nu}\otimes_{\cH,x_0}E,
\end{equation}
which is an isomorphism if $k_0\neq 0.$ 
As a matter of fact, it follows from \cite[Proposition~3.14]{bellaiche2012} that  this map is an isomorphism since $x_0$ is cuspidal. We therefore have the following commutative diagram:
\begin{equation}
\nonumber
\xymatrix{
M_c/\mathfrak m_{x_0}M_c \ar[r]^-{\sim}  \ar[d] &H^1_c(\Gamma_p, \DD^\dagger_{k_0}(E))^{\leq \nu}\otimes_{\cH,x_0}E \ar[d]\\
M/\mathfrak m_{x_0}M  \ar[r]^-{\sim} &H^1(\Gamma_p, \DD^\dagger_{k_0}(E))^{\leq \nu}\otimes_{\cH,x_0}E.
}
\end{equation}
To conclude with the proof of the asserted isomorphism in our proposition with sufficiently small $\mathcal{W}$ and $\mathcal{X}$, it suffices to prove that the right vertical map is an isomorphism. Observe that the following diagram commutes:
\begin{equation}\label{eqn_last_eqn_hopefully}
\begin{aligned}\xymatrix{
H^1_c(\Gamma_p, \DD^\dagger_{k_0}(E))^{\leq \nu}\otimes_{\cH,x_0}E
\ar[r] \ar[d]_{\rho_{k_0}^*}^{\sim}
& H^1(\Gamma_p, \DD^\dagger_{k_0}(E))^{\leq \nu}\otimes_{\cH,x_0}E \ar[d]^{\rho_{k_0}^*}_{\sim}\\
H^1_c(\Gamma_p, \mathscr{V}_{k_0}(E))^{\leq \nu}\otimes_{\cH,x_0}E \ar[r]^{\sim}_{{\rm (i)}}& H^1(\Gamma_p, \mathscr{V}_{k_0}(E))^{\leq \nu}\otimes_{\cH,x_0}E
} 
\end{aligned}
\end{equation}
All objects of this diagram have dimension $2$ over $E.$ 
The surjectivity of the vertical arrow in \eqref{eqn_last_eqn_hopefully} on the left follows from \cite[Lemma 5.1]{PollackStevensJLMS}, and
we infer that the vertical arrow on the left is an isomorphism.  
The bottom map is an isomorphism by part i). This implies that all the maps of the diagram are isomorphisms.  The proof of our proposition is now complete.

\end{proof}


\section{Interpolation of Beilinson--Kato elements}
\label{sec_interpolateBK_elements}

\subsection{Modular curves and Iwasawa sheaves }
\label{subsec_mod_curves_Iwasawa_sheaves}
\subsubsection{} For a pair of positive integers $M_1$ and $M_2$, we recall the modular curves $Y(M_1,M_2)$ from \cite[\S2]{kato04}. We then have $Y(N,N)=Y(N)$ and $Y(1,N)=Y_1(N)$, modular curves of level $\Gamma(N)$ and $\Gamma_1(N)$, respectively. Recall from \S2.8 of op. cit. also the modular curves denoted by $Y(M_1,M_2(B))$ and $Y(M_1(B),M_2)$, where $B$ is also a positive integer. 
In Sections~
\ref{Overconvergent etale sheaves}-
\ref{subsec_BK_families}, we will be working over the modular curve $Y(1,N(p))$ of $\Gamma_1(N)\cap \Gamma_0(p)$-level. We note that the modular curve $Y(1,N(p))$ is denoted by $Y(N,p)$ in \cite{andreattaiovitastevens2015}.

If $Y$ denotes any one of the modular curves above, we denote by  $\lambda:\mathscr{E}\to Y$ the universal elliptic curve with the appropriate level structure (which depends on $Y$, but we suppress this dependence from our notation).
We let $\umT :=R^1{\lambda_*}\ZZ_p(1)$ denote the 
pro-system $\left(\mathscr{T}_{n}\right)_{n\geq 1}$ 
of \'etale lisse sheaves  $\mathscr{T}_{n}:=R^1 \lambda _*\bbmu_{p^n}$ 
on the open modular curves $Y[1/p]$ given as in \cite[\S\S1--2]{kato04}. 
The sheaf $\mathscr{T}$ has rank $2$ and the Poincar\'e duality 
identifies it  with the $p$-adic Tate module $\underline{T}_p (\mathscr E)$ of
$\mathscr{E}$. We write $\mathscr{T}_{\QQ_p}$ for the associated sheaf of $\QQ_p$-vector spaces.

For each non-negative integer $k$ and a locally free sheaf $\mathscr{F}$ over $Y[1/p]$, we let ${\rm TSym}^k\,\mathscr{F}$ denote the locally free sheaf on $Y$ of symmetric $k$-tensors, given as in \cite[\S2.2]{KLZ2}.

\subsubsection{}
Let  $i_D\,:\,D \hookrightarrow \mathscr E$ be a subscheme. Assume that $D$
is etale over $Y$ and consider the diagram
\[
\xymatrix{
\mathscr E [p^n] \left <D\right > \ar[r] \ar[d]^{{p}_{D,n}}
& \mathscr E \ar[d]^{[p^n]}\\
D \ar[r]^{i_D} \ar[d]^{\lambda_D} & \mathscr E \ar[dl]^{\lambda}\\
Y &
}
\]
where $\mathscr E [p^n] \left <D\right >$ is the fiber product of $\mathscr E $
and $D$ over $\mathscr E.$ We define the pro-etale sheaf $\uLambda (\umT \left <D\right >)= (\uLambda_n (\umT \left <D\right >))_{n\geqslant 1}$ on setting
\[
\uLambda_n (\umT \left <D\right >)= \lambda_{D,*} p_{D,n,*} (\ZZ/p^n\ZZ).
\]
We refer the reader to \cite{Kings2015} where Kings develop these notions in a general framework.

If $D=Y$ and $i_D=s \,:\, Y \rightarrow \mathscr E$ is a section of 
$\lambda,$ we will write $\uLambda (\umT \left <s \right >)$ instead
$\uLambda (\umT \left <D\right >).$ In particular, we denote by 
$\uLambda (\umT)=\uLambda (\umT \left <0\right >)$ the sheaf associated to the identity section $0$. 
For any section $s,$ the sheaf $\uLambda (\mathscr{T}\langle s \rangle)$  
is a sheaf of modules  of   rank one over the sheaf of Iwasawa algebras $\uLambda (\mathscr{T})$ (c.f. \cite{Kings2015}, \S2.4).

\subsubsection{}  In the remainder of \S\ref{subsec_mod_curves_Iwasawa_sheaves}, we recall a number of notation and constructions from \cite{KLZ2,LZ1}. 

Fix a positive integer  $N.$  We denote by  $\lambda_{N}\,:\, \mathscr E_{N} \rightarrow Y_1(N)$ the universal elliptic curve
and by $\umT_N$  the associated etale sheaf.
Recall that  $Y_1(N)$ is the moduli space  of pairs $(E,\beta_{N})$, where $E$ is an elliptic curve and $\beta_{N}\,:\, \ZZ/N\ZZ \rightarrow E[N]$ is an injection. Then the map $(E,\beta_{N}) \mapsto \beta_{N} (1)$ defines a section $s_{N} \,:\, Y_1(N) \rightarrow \mathscr E_{N}$
of $\lambda_{N},$ and we denote by $\uLambda (\umT_N \left <s_N\right >)$ 
the associated sheaf. 

Let $g_n\,:\, Y_1(Np^n) \rightarrow Y_1(N)$ denote the canonical projection. 
We define the pro-etale sheaf $\uLambda_N$ on $Y_1(N)_{\textup{\'et}}$
on setting 
\[
\begin{aligned}
&\uLambda_N= \left (\uLambda_{N,n}\right )_{n\geqslant 1},
&& \uLambda_{N,n}=g_{n,*}(\ZZ/p^n\ZZ)\,.
\end{aligned}
\]
If $p \mid N,$ the moduli description of $Y_1(Np^n)$
shows that there exists a canonical isomorphism $Y_1(Np^n) \simeq \mathscr E[p^n] \left <s_N \right >,$ and therefore we have
\begin{equation}
\label{eqn: isomorphism of Lambda sheaves}
\uLambda (\umT_N \left <s_N\right >) \simeq \uLambda_N\,.
\end{equation}
(c.f. the proof of \cite{KLZ2}, Theorem 4.5.1). In particular, we can apply the formalism developed in \cite{Kings2015} with the sheaf $\uLambda_N$.

\subsubsection{}
Assume that $N$ is coprime to $p$ and consider the modular curve $Y(1,N(p))$ equipped with the universal elliptic curve $\mathscr{E}_{N(p)}$ and the  associated sheaf $\mathscr{T}_{N(p)}.$ Recall that $Y(1,N(p))$ is the moduli space for the triples $(E, \beta_N, C),$ where $E$ is an elliptic curve,  $\beta_N\,:\, \ZZ/N\ZZ \rightarrow E$ is an injection and $C\subset E$ is a (cyclic) subgroup of order $Np$ { that contains $\beta_N(1+N\ZZ)$}. Denote by $\mathscr C \subset \mathscr{E}_{N(p)}[p]$ the canonical subgroup of order $p$ and set $D:=\mathscr{E}_{N (p)}[p]-\mathscr C$ and $D^\prime:=\mathscr C-\{0\}$. Note that both $D$ and $D^\prime$ are finite \'etale over $Y(1,N(p))$, of degrees $p^2-p$ and $p-1$, respectively, and we denote by  $\uLambda (\mathscr{T}_{N(p)}\langle D \rangle)$ and $\uLambda (\mathscr{T}_{N(p)}\langle D' \rangle)$ the associated sheaves.

Since both $D$ and $D^\prime$ are contained in $\mathscr{E}_{N(p)}[p]$,  the ``multiplication-by-$p$'' morphism induces the trace map
\begin{equation}
\label{eqn_mult_by_p_map_LZ}
\uLambda (\mathscr{T}_{N(p)}\langle D^?\rangle)\xrightarrow{[p]_*}\uLambda (\mathscr{T}_{N(p)}),
\qquad D^?\in \{D,D'\}
\end{equation}
of sheaves on $Y(1,N(p))$.

The rule $(E, \beta_{Np}) \mapsto (E,\beta_{Np}(p), \mathrm{im}(\beta_{Np}))$
defines a morphism ${\pr}^\prime: Y(1,Np)\to Y(1,N(p))$, which in turn induces the cartesian square
$$\xymatrix{\mathscr{E}_{Np}\ar[d] \ar[r] &Y_1(Np)\ar[d]^{\pr^\prime}\\
\mathscr{E}_{N(p)} \ar[r] &Y(1,N(p))
}$$ 
For each positive integer $r$, we have a natural morphism
\begin{equation}
\label{eqn_BenoisHorte_70_1}
\mathscr{E}_{Np}[p^r]\langle s_{Np}\rangle \xrightarrow{[N]_*\circ\, {\pr^\prime}} \mathscr{E}_{N(p)}[p^r]\langle D^\prime\rangle\,.
\end{equation}	
where the schemes $\mathscr{E}_{?}[p^r]\langle\, \cdot \,\rangle$ are given as in \cite[Definition 4.1.4]{KLZ2}. 
The composition of this map with \eqref{eqn: isomorphism of Lambda sheaves}
gives a map
\begin{equation}
\label{eqn_BenoisHorte_70}
H^*_{\textup{\'et}}(Y_1(Np),\underline{\LL}_{Np})\lra H^*_{\textup{\'et}}(Y(1,N(p)),\underline{\LL}(\mathscr{T}_{N(p)}\langle D^\prime\rangle))\,,
\end{equation}
(see also  \cite[\S4.2.5]{BenoisHorte2020} for further details).


\subsection{Big Beilinson--Kato elements}
\label{subsec:Beilinson--Kato}
\subsubsection{}
Let us fix positive integers $M$ and $N$ such that $M+N\geqslant 5.$ We consider the following objects:
\begin{itemize}
\item[$\bullet$]{} The modular curves $Y(M,N)$ and $Y_1(N)$, which come equipped with the universal elliptic curves  $\lambda_{M,N}\,:\,\mathscr E_{M,N} \rightarrow Y(M,N)$ and  $\lambda_{N}\,:\,\mathscr E_{N} \rightarrow Y_1(N)$, as well as the sheaves $\umT_{M,N}=R^1\lambda_{M,N,*} \Zp(1)$ and $\umT_N=R^1\lambda_{N,*} \Zp(1)$ on $Y (M,N)_{\mathrm{\acute et}}$
and $Y_1 (N)_{\mathrm{\acute et}}$\,, respectively.

\item[$\bullet$]{} The morphisms $Y(M,N) \rightarrow Y_1(N) \times \mu_M^{\circ} \rightarrow Y_1(N),$ where $\mu_M^{\circ}=\mathrm{Spec}(\ZZ [\zeta_{M}]).$ These morphisms induce a canonical map $\mathscr T_{M,N} \rightarrow \mathscr T \vert_{Y(M,N)}.$

\item[$\bullet$]{} The canonical projections 
\begin{equation}
\nonumber 
\begin{aligned}
&f_n\,:\,Y(Mp^n,Np^n)
\rightarrow Y(M,N),    && h_n \,:\,\mu_{Mp^n}^{\circ}  \rightarrow \mu_{M}^{\circ} 
\\
&g_n\,:\, Y_1(Np^n)  \rightarrow Y_1(N), &&g^{\cyc}_n\,:\, Y_1(Np^n) \times \mu_{Mp^n}^{\circ} \rightarrow Y_1(N) \times \mu_{M}^{\circ} 
\end{aligned}
\end{equation}
which give rise to the inverse systems of sheaves
\begin{equation}
\nonumber
\begin{aligned}
&\uLambda_{M,N}=(\uLambda_{M,N,n})_{n\geqslant 1}, && 
\uLambda_{M,N,n}=f_{n,*}(\ZZ/p^n\ZZ),\\
&\uLambda_{\mu_M^\circ}^{\mathrm{cyc}}=(\uLambda_{\mu_M^\circ,n}^{\mathrm{cyc}})_{n\geqslant 1}, &&
\uLambda_{\mu_M^\circ,n}^{\mathrm{cyc}}=h_{n,*} (\ZZ/p^n\ZZ),\\
&\uLambda_N=(\uLambda_{N,n})_{n\geqslant 1}, 
&&\uLambda_{N,n}=g_{n,*} (\ZZ/p^n \ZZ), \\
&\uLambda_{N,\mu_M^\circ}^{\cyc}=(\uLambda_{N,\mu_M^\circ,n}^{\cyc})_{n\geqslant 1}, 
&&
\uLambda_{N,\mu_M^\circ,n}^{\cyc}=g^{\cyc}_{n,*} (\ZZ/p^n \ZZ).
\end{aligned}
\end{equation}
\end{itemize}
By the flatness of our sheaves,  it follows that 
\begin{equation}
\uLambda^{\cyc}_{N,\mu_M^\circ} \simeq 
\uLambda_N \boxtimes \uLambda^{\cyc}_{\mu_M^\circ}\,.
\end{equation}

For any integer $L$,  let us denote by $\text{prime} (L)$ the set of primes dividing $L.$ For any integer $m\geqslant 1,$ set 
$\Gamma_{\QQ (\zeta_m)}=\Gal (\QQ (\zeta_{mp^\infty})/\QQ (\zeta_m))$
and $\Lambda (\Gamma_{\QQ (\zeta_m)})=\Zp [[\Gamma_{\QQ (\zeta_m)}]].$

\subsubsection{}
{ Poincar\'e duality} gives a canonical isomorphism between  $\umT_{M,N}$ and the Tate module  $\underline T_p(\mathcal E_{M,N})$ of the universal elliptic curve $\mathcal E_{M,N}.$ The interpretation 
of $Y(Mp^n,Np^n)$ as the moduli space of pairs $(E, \alpha ),$
where $E$ is an elliptic curve and  $\alpha  \,:\, (\ZZ/Np^n\ZZ )^2 \simeq  E[Np^n]$ 
is an isomorphism, shows that the restriction of 
$\underline{\mathcal E}_{M,N}[p^n]\simeq \underline T_p(\mathcal E_{M,N})/p^n$
on  $Y(Mp^n,Np^n)$ has a canonical basis $\{e_{1,n}, e_{2,n}\}$ given by 
$e_{1,n}=\alpha (M,0)$ and $e_{2,n}=\alpha (0,N).$ Analogously, the interpretation 
of $Y(Mp^n,Np^n)$ as the moduli space of pairs $(E, \beta ),$
where $E$ is an elliptic curve and  $\beta \,:\, \ZZ/Np^n\ZZ  \rightarrow  E[Np^n]$ 
is an injection, shows that the restriction  of the sheaf $\umT_N$ on $Y_1(Np^n)$
has the  canonical section $\beta (N)$ which we also denote by $e_{2,n}$ to simplify notation.

\subsubsection{}

One has
an isomorphism of continuous  Galois modules
\begin{equation}
\nonumber
\label{eqn: cohomology of Lambda_N^cyc}
H^1_{\textup{\'et}}\left ((Y_1(N)\times \mu_M^\circ \right)_{\overline \QQ}, 
\uLambda_{N,\mu_M^\circ}^{\cyc} ) 
\simeq H^1_{\textup{\'et}} (Y_1(N)_{\overline \QQ},\uLambda_{N, \mu_M^\circ}^{\cyc} )
\otimes \Zp [[\Gal (\QQ (\zeta_{M})/\QQ) ]].
\end{equation}
Since $\uLambda_{\mu_M^\circ}^{\cyc}$ is  a constant sheaf on $Y_1(N)_{\overline \QQ},$ one also has
\[
H^1_{\textup{\'et}} (Y_1(N)_{\overline \QQ}, \uLambda_{N,\mu_M^\circ}^{\cyc} ) 
\simeq H^1_{\textup{\'et}} (Y_1(N)_{\overline \QQ}, \uLambda_N)
\widehat\otimes \Lambda (\Gamma_{\QQ (\zeta_M)}).
\]
The commutative diagrams (for each $n$)
\[
\xymatrix{
Y(Mp^n,Np^n) \ar[d]^{f_n} \ar[r] &Y_1(Np^n) \times \mu^\circ_{Mp^n} 
\ar[d]^{g^\cyc_n}\\
Y(M,N) \ar[r] &Y_1(N)\times \mu^\circ_{M} \ar[d]\\
& Y_1(N) \times \mu_{m}^\circ
} 
\] 
show that we have the trace map  
\begin{align}
\label{eqn:trace for modular curves}
\begin{aligned}
H^1_{\textrm{\textup{\'et}}}(Y_1(M,N)_{\overline \QQ}, \mathrm{TSym}^j(\umT_{M,N}) \otimes\uLambda_{M,N}) \lra 
H^1_{\textrm{\textup{\'et}}} \left ((Y_1(N)\times \mu_M^\circ)_{\overline\QQ}, \mathrm{TSym}^j(\umT_{N}) \otimes\uLambda^{\cyc}_{N, \mu_M^\circ}\right )\\
\lra H^1_{\textrm{\textup{\'et}}} \left ((Y_1(N)\times \mu_m^\circ)_{\overline\QQ}, \mathrm{TSym}^j(\umT_{N}) \otimes \uLambda^{\cyc}_{N, \mu_M^\circ}\right )\\
{}^{}\qquad \xrightarrow{\sim} H^1_{\textrm{\textup{\'et}}} \left (Y_1(N)_{\overline\QQ}, \mathrm{TSym}^j(\umT_{N}) \otimes \uLambda^{\cyc}_{N,\mu_M^\circ}\right )\otimes \Zp[\Gal (\QQ (\zeta_m)/\QQ].
\end{aligned}
\end{align}
for each $m$ dividing $M$.

\subsubsection{}
More generally, let  $(a,B,m)$ be a triple of integers such that $mB \mid M.$ Let us choose $L\geqslant 1$ such that 
$M\mid L,$ $N\mid L$ and 
\[
\text{prime}(L)=
\text{prime} (MN).
\]
The construction of \cite[\S5.2]{kato04} provides us with a map 
\begin{multline}
\label{eqn: generalized trace for modular curves}
\begin{aligned}
t_{m,a(B)} \,:\, 
H^1_{\textrm{\'et}}(Y_1(M,L)_{\overline \QQ}, \mathrm{TSym}^j(\umT_{M,L}) \otimes\uLambda_{M,L}) \lra 
H^1_{\textrm{\'et}} \left ((Y_1(N)\times \mu_m^\circ )_{\overline\QQ}, \mathrm{TSym}^j(\umT_{N}) \otimes \uLambda^{\cyc}_{N, \mu_M^\circ}\right )\\
\simeq 
H^1_{\textrm{\'et}} \left (Y_1(N)_{\overline\QQ}, \mathrm{TSym}^j(\umT_{N}) \otimes\uLambda^{\cyc}_{N,\mu_M^\circ}\right )\otimes \Zp[\Gal (\QQ (\zeta_m)/\QQ].
\end{aligned}
\end{multline}
Note that this map coincides with \eqref{eqn:trace for modular curves} if
$m=M,$ $a=0,$ $B=1$ and $L=N.$

\subsubsection{}  We write 
\begin{equation}
\nonumber
\partial_n \, : \, \cO(Y(Mp^n,Np^n))^\times 
\lra \varprojlim_r H^1_{\textup{\'et}}(Y(Mp^n,Np^n),{\bbmu}_{p^r})=H^1_{\textup{\'et}}(Y(Mp^n,Np^n),\ZZ_p(1))
\end{equation} 
for the Kummer map. 
For any integer $j\geqslant 0,$ consider the chain of maps 

\begin{align}
\label{eqn: map from K_2 to etale}
\begin{aligned}
\mathrm{Ch}_{M,N,j} \,:\,\varprojlim_n  K_2(Y(Mp^n,Np^n)) &\lra 
\varprojlim_n H^2_{\textup{\'et}}(Y(Mp^n,Np^n),\ZZ_p(2))\\ 
&\xrightarrow{\cup e_{1,n}^{\otimes j}}
\varprojlim_n H^2_{\textup{\'et}}(Y(Mp^n,Np^n), f_n^*(\mathrm{TSym}^{j} (\umT_{M,N,n}))\otimes \bbmu_{p^n}^{\otimes 2}) \\
&\stackrel{\sim}{\lra}
 H^2_{\textup{\'et}}(Y(M,N), \mathrm{TSym}^{j} (\umT_{M,N})\otimes 
 \uLambda_{M,N} (2)) \\
& \lra H^1 \left (\ZZ [1/MNp], H^1_{\textup{\'et}}(Y(M,N)_{\overline \QQ}, \mathrm{TSym}^{j} (\umT_{M,N})\otimes \uLambda_{M,N} (2))\right ),
\end{aligned}
\end{align}
where the first map is  induced by the composition of $\partial_n$ with the cup product in \'etale cohomology, and the very last map is deduced from the Hochschild--Serre spectral sequence.  

Let ${}_cg_{1/Mp^n,0}\in \cO (Y(Mp^n,1)^\times$ and ${}_dg_{0,1/Np^n}\in \cO(Y(1,Np^n)^\times$ 
denote Kato's Siegel units (see \cite[\S\S1--2]{kato04} for their definition). As in op. cit., we put  
\[
{}_{c,d}z_{Mp^n,Np^n}:={}_cg_{1/Mp^n,0} \cup {}_dg_{0,1/Np^n}\in
K_2(Y(Mp^n,Np^n)).
\]
We recall that
\[
({}_{c,d}z_{Mp^n,Np^n})_{n\geqslant 1} \in \varprojlim_n K_2(Y(Mp^n,Np^n))
\]
belongs to the source of the map $\mathrm{Ch}_{M,N,j}$ (c.f. \cite{kato04}, Proposition~2.3).

\subsubsection{} Let us fix a positive integer $N$. As in \cite[\S5]{kato04}, let $\xi$ denote either the symbol $a(B)$ with $a,B\in\ZZ$ and $B\geqslant 1$ or an element of $\mathrm{SL}_2(\ZZ)$.  For each integer $m\geqslant 1,$ 
we denote by $S$ the following set of primes: 
\begin{equation}
\nonumber
S=\begin{cases} \text{primes} (mBp),
& \textrm{if $\xi=a(B)$}\\
\text{primes} (mNp),
& \textrm{if $\xi\in \mathrm{SL}_2(\ZZ)$}.
\end{cases}
\end{equation} 
Let $(c,d)$ be a pair of positive integers satisfying the following conditions:
\[
\begin{aligned}
\nonumber
& (cd,6)=1, && (d,N)=1, 
&&&\text{prime} (cd) \cap S=\emptyset .
\end{aligned}
\]

If $\xi= a(B),$ we choose $M$ and $L\geqslant 1$ such that 
\begin{equation}
\label{eqn:conditions on M,N}
\begin{aligned}
\nonumber
&mB \mid M, 
& M \mid L,
&\qquad N\mid L,\\
&\text{prime}(M)=S,
&&\text{prime}(L)=S\cup \text{prime} (N)
\end{aligned}
\end{equation}
and denote by 
\begin{equation}
\label{eqn: the map ch first case}
\begin{aligned}
\mathrm{Ch}_{M,L,j,\xi}\,:\,  \varprojlim_n  & \,K_2(Y(Mp^n,Lp^n)) \lra H^1\left (\ZZ[1/S], H^1_{\textup{\'et}}  (Y_1(M,L)_{\overline \QQ}, \mathrm{TSym}^j(\umT_{M,L})\otimes \uLambda_{M,L}(2))
\right )
\\
&\quad\lra H^1\left (\ZZ[1/S],
H^1_{\textrm{\'et}} \left (Y_1(N)_{\overline\QQ}, \mathrm{TSym}^j(\umT_{N}) \otimes\uLambda_{N,\mu_m^\circ}^\cyc  (2)\right )\otimes \Zp[\Gal (\QQ (\zeta_m)/\QQ] \right )\\
&\qquad\qquad\stackrel{\sim}{\lra} H^1\left (\ZZ [{1}/{S}, \zeta_m ],
 H^1_{\textup{\'et}} \left (Y_1(N)_{\overline \QQ}, \mathrm{TSym}^{j} (\umT_{N})
\otimes
\uLambda_{N,\mu_m^\circ}^\cyc  (2) \right )
\right )
\\
&\qquad\qquad\qquad\stackrel{\sim}{\lra} H^1\left (\ZZ [{1}/{S}, \zeta_m ],
 H^1_{\textup{\'et}} \left (Y_1(N)_{\overline \QQ}, \mathrm{TSym}^{j} (\umT_{N})
\otimes
\uLambda_N (2) \right )
\widehat \otimes \Lambda (\Gamma_{\QQ (\zeta_m)})^\iota
  \right )
\end{aligned}
\end{equation}
the composition of the maps \eqref{eqn: map from K_2 to etale} and 
\eqref{eqn: generalized trace for modular curves}. Here, the symbol $\iota$ on the final line means that the target cohomology group is equipped 
with a continuous action of 
$\Gamma_{\QQ (\zeta_m)}=\Gal (\QQ (\zeta_{mp^\infty})/\QQ (\zeta_m))$
induced by the right action of this group on $\Lambda (\Gamma_{\QQ (\zeta_m)})$ via the canonical involution $\iota (g)=g^{-1}$
(c.f. \S\ref{subsubsection:notation phi-Gamma modules}).

If $\xi \in \mathrm{SL}_2(\ZZ),$ we fix $L\geqslant 3$ such that 
\begin{equation}
\nonumber
\begin{aligned}
&m\mid L,
&& N\mid L,
&&&\text{prime} (L)=S.
\end{aligned}
\end{equation}
The element $\xi$ induces  an automorphism of $Y(Lp^n)=Y(Lp^n,Lp^n).$ We 
consider the map

\begin{equation}
\label{eqn: the map ch second case}
\begin{aligned}
\mathrm{Ch}_{L,j,\xi}\,:\,\varprojlim_n &\, K_2(Y(Lp^n))
\lra H^1\left (\ZZ[1/S], H^1_{\textup{\'et}}  (Y(L)_{\overline \QQ}, \mathrm{TSym}^j(\umT_{L,L})\otimes \uLambda_{L,L}(2))\right )\\
&\qquad\lra 
H^1\left (\ZZ[1/S], H^1_{\textrm{\textup{\'et}}} \left (Y_1(N)_{\overline\QQ}, \mathrm{TSym}^j(\umT_{N}) \otimes\uLambda_{N,\mu_m^\circ}^\cyc  (2)\right )\otimes \Zp[\Gal (\QQ (\zeta_m)/\QQ] \right ) 
\\
&\qquad\qquad\lra
H^1\left (\ZZ[1/S, \zeta_m], H^1_{\textrm{\textup{\'et}}} \left (Y_1(N)_{\overline\QQ}, \mathrm{TSym}^j(\umT_{N}) \otimes\uLambda_{N,\mu_m^\circ}^\cyc (2)\right )\right )
\\
&\qquad\qquad\qquad\stackrel{\sim}{\lra}
H^1\left (\ZZ [{1}/{S}, \zeta_m ],
 H^1_{\textup{\'et}} \left (Y_1(N)_{\overline \QQ}, \mathrm{TSym}^{j} (\umT_{N})
\otimes
\uLambda_N (2) \right )
\widehat \otimes \Lambda (\Gamma_{\QQ (\zeta_m)})^\iota
  \right )\,.
\end{aligned}
\end{equation}

\begin{defn} 
\label{defn:Kato big-zeta}
\item[i)] If $\xi=a(B)$   we set
\[
{}_{c,d}\mathbb{BK}_{N,m}(j,\xi)= \mathrm{Ch}_{M,L,j,\xi}
\left ( ({}_{c,d}z_{Mp^n,Lp^n})_{n\geqslant 1}\right ),
\]
where $\mathrm{Ch}_{M,L,j,\xi}$ is the map \eqref{eqn: the map ch first case}.

\item[ii)] If $\xi \in \mathrm{SL}_2(\ZZ)$, we set 
\[
{}_{c,d}\mathbb{BK}_{N,m}(j,\xi)= \mathrm{Ch}_{L,j,\xi}
\left ( (\xi^*({}_{c,d}z_{Lp^n,Lp^n}))_{n\geqslant 1}\right ),
\]
where $\mathrm{Ch}_{L,j,\xi}$ is the map \eqref{eqn: the map ch second case}.
\end{defn}

\begin{lemma} The elements ${}_{c,d}\mathbb{BK}_{N,m}(j,\xi)$ do not depend
on the choice of $M$ and $L.$ 
\end{lemma}
\begin{proof}  The proof is similar to that of \cite[Proposition~8.7]{kato04}. Let $(M',L')$ be another pair of integers that satisfy \eqref{eqn:conditions on M,N} and such that $M\mid M',$ $N\mid N'.$ Since the trace map commutes with cup products, we have the following commutative diagram:
\[
\xymatrix{
\varprojlim_n  K_2(Y(M'p^n,L'p^n))
\ar[r] \ar[d]&H^1\left (\ZZ[1/S], H^1_{\textup{\'et}}  (Y_1(M',L')_{\overline \QQ}, \mathrm{TSym}^j(\umT_{M',L'})\otimes \uLambda_{M',L'}(2))
\right ) \ar[d]\\
\varprojlim_n  K_2(Y(Mp^n,Lp^n))
\ar[r] &H^1\left (\ZZ[1/S], H^1_{\textup{\'et}}  (Y_1(M,L)_{\overline \QQ}, \mathrm{TSym}^j(\umT_{M,L})\otimes \uLambda_{M,L}(2))
\right )
}
\]
On the other hand, the left vertical map sends $ ({}_{c,d}z_{M'p^n,L'p^n})_{n\geqslant 1}$ to $ ({}_{c,d}z_{Mp^n,Lp^n})_{n\geqslant 1}$ by \cite[Proposition~2.3]{kato04}. This proves the lemma in the case $\xi= a(B)$; the case $\xi \in \mathrm{SL}_2(\ZZ)$ is analogous. 
\end{proof}

\subsubsection{} 
\label{subsec:moments}
The isomorphism \eqref{eqn: isomorphism of Lambda sheaves}
allows us to consider  the moment maps
\begin{equation}
\label{eqn:definition moment map}
\mom_N^{[k]} \,:\, \uLambda_N \rightarrow \TSym^k (\umT_N)
\end{equation}
given as in \cite[Section~~2.5]{Kings2015} (see also \cite{KLZ2}, \S4.4).  
Let $\mom_{N,n}^{[k]}$ denote the map 
$\mom_N^{[k]}$ modulo $p^n.$ 
Then  the map $\mom_{N,n}^{[k]}$  can be described  as the composition
\[
g_{n,*}(\ZZ/p^n\ZZ) \xrightarrow{\cup e_{2,n}^{\otimes k}} g_{n,*}\left (\ZZ/p^n\ZZ \otimes g_{n}^*(\mathrm{TSym}^k(\umT_{N,n}) )\right )\simeq 
g_{n,*}\circ g_{n}^*(\mathrm{TSym}^k(\umT_{N,n}) ) 
\xrightarrow{\mathrm{trace}} \mathrm{TSym}^k(\umT_{N,n}).
\]
On the level of cohomology, one then obtains the commutative diagram
\begin{equation}
\nonumber
\xymatrix{
H^1_{\textup{\'et}}(Y_1(N)_{\overline \QQ}, \uLambda_{N,n }) \ar[r]^{\simeq}
\ar[d]^{\mom_{N,n}^{[k]}}
&
H^1_{\textup{\'et}}(Y_1(Np^n)_{\overline \QQ}, \ZZ/p^n)
\ar[d]^{\cup e_{2,n}^{\otimes k}}\\
H^1_{\textup{\'et}}(Y_1(N)_{\overline \QQ}, \mathrm{TSym}^k(\umT_{N,n}) )
& H^1_{\textup{\'et}}(Y_1(Np^n)_{\overline \QQ}, g_n^*(\mathrm{TSym}^k(\umT_{N,n})) ),
\ar[l]
} 
\end{equation}
where the bottom horizontal map is the trace map (c.f. \cite{Kings2015}, Proposition~2.6.8).

\subsubsection{}
\label{subsubsection:big Kato elements}
For each $r\in \ZZ,$ we have the moment maps of sheaves on $\mu_m^\circ$
\[
\mom^{[r,n]}_{\mu_m^\circ}\,:\,\uLambda_{\mu_m^\circ}^\cyc \lra 
h_{n,*}\circ h_n^*(\Zp (-r)), \qquad n\geqslant 0.
\]
The stalks of $\uLambda_{\mu_m^\circ}^\cyc$ are isomorphic to the $\Gamma$-module  $\Lambda (\Gamma_{\QQ (\zeta_m)})^\iota$  
at geometric points (c.f. \cite{KLZ2}, \S6.3) and the moment  map coincides
with the map
\[
\begin{aligned}
\Lambda (\Gamma_{\QQ (\zeta_m)}) &\lra \Zp [G_n] (-r),
\qquad\qquad G_n:=\Gal (\QQ (\zeta_{mp^n})/\QQ (\zeta_m)),
\\ g&\longmapsto \chi^{-r} (g)\bar{g}_n\otimes \chi^{-r}
\end{aligned} 
\]
where $\bar{g}_n\in G_n$ denotes the image of $g \in \Gamma_{\QQ (\zeta_m)}$
under the natural projection $\Gamma_{\QQ (\zeta_m)}\rightarrow G_n.$ 
Let us define
\[
\mom_{N, \mu_m^\circ}^{[k,r,n]} := \mom_N^{[k]} \boxtimes \mom^{[r,n]}_{\mu_m^\circ}
\,:\, \uLambda_{N,\mu_m^\circ}^\cyc \lra 
\mathrm{TSym}^k (\umT_N) (-r).
\]
Recall the elements ${}_{c,d} z^{(p)}_{1,N,m}(k,r,r',\xi,S)$ which Kato has constructed in \cite[\S8.9]{kato04}. Set
\begin{align*}
{}_{c,d}\mathrm{BK}_{N,m}(k,j,r,\xi)&:={}_{c,d} z^{(p)}_{1,N,m}(j+k+2, r,j+1,\xi,S)\\
&\qquad\qquad{\in H^1 \left (\ZZ [1/S, \zeta_{mp^n}], H^1_{\textup{\'et}} \left (Y_1(N)_{\overline \QQ}, \mathrm{Sym}^{j+k} (\umT_{N}) (2-r) \right )\right )}.
\end{align*}

Note that in his construction Kato makes use of the dual sheaf $\umT_{N}^{\vee}$ together with the canonical isomorphism $\umT_{N}^{\vee}\simeq \umT_{N}(-1)$ provided by duality.

\begin{proposition} 
\label{prop:big Kato elements}
For all integers $j,k \geqslant 0$ and $r\in \ZZ$, the induced map
\begin{align}
\nonumber
\mom_{N,j, \mu_m^\circ}^{[k,r,n]} \,:\,
H^1 &\left (\ZZ [1/S, \zeta_m], H^1_{\textup{\'et}} \left (Y_1(N)_{\overline \QQ}, \mathrm{TSym}^{j} (\umT_{N})
\otimes
\uLambda_{N,\mu_m^\circ}^\cyc  (2) \right )
\right ) 
\\
\nonumber
&\qquad\qquad\qquad\lra 
H^1 \left (\ZZ [1/S, \zeta_{mp^n}], H^1_{\textup{\'et}} \left (Y_1(N)_{\overline \QQ}, \mathrm{TSym}^{j+k} (\umT_{N}) (2-r) \right )
\right )
\end{align}
sends ${}_{c,d}\mathbb{BK}_{N,m}(j,\xi)$ to  ${}_{c,d}\mathrm{BK}_{N,mp^n}(k,j,r,\xi).$
\end{proposition}

\begin{proof}
This is clear by the construction.
\end{proof}

\subsection{Overconvergent \'etale sheaves}
\label{Overconvergent etale sheaves}
Until the end of this article, we  set $Y=Y(1,N(p))$ unless we state otherwise  and assume that $N$ is coprime to $p$. 
\subsubsection{} 
We denote by $\cF$ the canonical sheaf on $Y.$
Throughout  this section, $\mathcal W$ and $U$ denote an affinoid 
and a wide open disk such that 
\[
\mathcal W \subset U \subset W^*.
\] 
In what follows, we will allow ourselves to shrink both $U$ and $\mathcal W$ as necessary. We adopt  the notation and conventions of Section~\ref{sec_bellaiche_eigencurve}. In particular, $O_U^\circ$ denotes 
the ring of analytic functions on $U$ that are bounded by $1.$
We let  
\[
\chi_U:\ZZ_p^\times\lra O_U^{\circ, \times}
\]
 denote the composition of the cyclotomic character $\chi \,:\,G_{\QQ}\rightarrow \Zp^\times$ with the canonical weight character $\kappa_U\,:\, \ZZ_p^\times \hookrightarrow O_U^{\circ, \times}.$

We review the theory of overconvergent sheaves introduced in \cite{andreattaiovitastevens2015}. See also \cite[\S4]{LZ1} 
and \cite[\S4]{BSV2020} for further details concerning the material in this subsection.

\begin{defn}
\label{defn_DH_2}
\item[i)] We let $\mathscr{D}^\circ_U(\mathscr{T})$ (resp., $\mathscr{D}^\circ_U(\mathscr{T}^\prime)$, resp., $\mathscr{A}^0_{U}({\mathscr{T}^\prime})$ ) denote the pro-\'etale sheaf of $\LLU$-modules on $Y$, whose pullback to the pro-scheme $Y(p^\infty,Np^\infty)$ is the constant pro-sheaf  $\DD (U)^\circ$ (resp., $\DD^\prime (U)^\circ$, resp., ${\mathbb A}^\prime (U)^0$ ). 

\item[ii)] 
We set 
\begin{align*}
M_{U}^\circ({\mathscr{T}^\prime)}:=H^1_{\textup{\'et}}\left(Y_{\overline{\QQ}},\mathscr{D}^\circ_{U}({\mathscr{T}^\prime})\right)(1),\quad
M_{U}^\circ(\mathscr{T}):=H^1_{\textup{\'et}}\left(Y_{\overline{\QQ}},\mathscr{D}^\circ_{U}(\mathscr{T})\right)(\chi_U^{-1}),\quad
N_U^\circ (\mathscr{T}^\prime) :=H^1_{\textup{\'et},c}\left(Y_{\overline{\QQ}},\mathscr{A}_{U}^0({\mathscr{T}^\prime})\right)
\end{align*}
and 
\begin{equation}
\nonumber
\label{defn_nu}
M_{U}({\mathscr{T}^\prime}):=M_{U}^\circ({\mathscr{T}^\prime})[1/p]\,,
\qquad 
{ M_{U}({\mathscr{T}}):=M_{U}^\circ({\mathscr{T}})[1/p]}\,,
\qquad 
N_U(\mathscr{T}^\prime):=N_U^\circ (\mathscr{T}^\prime)[1/p]\,.
\end{equation}
\end{defn}

The evaluation map 
\[
\mathrm{ev}\,:\, {\mathbb A}^\prime (U)^\circ \otimes \DD^\prime (U)^\circ\lra \LLU
\]
induces a pairing 
\[
\mathscr{A}^0_{U}({\mathscr{T}^\prime}) \otimes \mathscr{D}^\circ_U(\mathscr{T}^\prime)
\lra \sheafLLU ,
\]
where $\sheafLLU$ is the sheaf associated to $\LLU$ as in \cite[\S4.2]{BSV2020}.
Together with the trace map
$H^2_{\textup{\'et},c}(Y_{\overline{\QQ}}, \sheafLLU  (1)) \stackrel{\sim}{\to} \LLU$, this pairing induces the pairing
\begin{equation}
\label{formula: pairing for overconvergent cohomology over R}
N_U^\circ (\mathscr{T}^\prime) \otimes_{\Lambda_U} M_{U}^0({\mathscr{T}^\prime}) \lra \LLU.
\end{equation}
We finally recall that Proposition 4.4.5 of \cite{LZ1} supplies us with a morphism of sheaves on $Y$:
\begin{equation}
\label{eqn_Lambda_to_distributions_sheafified}
 \uLambda (\umT \langle D^\prime\rangle) \lra \mathscr{D}^\circ_{U}(\umT^\prime).
\end{equation}

\subsubsection{}
\label{subsec:slope decomposition}

By GAGA, we have the canonical isomorphisms 
\begin{equation}
\label{eqn:GAGA iso}
 M_{U} ({\mathscr{T}})(\chi_U) \simeq H^1(\Gamma_p, \DD (U)),
\qquad  
 M_{U} ({\mathscr{T}^{\prime}})(-1) \simeq H^1(\Gamma_p, \DD^{\prime}(U)),
\qquad 
N_{U}({\mathscr{T}^\prime})  \simeq H^1_c(\Gamma_p, \AA^{\prime}(U))
 \end{equation}
(c.f. \cite{andreattaiovitastevens2015}, Proposition~3.18).
The isomorphisms  in \eqref{eqn:GAGA iso} allow us to define a Hecke action on the spaces on the left of each isomorphism. The submodules $M_{U}({\mathscr{T}^\prime})^{\leq \nu},$ $M_{U}({\mathscr{T}})^{\leq \nu}$ and $N_U (\mathscr{T}^\prime)^{\leq \nu}$ are stable under the action of $G_\QQ$ and the actions of prime-to-$p$ Hecke operators. 
 
 For an affinoid $\mathcal{W}$ contained in a wide open $U$ in the weight space, we set 
$M_{\mathcal W}({\mathscr{T}^\prime})^{\leq \nu}:=M_{U}({\mathscr{T}^\prime})^{\leq \nu}\otimes_{\LLUp} O_{\mathcal W}$. We similarly define $M_{\mathcal W}({\mathscr{T}})^{\leq \nu}$ and $N_{\mathcal W}({\mathscr{T}^\prime})^{\leq \nu}$.

\begin{proposition}
\label{prop_andreattaiovitastevens2015_Prop318}
Let $x_0$ denote a cuspidal point of the eigencurve $\mathcal C.$ 
There exist  a sufficiently small wide open $U,$ an affinoid $\mathcal{W} \subset U$ and a slope $\nu$ such that $x_0\in \mathcal C_{\mathcal{W},\nu}$ and  the following hold true.

\item[i)] Let $\mathcal X$  denote the connected component of $x_0$ in $\mathcal C_{\mathcal{W},\nu}.$ There is an isomorphism of $O_{\mathcal X}$-modules
$${\rm comp}: {\rm Symb}_{\Gamma_p}(\mathbb{D}^\dagger(\mathcal W))^{\leq \nu}{\otimes_{\cH_{\mathcal W}}O_{\mathcal X}}\stackrel{\sim}{\lra} M_{\mathcal W}({\mathscr{T}})^{\leq \nu}{\otimes_{\cH_{\mathcal W}}O_{\mathcal X}}$$
which interpolates Artin's comparison isomorphisms between Betti and \'etale cohomology. 
\item[ii)] All three $O_{\mathcal X}$-modules  
\begin{equation}
\begin{aligned}
\label{eqn_defn_prop_andreattaiovitastevens2015_Prop318}
 M_{\mathcal W}(\mathscr{T})^{\leq \nu}\otimes_{\cH_{\mathcal W}}
 O_{\mathcal X}\,,\qquad
 M_{\mathcal W}({\mathscr{T}^\prime})^{\leq \nu}\otimes_{\cH_{\mathcal W}}O_{\mathcal X}\,,\qquad
N_{\mathcal W}({\mathscr{T}^\prime})^{\leq \nu}\otimes_{\cH_{\mathcal W}}O_{\mathcal X}
\end{aligned}
\end{equation}
are free of rank $2$.
\end{proposition}

We remark that it is absolutely crucial that the affinoid neighborhood $\mathcal{X}$ of $x_0$ falls within the cuspidal eigencurve. 

\begin{proof}

\item[i)] This portion follows from \eqref{eqn:GAGA iso}, 
Lemma~\ref{lemma: shrinking wide open neighborhood} and Proposition~\ref{claim}.

\item[ii)] It follows from (i) and Theorem~\ref{thm_bellaiche} that  $M_{\mathcal W}(\mathscr{T})^{\leq \nu}\otimes_{\cH_{\mathcal W}}O_{\mathcal X}$ is free of rank $2$ over $O_{\mathcal X}.$ By \cite[Proposition~4.4.8.4]{LZ1}, $M_{\mathcal W}({\mathscr{T}^\prime})^{\leq \nu}\otimes_{\cH_{\mathcal W}}O_{\mathcal X}$ is isomorphic to $ M_{\mathcal W}(\mathscr{T})^{\leq \nu}\otimes_{\cH_{\mathcal W}}O_{\mathcal X}$ as $O_{\mathcal W}$-modules and therefore, they both have the same $\cO_{\mathcal X}$-rank. Finally, $N_{\mathcal W}({\mathscr{T}^\prime})^{\leq \nu}\otimes_{\cH_{\mathcal W}}O_{\mathcal X}$ have rank $2$ by the isomorphism \eqref{eqn:GAGA iso}
and Corollary~\ref{cor_main_Bellaiche_BCM_eigencurve_variant}.
 \end{proof}
 
\subsubsection{}
Let $j\geq 1$ be a positive integer. 
The map $F (x)\mapsto F(x-j)$ defines an isomorphism
$t_j\,:\,O^0_{U-j} \simeq O^0_U$. If $F\in \AA^\prime (U-j)^0$ and $G\in \mathscr{P}_j(O_E)$ is a homogeneous polynomial 
of degree $j,$ then $t_j\circ (FG)\in \AA^\prime (U)^0$, and we have a well defined 
map $\AA^\prime (U-j)^0 \otimes \mathscr{P}_j(O_E) \rightarrow  \AA^\prime (U)^0.$
Passing to the duals, we obtain a map 
\begin{equation}
\nonumber
\beta_j^* \,:\, \DD^{\prime}(U)^\circ \lra \DD^{\prime}(U-j)^\circ \otimes \TSym^j(O_E^2)
\end{equation}
and the associated morphism of sheaves:
\begin{equation}
\label{eqn: definition of beta}
\beta^*_j \,:\, \mathscr D_U^\circ (\mathscr T^{\prime})\lra \mathscr D_{U-j}^\circ (\mathscr T^{\prime})\otimes \TSym^j(\mathscr{T}).
\end{equation}
Let us denote by
\begin{equation}
\nonumber
\delta_j\,:\,
\AA^{\prime} (U)^\circ \lra \AA^{\prime}(U-j)^\circ \otimes \mathscr P_j(O_E)
\end{equation}
the map given by $\delta_j (F):=\dfrac{1}{j!}\underset{i+m=j}\sum
\dfrac{\partial^j F(x,y)}{\partial x^i\partial y^m}\otimes x^iy^m$. On transposing this map, we obtain
\begin{equation}
\nonumber
\delta_j^* \,:\, \DD^{\prime} (U-j)^\circ  \otimes \TSym^j(O_E^2) 
\lra \DD^{\prime} (U)^\circ  
\end{equation}
and the induced morphism of sheaves
\begin{equation}
\label{definition of delta}
\delta_j^* \,:\, \mathscr{D}_{U-j}^\circ (\mathscr{T}^\prime ) \otimes \TSym^j(\mathscr{T}) \lra \mathscr{D}_{U}^\circ (\mathscr{T}^\prime ).
\end{equation}

\subsection{Big Galois representations}

\subsubsection{} 
\label{subsubsec_621}
For the convenience of the reader, we review Deligne's construction of $p$-adic representations associated  to eigenforms. Let  $\mathscr{T}_{\QQ_p}^\vee$ be  the dual sheaf  of $\mathscr{T}_{\QQ_p}$. Let $f$ denote a $p$-stabilized eigenform of weight  $k+2\geqslant 2$ and level $\Gamma_p$, which is new away from $p$ (but could be new or old at $p$). Deligne's representation associated to $f$ is defined as the $f$-isotypic Hecke submodule $V[f]\subset H^1_{\textup{\'et},c}(Y_{\overline{\QQ}},{\rm Sym}^{k}\mathscr{T}_{\QQ_p}^\vee)$  for the Hecke operators $\{T_\ell\}_{\ell\nmid Np}$ and $\{U_\ell\}_{\ell\mid Np}$, which are given as in \cite[\S4.9]{kato04} (but denoted in op. cit. by $T(\ell)$ for all $\ell$). Recall that this is a two dimensional vector space equipped with a continuous  action of $G_{\QQ}$, which is de Rham at $p$ (crystalline iff $f$ is $p$-old) with Hodge--Tate weights  $(0,k+1)$.  By the theory of newforms, the Hecke module $H^1_{\textup{\'et}, c}(Y_{\overline{\QQ}},{\rm Sym}^{k}\mathscr{T}_{\QQ_p}^\vee)$ is $f$-semisimple. 
In particular, $V[f]$  coincides 
with the \emph{generalized} Hecke eigenspace $V[[f]]\subset H^1_{\textup{\'et},c}(Y_{\overline{\QQ}},{\rm Sym}^{k}\mathscr{T}_{\QQ_p}^\vee)$ in the sense of \cite[\S1.2]{bellaiche2012}. We denote  by $V_f$ the  $f$-isotypic quotient
of  $H^1_{\textup{\'et},c}(Y_{\overline{\QQ}},{\rm Sym}^{k}\mathscr{T}_{\QQ_p}^\vee)$. Then the natural projection $V[f]\rightarrow V_f$ is an isomorphism.

Dually, consider $H^1_{\textup{\'et}}(Y_{\overline{\QQ}},{\rm TSym}^{k}\mathscr{T}_{\QQ_p})(1)$. The Hecke algebra $\mathcal H$ acts on this space via the dual Hecke operators  the dual Hecke operators $\{T_\ell'\}_{\ell\nmid Np}$ and $\{U_\ell'\}_{\ell\mid Np}$, which are given as in \cite[\S4.9]{kato04} (denoted by $T'(\ell)$ for all $\ell$). Let $V^{\prime}[f]$ (respectively, $V^{\prime}_f$)  denote the $f$-isotypic Hecke submodule (respectively, the $f$-isotypic quotient)  of $H^1_{\textup{\'et}}(Y_{\overline{\QQ}},{\rm TSym}^{k}\mathscr{T}_{\QQ_p})(1)$. Paralleling the discussion in the previous paragraph, the Hecke module $H^1_{\textup{\'et}}(Y_{\overline{\QQ}},{\rm TSym}^{k}\mathscr{T}_{\QQ_p})(1)$ is semisimple at $f$ and $V^{\prime}[f]$ coincides with the generalized Hecke eigenspace $V^{\prime}[[f]]$ inside $H^1_{\textup{\'et}}(Y_{\overline{\QQ}},{\rm TSym}^{k}\mathscr{T}_{\QQ_p})(1).$
The canonical projection $V^{\prime}[f]\rightarrow V^{\prime}_f$ is an isomorphism.  Poincar\'e duality induces a perfect Galois-equivariant pairing 
\begin{equation}
\label{formula:Poincare for modular forms}
(\,,\,)_f\,:\,V_{f}^{\prime}\otimes V[f] \lra E .
\end{equation} 

\subsubsection{}
For any normalized eigenform $f=\sum_{n\geq 1}a_nq^n \in S_{k+2}(\Gamma_p)$,  we set
\[
f^c=\sum_{n\geq 1}\overline{a}_nq^n .
\]
Since  $f$ is an eigenform  for the Hecke operators $\{T_\ell\}_{\ell\nmid Np}$ and $\{U_\ell\}_{\ell\mid Np}$ with eigenvalues $\{a_\ell\}$, it is also an eigenform for the dual operators $\{T'_\ell\}_{\ell\nmid Np}$ and $\{U'_\ell\}_{\ell\mid Np}$, with eigenvalues $\{\overline{a}_\ell\}.$
The isomorphism $\mathscr{T}_{\QQ_p}^\vee(1) \simeq \mathscr{T}_{\QQ_p}$ induces a Hecke equivariant map
\[
H^1_{\textup{\'et},c}(Y_{\overline{\QQ}},{\rm Sym}^{k}\mathscr{T}_{\QQ_p}^\vee)(k+1)
\lra H^1_{\textup{\'et}}(Y_{\overline{\QQ}},{\rm TSym}^{k}\mathscr{T}_{\QQ_p})(1),
\]
which gives rise to a canonical isomorphism
\begin{equation}
\label{eq: duality V'[f] vs V[f*]}
V[f^c](k+1)\simeq V'[f]
\end{equation}
on the $f$-isotypic subspaces for the dual operators $\{T'_\ell\}_{\ell\nmid Np}$ and $\{U'_\ell\}_{\ell\mid Np}$\,. 
With obvious modifications, all these constructions also make sense for modular forms of level $\Gamma_1(N)$.

\subsubsection{}
\label{subsubsec_622}
We maintain  the notation and the conventions of Section~\ref{subsec:slope decomposition}. Fix a $p$-stabilized eigenform $f$ as in \S\ref{subsubsec_621} and denote by $x_0$ the corresponding point of the eigencurve $\mathcal C$. Let $\alpha_0$ denote the $U_p$-eigenvalue on $f$. 

Fix a slope  $\nu \geqslant v_p(\alpha_0)$ a wide open affinoid $U$ and an affinoid $\mathcal W$  as in Proposition~\ref{prop_andreattaiovitastevens2015_Prop318}. Let $\mathcal X$ denote the connected component of $x_0$ in $\mathcal C_{\mathcal W,\nu}.$

In the remainder of this paper we will work with the $O_{\mathcal X}$-adic Galois representations 
\begin{equation}
\begin{aligned}
\label{eqn_defn_prop_andreattaiovitastevens2015_Prop318}
&V_{\mathcal X}:=  N_{\mathcal W}({\mathscr{T}^\prime})^{\leq \nu}\otimes_{\cH_{\mathcal W}}O_{\mathcal X},\\
&V_{\mathcal X}^{\prime}:= M_{\mathcal W}({\mathscr{T}^\prime})^{\leq \nu}\otimes_{\cH_{\mathcal W}}O_{\mathcal X}.
\end{aligned}
\end{equation}
By Proposition~\ref{prop_andreattaiovitastevens2015_Prop318}, both $\cO_{\mathcal X}$-modules have rank two.

We denote by 
\begin{equation}
\label{formula: pairing for overconvergent cohomology over A}
(\,,\,)_{\mathcal X}\,: \,V_{\mathcal X}^{\prime}\otimes V_{\mathcal X} \lra O_{\mathcal W}
\end{equation}
the $G_{\QQ,S}$-equivariant cup-product pairing induced by 
\eqref{formula: pairing for overconvergent cohomology over R}.

\subsubsection{} 
\label{subsubsection_623}

For any $E$-valued classical point $x\in \mathcal{X}^{\rm cl}(E)$ of  $\mathcal{X}$, we let $f_x$ denote the corresponding $p$-stabilized cuspidal eigenform of weight $w(x)+2$. In particular, $f_{x_0}$ is our fixed $p$-stabilized eigenform $f$. By Proposition~\ref{prop_bellaiche_4_11} we have $O_{\mathcal X}\simeq O_{\mathcal W}[X]/(X^e-Y)$ for some $e\geqslant 1,$ so the formalism  of Section~\ref{sec_Bellaiche_axiomatized} applies in this context.

We define the representations $V_{\mathcal X,x},$ $V_{\mathcal X, w(x)},$ 
$ V_{\mathcal X}[x]$ and $V_{\mathcal X}[[x]]$
(respectively $V^{\prime}_{\mathcal X,x},$ $V^{\prime}_{\mathcal X, w(x)},$ 
$ V^{\prime}_{\mathcal X}[x]$ and $V^{\prime}_{\mathcal X}[[x]]$) by plugging in { ${M}=V_{\mathcal X}$ (respectively ${M}=V_{\mathcal X}^{\prime}$)}  in Definition~\ref{defn_PhiGamma_Big}. In particular, $V_{\mathcal X, x}:=V_{\mathcal X}\otimes_{O_{\mathcal X},x}E$
and $V_{\mathcal X, w(x)}:=V_{\mathcal X}\otimes_{O_{\mathcal W},w (x)}E$. We recall that $e=1$ except when the $p$-stabilized form $f_{x_0}$  is $\theta$-critical. 

Let $\mathscr P_{k}(O_E)$ denote the space of homogeneous polynomials
in two variables  of  degree $k$ with coefficients in $E$. If $k\in U,$ we have a canonical  injection $\mathscr P_{k}(O_E) \hookrightarrow A^{\prime}(U)^\circ,$ which induces 
a map ${\mathrm{Sym}}^{k}(O_E^{\oplus 2}) \hookrightarrow A^{\prime}(U)^\circ$
and the dual map $D^{\prime}(U)^\circ \rightarrow {\mathrm{TSym}}^{k}(O_E^{\oplus 2}).$
We have the associated morphisms of sheaves
\begin{equation}
\label{eqn:evaluation of distributions on polynomials} 
\mathscr{D}_U^\circ (\mathscr{T}^\prime) \lra \mathrm{Sym}^{k}\mathscr{T},
\qquad \mathrm{Sym}^{k}\mathscr{T}^\vee \lra \mathscr{A}_U^\circ (\mathscr{T}^\prime)
\end{equation} 
which  give rise to the pair of Hecke equivariant morphisms
\begin{align}
\begin{aligned}
&H^1_{\textup{\'et}, c}(Y_{\overline{\QQ}},{\rm Sym}^{k}\mathscr{T}_{\QQ_p}^\vee) \lra  H^1_{\textup{\'et}, c}(Y_{\overline{\QQ}},\mathscr{A}_U (\mathscr{T}^\prime) )\,,
\\
\label{eqn:projection of cohomology on k-weight}
&H^1_{\textup{\'et}}(Y_{\overline{\QQ}},\mathscr{D}_U (\mathscr{T}^\prime) )
\lra H^1_{\textup{\'et}}(Y_{\overline{\QQ}},{\rm TSym}^{k}\mathscr{T}_{\QQ_p})\,.
\end{aligned}
\end{align}
We obtain the following morphisms by the definitions of $M_U (\mathscr{T}^{\prime})$ and $N_U (\mathscr{T}^{\prime})$: 
\[
\begin{aligned}
&H^1_{\textup{\'et}, c}(Y_{\overline{\QQ}},{\rm Sym}^{k}\mathscr{T}_{\QQ_p}^\vee) \lra N_U (\mathscr{T}^{\prime})\,,\\
&M_U (\mathscr{T}^{\prime})\lra 
H^1_{\textup{\'et}}(Y_{\overline{\QQ}},{\rm TSym}^{k}\mathscr{T}_{\QQ_p}(1))\,.
\end{aligned}
\]
On tensoring with $O_{\mathcal X}$ and taking slope $\leq \nu$ submodules,
we deduce for any $x\in \mathcal X^{\mathrm{cl}}(E)$ that there are canonical 
morphisms
\begin{align}
\begin{aligned}
& V[f_x] \xrightarrow{j_x} V_{\mathcal X} [[x]]
\hookrightarrow V_{\mathcal X}\, ,\\
\label{eqn: specialization of big modular representation}
&V^{\prime}_{\mathcal X}  \lra    V^{\prime}_{\mathcal X, w(x)} 
\xrightarrow{\pi_x} V^{\prime}_{f_x}\,. 
\end{aligned}
\end{align}
It follows from definitions that 
\begin{itemize}
\item $V_{\mathcal X,w(x)}$ and $V_{\mathcal X,w (x)}^{\prime}$  are $E$-vector spaces of dimension $2e$, 
where $e$ denotes the ramification degree of the weight map $w\,:\,\mathcal X\rightarrow \mathcal W$ at $x_0$\,;
\item $V_{\mathcal X}{[x]}$ and $V^{\prime}_{\mathcal X}{[x]}$  are the $f_x$-isotypic Hecke eigenspaces inside $V_{\mathcal X,w (x)}$ and $V^{\prime}_{\mathcal X,w(x)}$ (in alternative wording, $V_{\mathcal X}^?{[x]}$ is the largest subspace of $V_{\mathcal X,w(x)}^?$ annihilated by $X-X(x)\in O_{\mathcal X}$);  and these are both $E$-vector spaces of dimension $2$.
\end{itemize}

For any $x\in \mathcal{X}^{\rm cl}(E)$, let us also write $(\,,\,)_{w (x)}$ for the pairing 
$$(\,,\,)_{w (x)}\,:\, V^{\prime}_{\mathcal X, w (x)}\otimes_E V_{\mathcal X,w (x)}\lra E$$
induced from the pairing \eqref{formula: pairing for overconvergent cohomology over A} by $O_{\mathcal W}$-linearity. We summarize the basic properties of these objects (that we shall make use of in what follows) in Proposition~\ref{lemma_Adj_diagram_eigencurve} below.

\begin{proposition}
\label{lemma_Adj_diagram_eigencurve} Suppose $x\in \mathcal X^{\mathrm cl}(E)$ is an $E$-valued classical point.

\item[i)]  The Hecke-equivariant map 
$$V^{\prime}_{\mathcal X} [[x]]\hookrightarrow V^{\prime}_{\mathcal X, w(x)}\xrightarrow{\pi_x} V^{\prime}_{f_x}$$
is surjective. In particular, it induces an isomorphism
\[
V^{\prime}_{\mathcal X,x}\stackrel{\sim}{\lra} V^{\prime}_{f_x}
\]
\item[ii)] The Hecke-equivariant map 
$$
V[f_x]\xrightarrow{j_x} V_{\mathcal X}[[x]],
$$
is injective. In particular, it induces an isomorphism
\[
V[f_x]\stackrel{\sim}{\lra} V_{\mathcal X}[x].
\]

\item[iii)] The pairing $(\,,\,)_{\mathcal X}$ verifies the property \ref{item_P1}.

\item[iv)] The restriction of $(\,,\,)_{w (x)}$ to $V^{\prime}_{\mathcal X,w(x)}\otimes V_{\mathcal X}[x]$ factors as
$$\xymatrix@C=.1cm{V^{\prime}_{\mathcal X,w (x)}\ar[d]_{\pi_x}&\otimes &V_{\mathcal X}[x]\ar[rrrr]^(.6){(\,,\,)_{w (x)}}&&&& E \\
V^{\prime}_{\mathcal X,x}&\otimes &V_{\mathcal X}[x]\ar@{=}[u]\ar[rrrr]&&&& E\ar@{=}[u]}$$
We shall denote the induced pairing on $V^{\prime}_{\mathcal X,x}\otimes V_{\mathcal X}[x]$ by $(\,,\,)_{x}$.
\item[v)] The diagram 
$$\xymatrix@C=.1cm{V^{\prime}_{\mathcal X,x}\ar[d]^{\sim}_{\pi_x}&\otimes &V_{\mathcal X}[x]\ar[rrrr]^(.6){(\,,\,)_{x}}&&&& E\\
V^{\prime}_{f_x}&\otimes &V[f_x]\ar^{\sim}[u]_{j_x}\ar[rrrr]_(.6){(\,,\,)_{f_x}}&&&& E
\ar@{=}[u]}$$
where the bottom pairing is the Poincar\'e duality \eqref{formula:Poincare for modular forms}, commutes. Namely, for all $v'\in V^{\prime}_{\mathcal X,x}$ and $v\in V[f_x]$ one has
\[
(\pi_x(v'),v)_{f_x}=(v', j_x(v))_x.
\]

\end{proposition}

\begin{proof} \item[i)] The surjectivity of  $V^{\prime}_{\mathcal X}[[x]] \rightarrow V^{\prime}_{f_x}$ follows from  \cite[Corollary 3.19]{bellaiche2012}, where the analogous statement is proved for the spaces of modular symbols, combined with Proposition~\ref{prop_andreattaiovitastevens2015_Prop318}(i).  
\item[ii)] This is clear by definitions.
\item[iii)] This portion follows from the \cite[\S4.2]{BSV2020} (see in particular the discussion following (69)). 
\item[iv)] This is a particular case of the factorization \eqref{eqn_prop_221_atempt_1_0}.
\item[v)] This statement follows from the functoriality  of the  cup products.
\end{proof}

\subsection{Beilinson--Kato elements over the eigencurve}
\label{subsec_BK_families}

\subsubsection{}
We maintain previous notation and conventions in \S\ref{sec_bellaiche_eigencurve} and \S\ref{sec_interpolateBK_elements}. Let  $x_0\in \mathcal C^{\mathrm{cusp}}(E)$ be a classical cuspidal point on the eigencurve. Fix a wide open $U$, an affinoid $\mathcal W$ and a slope $\nu \geqslant 0$ such that  $x_0\in\mathcal{W} \subset U$ 
and  the conditions of Proposition~\ref{prop_andreattaiovitastevens2015_Prop318}
are satisfied. Let $\mathcal X$ denote the connected
component of $x_0$ in $\mathcal C_{\mathcal W,\nu}.$
Consider the following maps:
\begin{itemize}
\item[$\bullet$]{} For each non-negative integer $k\in U,$ the map \eqref{eqn:projection of cohomology on k-weight} which we notate as
\begin{equation}
\nonumber
\rho_{U}^{[k]}\,:\, H^1_{\mathrm{\acute et}}(Y_{\overline{\QQ}}, \mathscr{D}_{U}(\mathscr T^{\prime}_{\Qp}))
\lra 
H^1_{\mathrm{\acute et}} (Y_{\overline{\QQ}}, \TSym^{k}(\mathscr{T}_{\Qp})).
\end{equation}

\item[$\bullet$]{} For each classical $x\in \mathcal X,$ the map 
\eqref{eqn: specialization of big modular representation},
which we shall notate as 
\[
\rho_{\mathcal X}^{[x]}\,:\, V^\prime_{\mathcal X} \rightarrow V'_{f_x}.
\]
Moreover, we have the following commutative diagram, which follows from the slope decomposition:
\begin{equation}
\label{eqn:compatibility of representations vs slopes}
\begin{aligned}
\xymatrix
{
M_{\mathcal W}(\mathscr{T}^\prime) \ar[d]
\ar[r]&H^1_{\mathrm{\acute et}}(Y_{\QQ}, \TSym^{w(x)}\mathscr T_{\Qp}(1))
\ar[d]\\
V^\prime_{\mathcal X}\ar[r] \ar[r] &V^\prime_{f_x} 
.
}
\end{aligned}
\end{equation}
Here, by definition, 
$M_{\mathcal W}(\mathscr{T}^\prime):= H^1_{\mathrm{\acute et}}(Y_{\overline{\QQ}}, \mathscr{D}_U(\mathscr{T}^{\prime}))_{\mathcal W}.$ 

\item[$\bullet$]{} The cyclotomic moment maps 
 \[
 \mathrm{mom}^{[r,n]}_{\mu_m^\circ}\,:\,\Lambda (\Gamma_{\QQ (\zeta_m)})^\iota 
\rightarrow E[G_n] (-r),\qquad r\in \ZZ,  n\geqslant 0,
\] 
which have defined defined in \S\ref{subsubsection:big Kato elements}. Together with the  maps $\rho_{U}^{[k]}$ and $\rho_{\mathcal X}^{[x]}$, the cyclotomic moment maps give rise to the morphisms
\begin{equation}
\label{eqn_Sec651_1}
\rho_{U,m}^{[k,r,n]}\,:\,
H^1(\ZZ [1/S, \zeta_m], H^1_{\mathrm{\acute et}}(Y_{\overline{\QQ}}, \mathscr{D}_{U}^\circ(\mathscr T^{\prime})\widehat{\otimes} \underline{\LL}_{\mu_m^\circ}^\cyc (2)))
\lra
H^1(\ZZ [1/S, \zeta_{mp^n}], H^1_{\mathrm{\acute et}} (Y_{\overline{\QQ}}, \TSym^{k}(\mathscr{T})(2-r) ))
\end{equation}
and 
\begin{equation}
\label{eqn_Sec651_2}
\rho_{\mathcal X,m}^{[x,r,n]}\,:\, H^1(\ZZ [1/S, \zeta_m] , V^\prime_{\mathcal X}\widehat{\otimes}\,
\Lambda (\Gamma_{\QQ (\zeta_m)})^\iota (1))
\lra 
H^1(\ZZ [1/S, \zeta_{mp^n}] , V^\prime_{f_x} (1-r))\,.
\end{equation}
It follows from the commutativity of the diagram \eqref{eqn:compatibility of representations vs slopes} that the maps \eqref{eqn_Sec651_1} and \eqref{eqn_Sec651_2} are compatible in the evident sense.
 
\end{itemize}

\subsubsection{}
The map \eqref{eqn_Lambda_to_distributions_sheafified} induces a morphism
\[
H^1_{\textrm{\textup{\'et}}} (Y_{\overline{\QQ}}, \underline{\Lambda} (\mathscr{T}\left < D'\right >))
\lra H^1_{\textrm{\textup{\'et}}} (Y_{\overline{\QQ}}, \mathscr{D}_U^\circ (\mathscr{T}')).
\]
The composition of this map with \eqref{eqn_BenoisHorte_70} gives rise to the map
\begin{equation}
\label{eqn:map from cohomology of Y1 to Y}
H^1_{\textrm{\textup{\'et}}} (Y_1(Np)_{\overline{\QQ}}, \underline{\Lambda}_{Np}(2))
\lra H^1_{\textrm{\textup{\'et}}} (Y_{\overline{\QQ}}, \mathscr{D}_U^\circ (\mathscr{T}')(2))
=: M_U^\circ (\mathscr{T}^{\prime}) (1).
\end{equation}
More generally, for each integer $j\geqslant 0$, the morphism  \eqref{definition of delta} together with the trace map induce 
\begin{align}
\label{eqn:map from cohomology of Y1 to Y with j}
\begin{aligned}
H^1_{\textrm{\textup{\'et}}} (Y_1(Np)_{\overline{\QQ}}, \TSym^j (\mathscr{T}_{Np})\otimes \underline{\Lambda}_{Np}(2))
\lra H^1_{\textrm{\textup{\'et}}} (Y_{\overline{\QQ}}, \TSym^j (\mathscr{T})\otimes \mathscr{D}_{U-j}^\circ (\mathscr{T}')(2))\\
\xrightarrow{\delta_j^*}
 H^1_{\textrm{\textup{\'et}}} (Y_{\overline{\QQ}}, \mathscr{D}_{U}^\circ (\mathscr{T}')(2))=: M_U^\circ (\mathscr{T}^{\prime}) (1)\,. 
 \end{aligned}
\end{align}

\begin{defn}
\label{defn_big_BK_lambdaadicsheaf}
Let $S, \xi, j, c, d, m$ be as in Definition~\ref{defn:Kato big-zeta}.

\begin{itemize}
\item[i)]{}
We let
$${}_{c,d}{\mathbb{BK}}^{[U]}_{N,m} (j, \xi) \in 
H^1_{\mathrm{\acute et}}\left (\ZZ [1/S, \zeta_m],  M_U^\circ (\mathscr{T}^{\prime})\widehat{\otimes}\Lambda (\Gamma_{\QQ (\zeta_m)})^\iota (1) \right )
$$
denote the image of the Beilinson--Kato element ${}_{c,d}{\mathbb{BK}}_{N,m} (j, \xi)$ under the map 
\begin{multline}
\nonumber
H^1_{\mathrm{\acute et}}\left (\ZZ [1/S, \zeta_m],
H^1_{\mathrm{\acute et}} (Y_1(Np), \TSym^j (\mathscr{T}_{Np})\otimes \underline{\Lambda}_{Np}\widehat{\otimes}\Lambda (\Gamma_{\QQ (\zeta_m)})^\iota(2)
\right )
\\ 
\lra H^1_{\mathrm{\acute et}}\left (\ZZ [1/S, \zeta_m],  M_U^\circ (\mathscr{T}^{\prime})\widehat{\otimes}\Lambda (\Gamma_{\QQ (\zeta_m)})^\iota (1) \right )
\end{multline}
induced by the map \eqref{eqn:map from cohomology of Y1 to Y with j}.

\item[ii)]{} Let $x_0$ be a cuspidal point of the eigencurve $\mathcal C.$
Let  $\mathcal{W} \subset U$ be an affinoid disk centered at $x_0$ such that 
$\mathcal W,$ $U$ and $\nu \geqslant 0$ satisfy the conditions of Proposition~\ref{prop_andreattaiovitastevens2015_Prop318}. We denote by 
$${}_{c,d}{\mathbb{BK}}^{[\mathcal{W},\leq \nu]}_{N,m} (j, \xi)   \in 
H^1\left (\ZZ [1/S, \zeta_m], M_{\mathcal W}({\mathscr{T}^\prime})^{\leq \nu}\widehat{\otimes}\,
\Lambda (\Gamma_{\QQ (\zeta_m)})^\iota (1)\right )$$
the image of ${}_{c,d}{\mathbb{BK}}^{[U]}_{N,m} (j, \xi) $ under the map
induced by the natural projection $M_U^\circ (\mathscr{T}^{\prime})
\rightarrow M_{\mathcal W}({\mathscr{T}^\prime})^{\leq \nu}.$

\item[iii)]{} Let $\mathcal X$ denote the connected component of $x_0.$
We denote by 
$${}_{c,d}{\mathbb{BK}}^{[\mathcal{X}]}_{N,m} (j, \xi)   
\in H^1(\ZZ [1/S, \zeta_m] , V^\prime_{\mathcal X}\widehat{\otimes}\,
\Lambda (\Gamma_{\QQ (\zeta_m)})^\iota (1))$$
the image of ${}_{c,d}{\mathbb{BK}}^{[\mathcal{W},\leq \nu]}_{N,m} (j, \xi)$ under the map
induced by the natural projection $M_{\mathcal W}({\mathscr{T}^\prime})^{\leq \nu}\rightarrow V^{\prime}_{\mathcal X}.$
\end{itemize}
\end{defn}

\subsubsection{} The following is the interpolation property for $O_{\mathcal X}$-adic Beilinson--Kato elements ${}_{c,d}{\mathbb{BK}}^{[\mathcal{X}]}_{N,m} (j, \xi)$.
\begin{theorem}
\label{thm_control_thm_interpolation_BK} 
In the setting of Definition~\ref{defn_big_BK_lambdaadicsheaf}, the following are valid:
\item[i)]{} For any $k\geqslant 0,$ the diagram
\begin{equation}
\label{eqn:diagram from interpolation thm}
\begin{aligned}
\xymatrix
{H^1_{\mathrm{\acute et}}(Y_1(Np)_{\overline \QQ}, \TSym^j(\mathscr{T}_{Np})\otimes \underline{\LL}_{Np}) \ar[rr]^{\mathrm{mom}_{Np,j}^{[k]}}
\ar[d]_{\eqref{eqn:map from cohomology of Y1 to Y with j}} &
&H^1_{\mathrm{\acute et}} (Y_1(Np)_{\overline{\QQ}}, \TSym^{j+k}(\mathscr{T}_{Np}) )
\ar[d]^{\mathrm{pr}'_*}
\\
H^1_{\mathrm{\acute et}}(Y_{\overline{\QQ}}, \mathscr{D}_{U}^\circ(\mathscr T^{\prime}))
\ar[rr]^{\rho_{U}^{[j+k]}}
&
&H^1_{\mathrm{\acute et}} (Y_{\overline{\QQ}}, \TSym^{j+k}(\mathscr{T}) )
}
\end{aligned}
\end{equation}
commutes. Here, $\mathrm{mom}_{Np,j}^{[k]}$ is the moment map induced from \eqref{eqn:definition moment map}, and the right vertical map $\mathrm{pr}'_*$ is the trace map associated to the projection $\mathrm{pr}'\,:\, Y_1(Np) \rightarrow Y$.

\item[ii)]{} For any $m\geqslant 1$ and $k\geqslant 0,$ the diagram \eqref{eqn:diagram from interpolation thm} induces the commutative diagram
\begin{equation}
\label{eqn:diagram 2 from interpolation thm}
\begin{aligned}
\resizebox{15.5cm}{!}
{
\xymatrix
{
H^1(\ZZ [1/S, \zeta_m], H^1_{\mathrm{\acute et}}(Y_1(Np)_{\overline \QQ}, \TSym^j(\mathscr{T}_{Np})\otimes \underline{\LL}_{Np,\mu_m^\circ}^{\cyc}(2)) \ar[rr]^-{\mathrm{mom}_{Np,j,\mu_m^\circ}^{[k,r,n]}}
\ar[d]^{\mathrm{pr}'_*} &
&H^1(\ZZ [1/S, \zeta_{mp^n}], H^1_{\mathrm{\acute et}} (Y_1(Np)_{\overline{\QQ}}, \TSym^{j+k}(\mathscr{T}_{Np})(2-r) ))
\ar[d]^{\mathrm{pr}'_*}
\\
H^1(\ZZ [1/S, \zeta_m], H^1_{\mathrm{\acute et}}(Y_{\overline{\QQ}}, \mathscr{D}_{U}^\circ(\mathscr T^{\prime})\widehat{\otimes} \underline{\LL}_{\mu_m^\circ}^\cyc (2)))
\ar[rr]^{\rho_{U,m}^{[j+k,r,n]}}
&
&H^1(\ZZ [1/S, \zeta_{mp^n}], H^1_{\mathrm{\acute et}} (Y_{\overline{\QQ}}, \TSym^{j+k}(\mathscr{T})(2-r) )),
}
}
\end{aligned}
\end{equation}
where $\mathrm{mom}_{Np,j,\mu_m^\circ}^{[k,r,n]}$ is the map defined in Proposition~\ref{prop:big Kato elements}.

\item[iii)]{} For any $x\in \mathcal X^{\mathrm{cl}}(E),$ the diagram \eqref{eqn:diagram 2 from interpolation thm} induces the following  commutative diagram: 
\begin{equation}
\nonumber
\resizebox{16.5cm}{!}
{
\xymatrix
{
H^1(\ZZ [1/S, \zeta_m], H^1_{\mathrm{\acute et}}(Y_1(Np)_{\overline \QQ}, \TSym^j(\mathscr{T}_{Np})\otimes \underline{\LL}_{Np,\mu_m^\circ}^{\cyc}(2)) \ar[rr]^-{\mathrm{mom}_{Np,j,\mu_m^\circ}^{[x,r,n]}}
\ar[d]^{\mathrm{pr}'_*} &
&H^1(\ZZ [1/S, \zeta_{mp^n}], H^1_{\mathrm{\acute et}} (Y_1(Np)_{\overline{\QQ}}, \TSym^{w(x)} \mathscr{T}(2-r) ))_{f_x}
\ar[d]^{\simeq}
\\
H^1(\ZZ [1/S, \zeta_m], V_{\mathcal X}^\prime \widehat{\otimes} \LL (\Gamma_{\QQ (\zeta_m)})^{\iota}{(1)})
\ar[rr]^{\rho_{\mathcal X,m}^{[x,r,n]}}
&
&H^1(\ZZ [1/S, \zeta_{mp^n}],  V^\prime_{f_x}{(1-r) } ),
}
}
\end{equation}
where the right vertical isomorphism is induced by the canonical isomorphism
between $f_x$-isotypic components of $H^1_{\mathrm{\acute et}}(Y_1(Np)_{\overline \QQ},
\TSym^{w(x)}(\mathscr{T}_{Np, \Qp}))$ and  $H^1_{\mathrm{\acute et}}(Y_{\overline \QQ},
\TSym^{w(x)}(\mathscr{T_{\Qp}}))$ and the horizontal map $\mathrm{mom}_{Np,j,\mu_m^\circ}^{[x,r,n]}$ on the first row is the composition of $\mathrm{mom}_{Np,j,\mu_m^\circ}^{[w(x),r,n]}$ with the projection to the $f_x$-component. 

\item[iv)]{}  For any  $x\in \mathcal X^{\mathrm{cl}}(E),$ one has 
\[
\rho_{\mathcal X,m}^{[x,r,n]} \left ({}_{c,d}{\mathbb{BK}}^{[\mathcal X]}_{N,m} (j, \xi)\right )=  {}_{c,d}{\mathrm{BK}}_{Np,mp^n} (f_x^c,j, r, \xi) \]
where 
$${}_{c,d}{\mathrm{BK}}_{Np,mp^n} (f_x^c,j, r, \xi) \in H^1(\ZZ[1/S,\zeta_{mp^n}], V_{f_x^c}(w(x)+2-r))
$$
is the projection of the Beilinson--Kato element ${}_{c,d}{\mathrm{BK}}_{N p,mp^n} (w(x)-j,j, r, \xi)$ introduced in \S\ref{subsubsection:big Kato elements} to the $f_x^c$-isotypic eigenspace for the action of Hecke operators $\{T_\ell\}_{\ell\nmid Np} and \{U_\ell\}_{\ell\mid Np}$. Here we have used the canonical isomorphism \eqref{eq: duality V'[f] vs V[f*]} to identify ${}_{c,d}{\mathrm{BK}}_{N,mp^n} (f_x^c,j, r, \xi)$ with an element of $H^1(\ZZ[1/S,\zeta_{mp^n}], V_{f_x}'(1-r))$. 
\end{theorem}

\begin{proof}
\item[i)] Consider the following diagram:
\begin{equation}
\nonumber
\resizebox{16.5cm}{!}
{\xymatrix{
&H^1_{\textrm{\textup{\'et}}}(Y_1(Np)_{\overline{\QQ}}, \TSym^j(\mathscr{T}_{Np})\otimes \underline{\LL}_{Np})
\ar `[rr]^{\mathrm{mom}_{Np,j}^{[k]}}  [drr]
\ar    `l[dl]  `d[dddd] `r[dddrr]_{\eqref{eqn:map from cohomology of Y1 to Y with j}} [dddrr]
\ar[d] & &\\
&H^1_{\textrm{\textup{\'et}}}(Y_1(Np)_{\overline{\QQ}}, \TSym^j(\mathscr{T}_{Np})\otimes {\LL}(\mathscr{T}_{Np}\left <s_{Np}\right >)) \ar[r]^-{[Np]_*}  \ar[d]^{\mathrm{pr}^{\prime}_*}
\ar @{}  [dr]|{\refsymbolone} 
 &H^1_{\textrm{\textup{\'et}}}(Y_1(Np)_{\overline{\QQ}}, \TSym^j(\mathscr{T}_{Np})\otimes {\LL}(\mathscr{T}_{Np})) \ar[r]^-{\textrm{mom}^{[k]}}  \ar[d]^{\mathrm{pr}^{\prime}_*}
\ar @{}  [dr]|{\refsymboltwo}  
& H^1_{\textrm{\textup{\'et}}}(Y_1(Np)_{\overline{\QQ}}, \TSym^{j+k}(\mathscr{T}_{Np}))
\ar[d]^{\mathrm{pr}^{\prime}_*}
\\
&H^1_{\textrm{\textup{\'et}}}(Y_{\overline{\QQ}}, \TSym^j(\mathscr{T})  \otimes 
{\LL}(\mathscr{T}\left <D'\right >)  ) \ar[d] \ar[r]^{[p]_*}
&H^1_{\textrm{\textup{\'et}}}(Y_{\overline{\QQ}}, \TSym^{j}(\mathscr{T})\otimes {\LL}(\mathscr{T}) )  \ar[r]^-{\textrm{mom}^{[k]}} 
\ar @{} [dr]|(.65){\refsymbolfour} 
&H^1_{\textrm{\textup{\'et}}}(Y_{\overline{\QQ}}, \TSym^{j+k}(\mathscr{T}) ) 
\\
&H^1_{\textrm{\textup{\'et}}}(Y_{\overline{\QQ}}, \TSym^j(\mathscr{T})\otimes \mathscr{D}_{U-j}^\circ(\mathscr T^{\prime}))
\ar @{}  [uur]|(.3){\refsymbolthree}
 \ar[rr]_{\delta_j^*} \ar[urr]^(.4){\rho_{U}^{[k]}}
& &
H^1_{\textrm{\textup{\'et}}}(Y_{\overline{\QQ}}, \mathscr{D}_{U}^\circ(\mathscr T^{\prime}))
\ar[u]^{\rho_{U}^{[j+k]}}\\
& & &
}
}
\end{equation}
It follows from definitions that the map \eqref{eqn:map from cohomology of Y1 to Y with j} decomposes as explained in the diagram. The factorization of the  moment map $\mathrm{mom}_{Np,j}^{[k]}$ follows from its definition (c.f. \cite{KLZ2}, \S4) and the fact that it agrees with the description given in Section~\ref{subsec:moments} proved in \cite[Theorem~4.5.2.2]{KLZ2}. Therefore, it suffices to prove that this diagram commutes. The commutativity of the squares $\ensuremath{\circled{$1$}}$ and $\ensuremath{\circled{$2$}}$ follow from the functoriality of moment maps. 
The commutativity of the triangle $\ensuremath{\circled{$3$}}$ follows from
\cite[Proposition 4.2.10]{LZ1}. The commutativity of $\ensuremath{\circled{$4$}}$
is clear from definitions. 
\item[ii-iii)] These assertions follow immediately from (i). 
\item[iv)] This portion follows from (iii) and Proposition~\ref{prop:big Kato elements}.

\end{proof}

\begin{corollary}
\label{cor_thm_control_thm_interpolation_BK} We retain the notation and hypotheses of Theorem~\ref{thm_control_thm_interpolation_BK}. We let 
$${}_{c,d}{\mathrm{BK}}_{Np,{\rm Iw}} ( f_x^c,j, \xi) \in H^1(\ZZ[1/S], V_ {f_x^c}(w(x)+2)\otimes \LL^{ \iota})= H^1(\ZZ[1/S], V_{f_x}'\otimes \LL^{\iota}(1))$$
denote the unique class which specializes to ${}_{c,d}{\mathrm{BK}}_{Np,p^n} ( f_x^c,j, r, \xi)\in  H^1(\ZZ[1/S,\zeta_{p^n}], V_{ f_x^c}(w(x)+2-r))$  for all $n\in \mathbb{N}$ under the evident morphisms. Then,
$${}_{c,d}{\mathbb{BK}}^{[\mathcal X]}_{N,1} (j, \xi,x)={}_{c,d}{\mathrm{BK}}_{Np,{\rm Iw}} (f_x^c,j, \xi)$$
where ${}_{c,d}{\mathbb{BK}}^{[\mathcal X]}_{N,1} (j, \xi,x)$ is the specialization of ${}_{c,d}{\mathbb{BK}}^{[\mathcal X]}_{N,1} (j, \xi)$ to $x$.
\end{corollary}

\begin{remark}
\label{remark_compare_with_classical_Kato_newforms}
We retain the notation and hypotheses of Theorem~\ref{thm_control_thm_interpolation_BK}. Suppose that $x\in \cX^{\rm cl}(E)$ is a classical point and $f_x$ is $p$-old, arising as the $p$-stabilization of a newform $g$ of level $N$ coprime to $p$, with respect to the root $a_p(f_x)$ of the Hecke polynomial of $g$ at $p$. Then $f_x^c$ is the $p$-stabilization of the newform $g^c$  with respect to the root $\overline{a}_p(f_x)$ of the Hecke polynomial of $g^c$ at $p$. In this remark, we shall explain the relation between ${}_{c,d}{\mathrm{BK}}_{Np,\Iw} (f_x^c,j, \xi)$ and the Beilinson--Kato element ${}_{c,d}\mathrm{BK}_{N,\Iw}(g^c,j,\xi)$ that one associates to the newform $g^c$. 
The natural projection $Y(1,N(p))=:Y\xrightarrow{\mathrm{Pr}} Y_1(N)$ induces the trace map 
$$\mathrm{Pr}_*: H^1_{\textup{ \'et},c} (Y_{\overline{\QQ}},{\rm Sym}^{k}\mathscr{T}_{\QQ_p}^{\vee})\lra  H^1_{\textup {\'et},c} (Y_1(N)_{\overline{\QQ}},{\rm Sym}^{k}\mathscr{T}_{\QQ_p}^{\vee}),
$$
which restricts to an isomorphism $($since this map commutes with the action of Hecke operators $\{T_\ell'\}_{\ell\nmid Np}$$)$
\begin{equation}
\label{eqn_trace_map_fxctofx_1} 
\mathrm{Pr}_*: \quad V_{f_x^c}\simeq  V_{g^c} \,.
\end{equation}
Then by construction, the image of the Beilinson--Kato element ${}_{c,d}\mathrm{BK}_{N,\Iw}(g^c,j,\xi)$ under the natural morphism induced from the isomorphism \eqref{eqn_trace_map_fxctofx_1} coincides with ${}_{c,d}\mathrm{BK}_{N{p},\Iw}(f_ x^c,j,\xi)$ (c.f. \cite{kato04}, Proposition~8.7).
\end{remark}

\subsection{Normalized Beilinson--Kato elements}
\label{subsec_BK_big_normalized}
In \S\ref{subsec_BK_big_normalized}, we normalize the $\cO_{\mathcal X}$-adic Beilinson--Kato elements (in a weak sense). These normalized elements interpolate Kato's normalized elements in \cite[\S13]{kato04}. 

\subsubsection{} We recall that $\Lambda (\Gamma_1):=\Zp [[\Gamma_1]],$
where $\Gamma_1=\Gal (\Qp (\zeta_{p^\infty})/\Qp (\zeta_p))$ and $\Lambda=\Zp [\Delta]\otimes_{\Zp}\Lambda (\Gamma_1),$
where $\Delta=\Gal (\QQ (\zeta_p)/\QQ).$ One then has 
\[
\LL =\underset{\eta \in X(\Delta)}\bigoplus  \Lambda (\Gamma_1) e_{\eta}.
\]
Let us set $\Lambda_{\mathcal X}(\Gamma_1)= O_{\mathcal X}\widehat \otimes_{\Zp}\LL (\Gamma_1)$ and put $\Lambda_{\mathcal X}= O_{\mathcal X}\widehat \otimes_{\Zp}\LL. $

We start with the following auxiliary lemma.

\begin{lemma} 
\label{lemma: elimination of torsion}
Let $M$ be a finitely generated $\Lambda_{\mathcal X}(\Gamma_1)$-module. Suppose that $M_{x_0}:=M/XM$ is torsion-free as a $\LL (\Gamma_1) [1/p]$-module and that $M[X]=0$.
Then there exist an affinoid disk $\mathcal X'\subset \mathcal X$ such that $M_{O_{\mathcal X'}}:=M\otimes_{O_{\mathcal X}}O_{\mathcal X'}$ is torsion-free as a $\LL_{\mathcal X'} (\Gamma_1)$-module.
\end{lemma}

\begin{proof} 
Suppose on the contrary that $0\neq m\in M_{\mathrm{tor}}$ is a torsion element. For sufficiently small $\mathcal X'\subset \mathcal X$, the affinoid domain $\cO_{\mathcal X'}$ is a PID. Let us choose $\mathcal X'$ with that property and assume until the end of our proof that $\mathcal X=\mathcal X'$.   Since $\LL_{\mathcal X}$ is a UFD, we may choose $m$ so that $m$ us annihilated by an irreducible element $0\neq f\in \LL_{\mathcal X}$. Note that since $M[X]=0$, the irreducible element $f$ cannot be an associate of $X$ and moreover, we can assume without loss of generality that $m\in M\setminus XM$. 

Let us write $f=\sum_{n=0}^\infty a_n(X)(\gamma_1-1)^n$ where $a_n(X)\in \cO_{\mathcal X}$. Since $M_{x_0}$ is torsion-free as a  $\LL (\Gamma_1) [1/p]$-module, we infer that $a_0(0)=0$. Indeed, since we have $f\cdot (m+XM)=a_0(0)(m+XM)\in (M_{x_0})_{\rm tor}=\{0\}$ and as $m+XM\neq 0$ as an element of the torsion-free module $M/XM=M_{x_0}$, the required conclusion that $a_0(0)=0$ follows. This in turn shows that $X\mid f$. Since $f$ is irreducible, this in turn implies that $f$ is an associate of $X$, contrary to the discussion in the previous paragraph. This concludes the proof of our lemma.
\end{proof}

\subsubsection{} \label{subsubsec_662_CK1} Fix a $p$-stabilized normalized eigenform $f\in S_{k_0+2}(\Gamma_1(N)\cap \Gamma_0(p))$ where   $(N,p)=1$. Let us denote by $x_0$ the corresponding point of the cuspidal eigencurve and set
\[
\begin{aligned}
&
H^1_{\Iw} \left (\ZZ [1/S],V^\prime_{\mathcal X} (1) \right ):= H^1 \left (\ZZ [1/S],V^\prime_{\mathcal X}\widehat\otimes_{\Zp} \Lambda^{\iota} (1) \right ),\\
&
H^1_{\Iw} \left (\ZZ [1/S],V^\prime_{f} (1) \right ):= H^1 \left (\ZZ [1/S],V^\prime_{f}\widehat\otimes_{\Zp} \Lambda^{\iota} (1) \right )\,.
\end{aligned}
\]

\begin{lemma} 
\label{lemma_H1_Iw_torsion_free_small_X}
There exists a neighborhood $\mathcal X$ of $x_0$ such that $H^1_{\Iw} \left (\ZZ [1/S],V^\prime_{\mathcal X}(1)\right )$ is torsion-free  as a $\Lambda_{\mathcal X}(\Gamma_1)$-module. 
\end{lemma} 
\begin{proof}
We apply Lemma~\ref{lemma: elimination of torsion} to 
$M=H^1_{\Iw} \left (\ZZ [1/S],V^\prime_{\mathcal X}(1)\right )$.  The long exact sequence of Galois cohomology associated to the  short exact sequence 
\[
0\rightarrow  V^\prime_{\mathcal X} \widehat\otimes_{\Zp} \Lambda^{ \iota}   (1) \xrightarrow{X} 
V^\prime_{\mathcal X} \widehat\otimes_{\Zp} \Lambda^{ \iota} (1) \rightarrow V^\prime_{f}\widehat\otimes_{\Zp} \Lambda^{ \iota}  (1)
\rightarrow 0
\]
shows that we have an injection 
$$M\big{/}XM\hookrightarrow H^1_{\Iw} (\ZZ [1/S],V^\prime_{f} (1) )\,.$$
By Kato's fundamental result (c.f. \cite{kato04}, Theorem~12.4), the $\LL [1/p]$-module $H^1_{\Iw}  (\ZZ [1/S],V^\prime_{f} (1)  )$ is free of rank one. Therefore $M_{x_0}=M\big{/}XM$ is also torsion-free.
The same long exact sequence together with the fact that 
$ H^0_{\Iw} (\ZZ [1/S],V^\prime_{f} (1))=0$ gives 
$$
M[X]=H^1_{\Iw} \left (\ZZ [1/S],V^\prime_{\mathcal X} (1) \right )[X]={\rm im}\left( H^0_{\Iw} (\ZZ [1/S],V^\prime_{f}(1) )\lra H^1_{\Iw}  (\ZZ [1/S],V^\prime_{\mathcal X} (1)  )\right)=0\,.$$
Thus the hypotheses of Lemma~\ref{lemma: elimination of torsion} are satisfied and one can shrink $\mathcal X$ so as to ensure that $M$ is torsion-free.
\end{proof}

Let $Q(\LL_{\mathcal X})$ (resp., $Q(\LL)$) denote the total quotient  field  of $\LL_{\mathcal X}$ (resp., of $\LL$). It follows from Lemma~\ref{lemma: elimination of torsion} that on shrinking $\mathcal X$ as necessary, one can ensure that the natural $$H^1(\ZZ [1/S],V^\prime_{\mathcal X} (1) )\lra H^1  (\ZZ [1/S],V^\prime_{\mathcal X}(1) )\otimes_{\LL_{\mathcal X}} Q(\LL_{\mathcal X})$$
is an injection.

\subsubsection{}
We consider the product 
\begin{equation}
\label{eqn: E(x)-factor}
{\mathcal E}_N(x)=\underset{\ell \mid  N}\prod (1- a_\ell(\f) \sigma_\ell^{-1})  \in \LL_{\cX}
\end{equation}
as a two variable function in the  weight and cyclotomic variables. 

\begin{lemma}\label{lemma_bad_factors_monodromy_conj}
For each classical point $x$ and integer $m$, we have ${\mathcal E}_N(x,\chi^m)\neq 0$ provided that $m\not\in\{\frac{w(x)}{2}, \frac{w(x)+1}{2}\}$. In particular, we have ${\mathcal E}_N(x,\chi^{\frac{w(x)}{2}+1})\neq 0$.
\end{lemma}
\begin{proof}
According to \cite[Corollary 2]{SaitoWeightMonodromySupplement} (the Weight-Monodromy Conjecture for modular forms), $a_\ell(f_x)$ is an algebraic integer with complex absolute value either 
${ \ell^{\frac{w(x)}{2}}}$ or ${ \ell^{\frac{w(x)+1}{2}}}$. In particular, we have ${\mathcal E}_N(x,\chi^m)=(1-a_\ell(f_x)\ell^{-m})\neq 0$ whenever $m\not\in\{\frac{w(x)}{2}, \frac{w(x)+1}{2}\}$.
\end{proof}

\subsubsection{}
For an integer $c$, we let $\sigma_c \in \Gamma$ denote the element defined by the requirement that $\chi (\sigma_c)=c$.  We shall consider $\sigma_c$ also as an element of $\Lambda$ in the obvious manner. 
Following Kato,  we define the partial normalization factors 
\begin{equation}
\label{eqn_normalization_factor_in_families_bis}
\begin{aligned}
&\mu_1(c,j)=c^2-c^{-j} \sigma_c,\\
&\mu_2(d,j,x)=d^2-d^{j-w(x)}\sigma_d,\\
&\mu_0 (c,d,j,x):={\mu_1(c,j)}{\mu_2(d,j,x)}  \in \LL_{\mathcal X}\,.&&
\end{aligned}
\end{equation}
Here,  for any $x\in \mathcal{X}$, we define $d^{w(x)}:=d^{k_0}\langle d \rangle^{w(x)-k_0}$\,, where $\langle d\rangle^{w(x)-k_0}:=\exp ((w(x)-{k_0})\log\langle m\rangle )$. We remark that $\langle d\rangle^{w(x)-k_0} = d^{w(x)-k_0}$ whenever $x\in \mathcal{X}^{\rm cl}$, since $w(x)-k_0$ is divisible by $p-1$.
We also define the full normalization factor:  
\begin{equation}
\label{eqn_normalization_factor_in_families}
\mu (c,d,j,x):=\mu_0 (c,d,j,x)  {\mathcal E}_N(x) \in \LL_{\mathcal X}\,.
\end{equation}

\subsubsection{}
To simplify notation, we shall write $ {}_{c,d}\mathbb{BK}_{N}^{[\mathcal X]}(j,\xi)$ in place of  ${}_{c,d}\mathbb{BK}_{N,1}^{[\mathcal X]}(j,\xi).$  When $c,d$ and $j$ are understood, we will write $\mu_?$ in place of $\mu_?(c,d,j,x)$ for $?\in\{0,\emptyset\}$.  We denote by $\LL_{\mathcal X}[\mu^{-1}_0]\subset Q(\Lambda_{\mathcal X})$ the $\LL_{\mathcal X}$-subalgebra generated by $\mu_0(c,d,j,x)^{-1}$. We also denote by $\LL[\mu_{x_0}^{-1}]\subset Q(\LL)$ the $\LL[\frac{1}{p}]$-subalgebra generated by $\mu_0(c,d,j,x_0)^{-1}$.

\begin{defn}
\label{defn: normalized BK elements}
Let us fix $\mathcal W$ so that $H^1_{\Iw} \left (\ZZ [1/S],V^\prime_{\mathcal X} (1)\right )$ is torsion-free as a $\LL_{\mathcal X}$-module (we can choose such $\cW$ thanks to our discussion in \S\ref{subsubsec_662_CK1}). Let us fix integers $c\equiv 1\equiv d\,\,  {\rm mod\, }{N}$ with  $(cd,6p)=1$ and $c^2\neq 1\neq d^2$, and an integer $j\in [0,k_0]$. 
We define the partially normalized Beilinson--Kato elements on setting 
 \[
\mathbb{BK}_{N}^{[\mathcal X]}(j,\xi):=\mu_0(c,d,j,x)^{-1}{}_{c,d}\mathbb{BK}_{N}^{[\mathcal X]}(j,\xi)\in  H^1_{\Iw} \left (\ZZ [1/S],V^\prime_{\mathcal X} (1) \right ) \otimes_{\LL_{\mathcal X}} \LL_{\mathcal X}[\mu^{-1}_0]\,, \qquad \xi \in \mathrm{SL}_2(\ZZ)\,. 
\]
\end{defn}

 The integrality properties of the partial normalizations of the Beilinson--Kato elements are established in Proposition~\ref{prop_partial_normalization_of_BK_elements}, whose proof has been relegated to Appendix~\ref{sec_appendix_integrality}. 

\begin{proposition}
\label{prop_partial_normalization_of_BK_elements} 
Suppose $c,d$ and $j$ are as in Definition~\ref{defn: normalized BK elements}.
\item[i)] The partially normalized Beilinson--Kato element $\mathbb{BK}_{N}^{[\mathcal X]}(j,\xi):=\mu_0(c,d,j,x)^{-1}{}_{c,d}\mathbb{BK}_{N}^{[\mathcal X]}(j,\xi)$ is independent of the choice of $c$ and $d$. 
\item[ii)]  There exists an affinoid neighborhood $\mathcal W$ of $k_0$ so that  $\mathbb{BK}_{N}^{[\mathcal X]}(j,\xi)\in  H^1_{\Iw} \left (\ZZ [1/S],V^\prime_{\mathcal X} (1) \right ) $.
\end{proposition}

\begin{proof}
The first assertion is  Proposition~\ref{prop_independce_on_c_d_j_bis}, whereas the second is Proposition~\ref{prop_prop_partial_normalization_of_BK_elements_APPENDIX}  combined with (i).
\end{proof}

\subsubsection{} We check that $\mu (c,d,j,x)$ agrees with the normalization factor considered in \cite[\S13.9]{kato04}, whenever $x\in \cX^{\rm cl}$. For each newform $g$ of weight $k$ Kato  normalizes the elements that he denotes by 
$$\left({}_{c,d}{\bf z}_{p^n}^{(p)}(g,k,j,\xi,{{\rm prime}(pN)})\right)_{n\geq 1} \in  H^1_{\rm Iw}(\ZZ[1/S], V_{g}).$$ 
These elements are given as the Hecke equivariant projection of the elements
$\left({}_{c,d}{\bf z}_{p^n}^{(p)}(k,k,j,\xi,{{\rm prime}(pN)})\right)_{n\geq 1}$
recalled at the end of Section~\ref{subsec:Beilinson--Kato} to the $g$-isotypic Hecke eigenspace.
He introduces the normalization factor 
\[
\mu_{\rm Kato}(c,d,j,g):=\underbrace{(c^2-c^{k+1-j} \sigma_c) (d^2-d^{j+1}\sigma_d)}_{\mu_{\rm Kato}^{(0)}} \underbrace{\underset{\ell \mid  N}\prod (1- \overline{a}_\ell(g) \ell^{-k} \sigma_\ell^{-1})}_{\mathcal{E}_{\rm Kato}}\,,
\]
see \cite[\S13.9]{kato04}.

Suppose $x\in \mathcal{X}^{\rm cl}(E)$ is a classical point. As we have explained in Corollary~\ref{cor_thm_control_thm_interpolation_BK} (see also Remark~\ref{remark_compare_with_classical_Kato_newforms}), our big Beilinson--Kato element ${}_{c,d}\mathbb{BK}_{N,1}^{[\mathcal X]}(j,\xi)$ interpolates the elements 
$${}_{c,d}\mathrm{BK}_{Np,\Iw}(f_ x^c,j,\xi)\in H^1_{\rm Iw}(\ZZ[1/S], V_{ f_x^c}(w(x)+2))\,,\qquad x\in {\cX}^{\rm cl}(E)$$ 
which coincide with the $f_x^c$-isotypic projection of Kato's element $\left({}_{c,d}{\bf z}_{p^n}^{}(k+j+2,0,j+1,\xi,{{\rm prime}(pN)})\right)_{n\geq 1}$, where $w(x)=k+j$.
Kato's element  $\left({}_{c,d}{\bf z}_{p^n}^{}(k+j+2,0,j+1,\xi,{{\rm prime}(pN)})\right)_{n\geq 1}$ differs from 
\begin{equation}
\label{eqn_kato_wants_to_normalize_0}
\left({}_{c,d}{\bf z}_{p^n}^{}(k+j+2,k+j+2,j+1,\xi,{{\rm prime}(pN)})\right)_{n\geq 1}
\end{equation}
by the twist by $\chi^{w(x)+2}=\chi^{k+j+2}.$
Recall that we denote by  $\Tw_{m} \,:\, \LL\to \LL$ the twisting morphism which is given on the group like elements by  $\gamma\mapsto \chi^{-m}(\gamma)\gamma$. 

Since $f$ is new away from $p$, it follows from \cite[Lemma 2.7]{bellaiche2012} that one can shrink $\mathcal X$ as necessary to ensure that $f_x$ is new away from $p$ and $p$-old for all $x\in \mathcal{X}^{\rm cl}(E) \setminus\{x_0\}$. There are two scenarios:
\begin{itemize}
\item[i)] $f_x$ is a newform (so $x=x_0$ and $f_x=f$, owing to our choice of $\cX$). Then $f_x^c$ is a newform too.

\item[ii)] $f_x$ is the $p$-stabilization of a newform $f_x^\circ=:g$ of level $N$ coprime to $p$ with respect to the root $a_p(f_x)$ of the Hecke polynomial of $g$ at $p$. As we have explained in Remark~\ref{remark_compare_with_classical_Kato_newforms}, $f_x^c$ is a $p$-stabilization of the newform $g^c$ and one naturally identifies ${}_{c,d}\mathrm{BK}_{Np,\Iw}(f_ x^c,j,\xi)$ with ${}_{c,d}\mathrm{BK}_{N,\Iw}(g^c,j,\xi)$. In this scenario, we have $\mu_{\rm Kato}(c,d,j,f_x^c)=\mu_{\rm Kato}(c,d,j,g^c)$.
\end{itemize}
We then have, in both scenarios,
\begin{align*}
\Tw_{k+j+2}(\mu_{\rm Kato}(c,d,j,f_x^c))&=
\Tw_{k+j+2}\left ( {(c^2-c^{k+2} \sigma_c)} {(d^2-d^{j+2}\sigma_d)} 
{\underset{\ell \mid  N}\prod (1- {a}_\ell(f_x)\ell^{-k-j-2}\sigma_\ell^{-1})}\right )\\
&{=}{(c^2-c^{-j} \sigma_c)} {(d^2-d^{j-w(x)}\sigma_d)} 
{\underset{\ell \mid  N}\prod (1-  a_\ell(f_x)\sigma_\ell^{-1})}=:\mu(c,d,j,x)\,.
\end{align*}
In \cite[Section~13]{kato04}, Kato shows that there exists an element 
$\bz (g^c,j,\xi) \in H^1_{\Iw}\left (\ZZ [1/S],V_{g^c}(w(x)+2)\right )$ such that 
\[
{}_{c,d}\mathrm{BK}_{N,\Iw}( g^c,j,\xi)= \mu (c,d,j,x) \,\bz (g^c,j,\xi).
\]
We remark that in Kato's notation,  we have  $\bz (g^c,j,\xi)={\rm Tw}_{w(x)+2}(\bz_\gamma^{(p)})$, where 
$\gamma$ is the  cohomology class associated to $j$ and $\xi.$ Combining this fact with the preceding remarks, we deduce that
\begin{equation}
\label{eqn:comparision with Kato normalized element}
\mathbb{BK}_{N}^{[\mathcal X]}(j,\xi,x)= \mathcal E_N(x) \,\bz (g^c,j,\xi).
\end{equation}


\section{``\'Etale'' construction of $p$-adic $L$-functions}
\label{sec_geom_critical_padicL} 

Our objective in \S\ref{sec_geom_critical_padicL} is to use the $\cO_{\mathcal X}$-adic Beilinson--Kato element we have introduced in Definition~\ref{defn_big_BK_lambdaadicsheaf}(iii) and \ref{defn: normalized BK elements} to give an ``\'etale'' construction of Stevens' two-variable $p$-adic $L$-function in neighborhoods of non-$\theta$-critical cuspidal points on the eigencurve. In particular, we assume throughout \S\ref{sec_geom_critical_padicL} that $\mathcal{X}$ is \'etale over $\mathcal W$ (i.e. $e=1$); see our companion article \cite{BB_CK1_B} for the treatment of the complementary case (which concerns the neighborhoods of $\theta$-critical points on the eigencurve). 

Recall the representations $V_{\mathcal X}$ and $V_{\mathcal X}'$ given as in (\ref{eqn_defn_prop_andreattaiovitastevens2015_Prop318}). We will always assume that $\mathcal X$ is sufficiently small to satisfy  the conditions of Proposition~\ref{prop_partial_normalization_of_BK_elements}.  For each classical $x\in \mathcal X^{\rm cl}(E)$, we denote by $f_x$ the corresponding eigenform. The $\cO_{\mathcal X}$-adic Beilinson-Kato elements ${\mathbb{BK}}^{[\mathcal{X}]}_{N} (j, \xi)$ take coefficients in the cohomology of the universal cyclotomic twist of $V_{\mathcal X}^\prime$. 

Recall also that these representations are equipped with the $\cO_{\mathcal X}$-linear pairing
\[
(\,,\,)\,:\,V_{\mathcal X}'\otimes V_{\mathcal X} \lra \cO_{\mathcal X}
\]
(c.f. \eqref{formula: pairing for overconvergent cohomology over A}, bearing in mind that we have assumed $e=1$), which satisfies the formalism of Section~\ref{sec_Bellaiche_axiomatized}. We will use repeatedly the properties of this pairing summarized in Proposition~\ref{lemma_Adj_diagram_eigencurve}.

Let $f=f_0^{\alpha_0}$ be a $p$-stabilized non-$\theta$-critical cuspidal eigenform of weight $ k_0+2\geqslant 2$ and as before, which corresponds to the point $x_0$ on the cuspidal eigencurve. The eigencurve is \'etale at $x_0$ over the weight space (c.f. \cite{bellaiche2012}, Lemma 2.8).  In other words, $e=1$ and $\cO_{\mathcal X}=\cO_{\mathcal W}$. As we have explained in \S\ref{subsubsecKPX}, the simplification of the local behaviour of the eigencurve in this scenario exhibits itself also in the the $p$-local study (namely, the properties of the triangulation over the eigencurve). The reader might find it convenient to identify ${\mathcal X}$ with a Coleman family $\f$ through $f_0^{\alpha_0}$, so that the overconvergent eigenform $f_x$ that corresponds to a point $x\in \mathcal{X}(E)$ is simply the specialization $\f(x)$ of the Coleman family $\f$.

To simplify notation, we write $\mathscr{R}_{\mathcal X}$ for the relative  Robba ring $\mathscr{R}_{O_{\mathcal X}}$ and $\mathscr{H}_{\mathcal X}(\Gamma)$ for the relative  Iwasawa algebra $\mathscr{H}_{O_{\mathcal X}}(\Gamma)$. The $p$-adic representation $V_{\mathcal X}$ (given as in \eqref{eqn_defn_prop_andreattaiovitastevens2015_Prop318}) verifies the properties \ref{item_C1}--\ref{item_C3} and therefore comes equipped with a triangulation $\bD\subset \bD_{\rm rig,\cO_\mathcal{X}}^\dagger({V_{\mathcal{X}}})$ over the relative Robba ring $\mathscr{R}_{\mathcal X}$, verifying the properties \ref{item_phiGamma1}--\ref{item_phiGamma3}. Fix an $\cO_{\mathcal X}$-basis $\{\eta\}$ of $\DCc(\bD)$ and let $\eta_x\in \DCc(\bD_x)=\DCc(\bD[x])$ denote its specialization at $x\in \mathcal{X}(E)$ (equivalently, at weight $w(x)$).

Since $\cO_{\mathcal X}=\cO_{\mathcal W}$  in this setting (therefore also $\mathscr{H}_{{\mathcal X}}(\Gamma)= \mathscr{H}_{\cO_{\mathcal W}}(\Gamma)$), the $G_{\QQ,S}$-equivariant pairing $(\,,\,)$ in 
{(\ref{formula: pairing for overconvergent cohomology over A})}
is $\cO_{\mathcal X}$-linear. Recall from \eqref{eqn_PR_pairing_Iw} that this pairing induces a canonical $\mathscr{H}_{{\mathcal X}}(\Gamma)$-linear pairing
\begin{equation}
\label{eqn_PR_pairing_Iw_bis}
\left <\,\,,\,\,\right > \,:\,\left(\CH_{{\mathcal X}}(\Gamma)\otimes_{\cO_{\mathcal X}} H^1_{\Iw}(\Qp,{V_{\mathcal X}'(1)})\right)
\otimes  \left(\CH_{{\mathcal X}}(\Gamma)\otimes_{\cO_{\mathcal X}} H^1_{\Iw}(\Qp,{V_{\mathcal X}})^\iota\right) \longrightarrow \CH_{{\mathcal X}}(\Gamma)\,.
\end{equation}
Recall the Perrin-Riou exponential map 
$$\Exp_{\bD,0}: \CR_{{\mathcal X}}^{\psi=0}\otimes_A \DCc (\bD)\lra H^1_{\Iw}(\Qp,\bD)$$
as well as the canonical embedding $H^1_{\Iw}(\Qp,\bD)\hookrightarrow \CH_{{\mathcal X}}(\Gamma)\otimes_{\cO_{\mathcal X}} H^1_{\Iw}(\Qp,V_{\mathcal X})$ that the Fontaine--Herr complex gives rise to, through which we treat the image of $\Exp_{\bD,0}$ as a submodule of $ \CH_{{\mathcal X}}(\Gamma)\otimes_{\cO_{\mathcal X}} H^1_{\Iw}(\Qp,V_{\mathcal X})$. 

Consider the action of the complex conjugation on $H^1_{\Iw}(\ZZ [1/S],V^{\prime}_{\mathcal X}(1))$. For all integers $c,d$ as in Definition~\ref{defn: normalized BK elements}, $0\leqslant j\leqslant k_0$ and $\xi\in \mathrm{SL}_2(\ZZ)$, we denote by ${}_{c,d}\mathbb{BK}_N^{[\mathcal X], \pm}(j, \xi)\in H^1_{\Iw}(\Qp,V_{\mathcal X}(1))$  the $\pm$-parts of the Beilinson--Kato element ${}_{c,d}\mathbb{BK}_N^{[\mathcal X]}(j, \xi)$. We similarly define $\mathbb{BK}_N^{[\mathcal X], \pm}(j, \xi)\in H^1_{\Iw}(\Qp,V_{\mathcal X}(1))$ (c.f. Proposition~\ref{prop_partial_normalization_of_BK_elements}(ii)) for the partially normalized Beilinson--Kato element $\mathbb{BK}_N^{[\mathcal X]}(j, \xi):=\mu_0(c,d,j,x)^{-1}{}_{c,d}\mathbb{BK}_N^{[\mathcal X]}(j, \xi)$

Let us also denote by $\epsilon(j)$ the sign of $(-1)^j$.

\begin{defn}[Arithmetic non-$\theta$-critical $p$-adic $L$-function in two-variables]
\label{def_two_var_padicL_function_bis} 
We set (sic!)
\[
\begin{aligned}
&L_{p,\eta}^{+}(\f;j, \xi):= \left <\res_p\left (\mathbb{BK}_N^{[\mathcal X], \epsilon(k_0-1)}(j, \xi)\right )\,,\, c \circ \Exp_{\bD,0}( \widetilde\eta)^\iota \right >\in \mathscr{H}_{\mathcal X}(\Gamma)\\
&L_{p,\eta}^{-}(\f;j, \xi):=\left <\res_p\left ( \mathbb{BK}_N^{[\mathcal X], \epsilon(k_0)}(j, \xi)\right )\,,\, c \circ \Exp_{\bD,0}( \widetilde\eta)^\iota \right >\in \mathscr{H}_{\mathcal X}(\Gamma)\,,
\end{aligned}
\]
where $\widetilde\eta =\eta\otimes (1+\pi).$
\end{defn}



\begin{defn}
\label{def_two_var_padicL_function_bis_bis} 
For any $x\in \mathcal{X}(E)$, we put  $L_{p,\eta}^{\pm}(\f; j,\xi ,x)=x\circ {}L_{p,\eta}^{\pm}(\f;  j,\xi ,x)  \in \mathscr{H}_E(\Gamma)$. 
For a character $\rho\in X(\Gamma)$, we set $L_{p,\eta}^{\pm}(\f; j,\xi ,x,\rho):=\rho \circ {}L^{\pm}_{p,\eta}(\f;  j,\xi ,x)$.

For $x$ above, we let $f_x^\circ$ denote the newform associated to $f_x$. We write $\alpha(x)$ for the $U_p$-eigenvalue on $f_x$, so that when $f_x$ is $p$-old, it is the $p$-stabilization of the newform $f_x^\circ$ with respect to the root $\alpha(x)$ of the Hecke polynomial of $f_x^\circ$ at $p$. We let $L_{p,\alpha(x)}(f_x^\circ)$ 
denote the Manin--Vi\v{s}ik $p$-adic $L$-function associated to an unspecified choice of $($a pair of$)$ Shimura periods $($see, however, item~\ref{item_Kato_c} in the proof of Theorem~\ref{thm_interpolative_properties_bis} below, where this choice is made explicit$)$.
\end{defn}

The following result (in varying level of generality) has been previously announced in the independent works of Hansen, Ochiai and Wang.

\begin{theorem}
\label{thm_interpolative_properties_bis}
 There exists a neighborhood $\mathcal X$ of $x_0$ such that the following hold true. 

\item[i)] There exist { $(j_+,\xi_+)$ and $(j_-,\xi_-)$ such that $L_{p,\eta}^{\pm}(\f;j_\pm, \xi_\pm, x_0)$} are nonzero elements of $\CH_E(\Gamma)$.  

\item[ii)] Assume that { $(j_\pm,\xi_\pm)$} satisfy the conditions in (i). Then for each 
$x \in \mathcal{X}^{\rm cl}(E)$ such that $v_p(\alpha (x))<w(x)+1$,  the $p$-adic $L$-functions { $L_{p,\eta}^{\pm}(\f;j_\pm, \xi_\pm, x)$} agree with the Manin--Vi{\v{s}}ik  $p$-adic $L$-function $L_{p,\alpha(x)}(f_x^\circ)$, up to multiplication by $D^\pm {\mathcal E}_N(x)$ where $D^\pm\in E^\times$ and ${\mathcal E}_N(x)$ is the product of Euler factors at bad primes, given as in  \eqref{eqn: E(x)-factor}. 

\item[iii)] Assume that $v_p(\alpha_0)=k_0+1,$ but $f=\f(x_0)$ is not $\theta$-critical. Then { $L_{p,\eta}^{\pm}(\f;j_\pm, \xi_\pm, x_0)$} agree with the one variable $p$-adic $L$-functions of  Pollack--Stevens \cite{PollackStevensJLMS} up to multiplication by $D^\pm {\mathcal E}_N(x)$,  where $D^\pm\in E^\times$ is a constant.

\item[iv)] Let $L_p(\f,\Phi^{\pm}) $ denote the two-variable $p$-adic $L$-functions of Bella\"{\i}che and Stevens associated to modular symbols $\Phi^{\pm}\in \mathrm{Symb}_{\Gamma_p}^\pm(\mathcal X)$ (c.f. \cite{bellaiche2012}, Theorem~3). Then there exist  functions $u^\pm (x)\in O_{\mathcal X}^\times$
such that 
\[
{ L_{p,\eta}^{\pm}(\f;j_\pm, \xi_\pm)=u^\pm(x){{\mathcal E}_N(x)}L_p(\f,\Phi^{\pm})\,.}
\]
\end{theorem}

\begin{proof}
Let  $x \in \mathcal{X}^{\rm cl}(E)$ be a classical point as in the statement of the theorem.
Let $g$ denote the unique newform which admits $f_x$ as a $p$-stabilization if $f_x$ is $p$-old 
and $g=f_x$ if $f_x$ is new. By \eqref{eqn:comparision with Kato normalized element},  the specialization of the big Beilinson--Kato element at $x$ can be compared with the normalized Beilinson--Kato element for $g$:
\[
\mathbb{BK}_{N}^{[\mathcal X]}(j,\xi,x)= \mathcal E_N(x) \,\bz (g^c,j,\xi).
\]
Since Perrin-Riou's exponential map commutes with base-change, the specializations $L_{p,\eta}^{\pm}(\f;j, \xi, x)$ of $p$-adic $L$-functions $L_{p,\eta}^{\pm}(\f;j, \xi)$ 
coincide with one variable Perrin-Riou's $L$-functions constructed using the element $\mathbb{BK}_N^{[\mathcal X]}(j, \xi,x)$:
\[
\begin{aligned}
&L_{p,\eta}^{+}(\f;j, \xi, x):=\left <\res_p\left ( \mathbb{BK}_N^{[\mathcal X], \epsilon(k_0-1)}(j, \xi,x)\right )\,,\, c \circ \Exp_{\bD_x,0}( \widetilde\eta_x)^\iota \right >_x\in \mathscr{H}_{E}(\Gamma)\,,\\
&L_{p,\eta}^{-}(\f;j, \xi, x):=\left <\res_p\left ( \mathbb{BK}_N^{[\mathcal X], \epsilon(k_0)}(j, \xi,x)\right )\,,\, c \circ \Exp_{\bD_x,0}( \widetilde\eta_x)^\iota \right >_x\in \mathscr{H}_{E}(\Gamma)\,.
\end{aligned}
\]
We identify $\bz (g^c,j,\xi)$ with an element of $H^1_{\Iw}(\ZZ [1/S], V'_x(1))$
{\it via} the isomorphism $V'_x(1)\simeq V_{g}(w(x)+2)$ and define 
\[
\begin{aligned}
&L_{p,\eta}^{+}(g, j, \xi):=\left <\res_p\left ( \bz (g^c,j,\xi)^{\epsilon(k_0-1)}\right )\,,\, c \circ \Exp_{\bD_x,0}( \widetilde\eta_x)^\iota \right >_x\in \mathscr{H}_{E}(\Gamma)\,,\\
&L_{p,\eta}^{-}(g, j, \xi):=\left <\res_p\left ( \bz (g^c,j,\xi)^{\epsilon(k_0)}\right )\,,\, c \circ \Exp_{\bD_x,0}( \widetilde\eta_x)^\iota \right >_x\in \mathscr{H}_{E}(\Gamma).
\end{aligned}
\]
Therefore 
\begin{equation}
\label{eqn:comarision p-adic L-functions}
L_{p,\eta}^{\pm}(\f;j, \xi, x)=\mathcal E_N(x) L_{p,\eta}^{\pm}(g, j, \xi)\,.
\end{equation}
The following statements concerning $L_{p,\eta}^{\pm}(\f;j, \xi, x)$  are proved by Kato:
\begin{itemize}[leftmargin=*]
\item[\mylabel{item_Kato_a}{\bf a)}] There exist { $(j_+,\xi_+)$ and $(j_-,\xi_-)$} such that { $L_{p,\eta}^{+}(\f;j_+, \xi_+, x)$ and $L_{p,\eta}^{-}(\f;j_-, \xi_-, x)$} are non-zero functions. This follows from Kato's explicit reciprocity law \cite[Theorem~6.6]{kato04} together with Ash--Stevens theorem (see \cite{kato04}, Theorem~13.6), the non-vanishing results of Jacquet--Shalika and Rohrlich (see \cite{kato04}, Theorem~13.5) and the fact that the product ${\mathcal E}_N(x)\in \mathscr{H}_E(\Gamma)$ of bad Euler factors is a non-zero-divisor (c.f. Lemma~\ref{lemma_bad_factors_monodromy_conj}). 

\item[\mylabel{item_Kato_b}{\bf b)}]  For any pair $(j,\xi),$ there exist constants $C^{\pm}=C^{\pm}(j,\xi)\in E$ such that { $L_{p,\eta}^{\pm}(g;j, \xi, x)=C^\pm L_{p,\eta}^{\pm}(g;j_\pm, \xi_\pm, x)$.}  This follows from Kato's explicit reciprocity law and the fact that 
$H^1_{\Iw}(\ZZ [1/S], V'_x(1))$ is free of rank one over $\LL [1/p].$ See \cite[\S 13.9]{kato04}. 

\item[\mylabel{item_Kato_c}{\bf c)}]   If $v_p(\alpha (x))<w(x)+1,$ the functions { $L_{p,\eta}^{\pm}(g;j_\pm, \xi_\pm)$} agree  with the  Manin--Vi\v sik $p$-adic $L$-functions associated to the  Shimura periods that come attached to the choice of the classes  { $\delta (g,j_\pm,\xi_\pm)^\pm$} (see \cite{kato04} for precise definitions). 
\end{itemize}

Plugging in $x=x_0$ in  \eqref{eqn:comarision p-adic L-functions} we have 
$$ L_{p,\eta}^{\pm}(\f;j, \xi, x_0)=\mathcal E_N(x) L_{p,\eta}^{\pm}(g, j, \xi)\,.$$ 
Applying \ref{item_Kato_a} and \ref{item_Kato_b}, and noticing  that  ${\mathcal E}_N(x_0)\in \mathscr{H}_E(\Gamma)$  is a non-zero-divisor (c.f. Lemma~\ref{lemma_bad_factors_monodromy_conj}), the proof of Part (i) follows. 

Let us choose { $(j_+,\xi_+)$} as in (i), so that { $L_{p,\eta}^{+}(\f;j_+, \xi_+, x_0)$} is non-zero. In particular, for each even integer $1\leq i \leq p-1$, there exists a character $\chi_i\in X(\Gamma_1)$ so that { $L_{p,\eta}^{+}(\f;j_+, \xi_+, x_0, \omega^i\chi_i)\neq 0$}. Since { $L_{p,\eta}^{+}(\f;j_+, \xi_+, x, \omega^i\chi_i)$} is analytic in the variable $x$ (where $\omega$ is the Teichm\"uller character), there exists a neighborhood $\mathcal X$ of $x_0$ such that { $L_{p,\eta}^{+}(\f;j_+, \xi_+, x)$} is a non-zero-divisor of $\mathscr{H}_E(\Gamma)^+$, for all $x\in \mathcal X$. Similarly, one checks that { $L_{p,\eta}^{-}(\f;j_-, \xi_-, x)$} is a non-zero-divisor of $\mathscr{H}_E(\Gamma)^-$. Combining these observations with \ref{item_Kato_a}, \ref{item_Kato_b} and \ref{item_Kato_c} above, we conclude the proof of (ii). 

We next prove  Part (iv). Using the Amice transform, we can consider the functions $L_p(\f,\Phi^\pm)$
as elements of {$\CH_{\mathcal X}(\Gamma)$}. It follows from \cite[Theorem~3]{bellaiche2012} for classical non-critical points $x\in \mathcal X(E)$ that $L_p(\f,\Phi^+,x)$ coincides, up to multiplication by a non-zero constant, with the Manin--Vi\v sik $p$-adic $L$-function. We infer from (i) that for classical non-critical  points  $x\in \mathcal X^{\rm cl}(E)$, we have
\[
{ L_{p,\eta}^{+}(\f;j_+, \xi_+, x)  = u_x{\mathcal E}_N(x) L_p(\f,\Phi^+,x)}
\]
for some constant $u_x\in E$, with $u_{x_0}\neq 0$ by the choice of { $(j_+,\xi_+)$}. On expressing { $L_{p,\eta}^{+}(\f;j_+, \xi_+, x)$} and ${\mathcal E}_N(x) L_p(\f,\Phi^+,x)$ as power series with coefficients in $O_{\mathcal X}$, we immediately deduce that there exists a function $u^+(x)\in O_{\mathcal X} $ such that  $u^+(x_0)\neq 0$ and 
\[
{ L_{p,\eta}^{+}(\f;j_+, \xi_+, x)  =u^+(x)  {\mathcal E}_N(x) L_p(\f,\Phi^+,x).}
\]
On shrinking $\mathcal{X}$ as necessary, we can ensure that $u^+(x)$ does not vanish on $\mathcal X.$
The same argument proves that ${ L_{p,\eta}^{-}(\f;j_-, \xi_-, x) } =u^-(x) L_p(\f,\Phi^-,x)$
for some $u^-(x)\in O_{\mathcal X}.$  This concludes the proof of (iv). 

Part (iii) is a direct consequence of (iv) (as per the definition of the Pollack--Stevens $p$-adic $L$-function) and our theorem is proved.
\end{proof}

We note that in the particular case when $\f$ is the Coleman family passing through a newform $f_0=f^{\alpha_0}_0$ of level $\Gamma_0(Np)$ and weight $k_0+2$ with $\alpha_0=a_p(f_0)=p^{k_0/2}$, the conclusions of Theorem~\ref{thm_interpolative_properties_bis} play a crucial role in \cite{benoisbuyukboduk}. When $f=f_0^{\alpha_0}$ has critical slope but not $\theta$-critical, it is also crucially used in the proof of a conjecture of Perrin-Riou in \cite{BPSI}.



\appendix
\section{Integrality of normalizations}
\label{sec_appendix_integrality}

For any $x\in \mathcal{X}$, let us define $d^{w(x)}:=d^{k_0}\langle d \rangle^{w(x)-k_0}$. We define the partial normalization factor
\begin{equation}
\label{eqn_normalization_factor_in_families_bis}
\mu_0 (c,d,j,x):=\underbrace{(c^2-c^{-j} \sigma_c)}_{\mu_1(c,j)}\underbrace{(d^2-d^{j-w(x)}\sigma_d)}_{\mu_2(d,j,x)}  \in \LL_{\mathcal X}
\end{equation}
so that we have 
$$\mu(c,d,j,x)=\mu_0(c,d,j,x) \underbrace{\underset{\ell \mid  N}\prod (1- a_\ell(\f) \sigma_\ell^{-1})}_{{\mathcal E}_N}$$
for the normalization factor $\mu(c,d,j,x)$ given as in \eqref{eqn_normalization_factor_in_families}.

\subsection{The normalization factors revisited} 

Let ${\bbchi}_\cyc:G_\QQ\twoheadrightarrow \Gamma\hookrightarrow\LL^\times$ denote the universal cyclotomic character, so that ${\bbchi}_\cyc^{-1}$ gives the action on $\LL^{\iota}$. We also define the universal cyclotomic character parametrized by the weight space $\bbchi_{\rm wt}$ as the compositum
$$\bbchi_{\rm wt}: G_\QQ\twoheadrightarrow \Gamma\xrightarrow[\chi]{\sim}\ZZ_p^\times\hookrightarrow \ZZ_p[[\ZZ_p^\times]]^\times\hookrightarrow \cO_\cW^\times\,.$$
Let $\gamma\in \Gamma$ denote a fixed topological generator.

\begin{lemma}
\label{lemma_vanishing_locus_mu}
Suppose $c$ and $d$ are integers coprime to $p$. 
\item[i)] We have 
$$-c^{-j}\mu_1(c,j)=\sigma_c-\chi^{j+2}(\sigma_c).$$ 
In particular, if $c$ is a primitive root modulo $p^2$ (so that $\sigma_c$ generates $\Gamma$),  then $\mu_1(c,j)\,\dot{=}\, \gamma-\chi^{j+2}(\gamma)$, where ``$\dot{=}$'' means equality up to a unit multiple.
\item[ii)]  We have 
$$d^{j-w(x)}\mu_2(d,j,x)=\sigma_d-\bbchi_{\rm wt}\chi^{-j+2}(\sigma_d).$$ 
In particular, if $d$ is a primitive root modulo $p^2$ (so that $\sigma_d$ generates $\Gamma$),  then $\mu_2(d,j,x)\,\dot{=}\, \gamma-\bbchi_{\rm wt}\chi^{-j+2}(\gamma)$\,.
\item[iii)] Suppose $c$ and $d$ are primitive roots modulo $p^2$. For each $y\in \cX^{\rm cl}(E)$ with $w(y)\neq 2j$, the elements $\mu_1(c,j)$ and $\mu_2(d,j,y)$ of $\LL\otimes_{\ZZ_p}\cO_E$ are coprime.
\end{lemma}


\begin{proof}
Direct calculation.
\end{proof}




\subsection{Dependence on $c$ and $d$}
In this subsection, we prove the first part of Proposition~\ref{prop_partial_normalization_of_BK_elements}.

\begin{proposition}
\label{prop_independce_on_c_d_j_bis}
Assume that $\mathcal X$ is sufficiently small. 
\item[i)]{}
Suppose $c,c',d,d',j$ are auxiliary integers so that $(c,d,j)$ and $(c',d',j)$ are as in Definition~\ref{defn:Kato big-zeta}. Then
$$\mu_0(c',d',j,x)\, {}_{c,d}\mathbb{BK}_{N}^{[\mathcal X]}(j,\xi)= \mu_0(c,d,j,x)\,{}_{c',d'}\mathbb{BK}_{N}^{[\mathcal X]}(j,\xi)\,.$$

\item[ii)]{}
For $c,d,j$ as above, the partially normalized Beilinson--Kato element
$$\mathbb{BK}_{N}^{[\mathcal X]}(j,\xi):=\mu_0(c,d,j,x)^{-1}{}_{c,d}\mathbb{BK}_{N}^{[\mathcal X]}(j,\xi)\in H^1_{\rm Iw}(\ZZ[1/S],V_{\cX}'(1))\otimes_{\LL_\cX}\LL_{\cX}[\mu_0^{-1}]$$
is independent of the choice of $c$ and $d$. Here and elsewhere, $\mu_0$ is a shorthand for $\mu_0(c,d,j)$ when $c,d,j$ are understood.
\end{proposition}

Before we proceed with the proof of Proposition~\ref{prop_independce_on_c_d_j_bis}, we record the following auxiliary lemma. 

\begin{lemma}
\label{lemma_tor_free_means_finitely_divisible}
Suppose $R$ is an integral domain and $M$ is a finitely generated torsion-free $R$-module. Let $\{P_{i}\}_{i\in I}$ be  an infinite collection of prime ideals of $R$ such that $\bigcap_{ i\in I} P_{i}=\{0\}.$ Then,
$$\bigcap_{ i\in I} P_{i} M=\{0\}\,.$$
\end{lemma}

\begin{proof}[Proof of Lemma~\ref{lemma_tor_free_means_finitely_divisible}]
 Let $r$ be the rank of $M.$ Then $M$ contains a free submodule $M_0:=\underset{k=1}{\overset{r}\oplus} Re_k$  of rank $r$, and $M/M_0$ is annihilated by multiplication by some $a\in R.$ Let $m\in \bigcap_{i\in I} P_i M$ be any element. Then $am$ belongs to  $\bigcap_{i\in I} P_i M_0$. Since $am$ can be written in a unique way in the form
\[
am= \underset{k=1}{\overset{r}\sum} \alpha_ke_k, \qquad \alpha_k\in R,
\]
we see that $\alpha_k\in  \bigcap_{i\in I} P_i=\{0\}$ for all $k$.

\end{proof}

\begin{proof}[Proof of Proposition~\ref{prop_independce_on_c_d_j_bis}]
We will proceed in two steps.
\item[a)] We consider $H^1_{\rm Iw}(\ZZ[1/S],V_{\cX}'(1))$ as a module 
over $\LL_{\cX}(\Gamma_1).$
For any $x\in \cX^{\rm cl}(E),$ let $\mathfrak m_x\subset \LL_{\cX}(\Gamma_1)$ 
denote the kernel of the evaluation-at-$x$ map
\[
\LL_{\cX}(\Gamma_1)  \lra \LL_E(\Gamma_1),\qquad
\underset{i=0}{\overset{\infty}\sum} a_i(X)(\gamma_1-1)^i \longmapsto
\underset{i=0}{\overset{\infty}\sum} a_i(x)(\gamma_1-1)^i.
\]
It is clear from its definition that $\mathfrak m_x$ is the principal ideal generated by $(X-x).$
Then  $\bigcap_{x\in \cX^{\rm cl}(E)} \frak{m}_x =0.$ 
Combining Lemma~\ref{lemma_H1_Iw_torsion_free_small_X} and Lemma~\ref{lemma_tor_free_means_finitely_divisible},
we deduce for sufficiently small $\cX$ that
$$\bigcap_{x\in \cX^{\rm cl}(E)} \frak{m}_x H^1_{\rm Iw}(\ZZ[1/S],V_{\cX}'(1))=\{0\}\,.$$
\item[b)] It follows from \cite[\S13.9]{kato04} that 
\begin{align*}
\mu_0(c',d',j,x)\, {}_{c,d}\mathbb{BK}_{N}^{[\mathcal X]}(j,\xi)- \mu_0(c,d,j,x)\,{}_{c',d'}\mathbb{BK}_{N}^{[\mathcal X]}(j,\xi) &\in \ker \left( H^1_{\rm Iw}(\ZZ[1/S],V_{\cX}'(1))\to H^1_{\rm Iw}(\ZZ[1/S],V_{x}'(1))\right)\\
&\quad=\frak{m}_x H^1_{\rm Iw}(\ZZ[1/S],V_{\cX}'(1))
\end{align*}
for all $x\in \cX^{\rm cl}$. Now the part i) of the proposition  follows from a).
The part ii) is clear. 
\end{proof}

\subsection{Integrality of partial normalizations}

\subsubsection{}
We will next analyze the regularity of the partially normalized Beilinson--Kato element $\mathbb{BK}_{N}^{[\mathcal X]}(j,\xi)$. This amounts to an analysis of the divisibility of ${}_{c,d}\mathbb{BK}_{N}^{[\mathcal X]}(j,\xi)$ by $\mu_0(c,d,j,x)$. In view of Proposition~\ref{prop_independce_on_c_d_j_bis}, we may (and henceforth will) work with $c$ and $d$ which are both primitive roots mod $p^2$, without any loss.
\begin{proposition}
\label{prop_prop_partial_normalization_of_BK_elements_APPENDIX} 
Suppose $c$ and $d$ are primitive roots modulo $p^2$. There exists an affinoid neighborhood $\mathcal W$ of $k_0$ and a unique element 
$\mathbb{BK}_{N}^{[\mathcal X]}(j,\xi)\in H^1_{\Iw} \left (\ZZ [1/S],V^\prime_{\mathcal X} (1) \right )$  so that 
$${}_{c,d}\mathbb{BK}_{N}^{[\mathcal X]}(j,\xi)= \mu_0(d,j,x)\cdot\mathbb{BK}_{N}^{[\mathcal X]}(j,\xi)\,.$$
\end{proposition}

The uniqueness of $\mathbb{BK}_{N}^{[\mathcal X]}(j,\xi)$ follows from Lemma~\ref{lemma_H1_Iw_torsion_free_small_X} (on shrinking $\cX$ as necessary). The existence will be proved in the \S\ref{subsubsec_Appendix_A32} and \S\ref{subsubsec_Appendix_A32} below.

\subsubsection{}\label{subsubsec_Appendix_A32} We first prove that there exists an affinoid neighborhood $\mathcal W$ of $k_0$ such that 
\begin{equation}
\label{eqn:first part normalization prop}
{}_{c,d}\mathbb{BK}_{N}^{[\mathcal X]}(j,\xi)\in   \mu_1(c,j)\cdot H^1_{\Iw} \left (\ZZ [1/S],V^\prime_{\mathcal X} (1) \right ) 
\end{equation}
for all $\xi \in \mathrm{SL}_2(\ZZ)$. 

Considering the long exact sequence of Galois cohomology induced from
$$0\lra V_{\mathcal{X}}'(-j-1)\xrightarrow{ X} V_{\mathcal{X}}'(-j-1) \lra V_{f}'(-j-1)\lra 0$$
together with the vanishing of  $H^0(\ZZ[1/S],V_f'(-j-1))$ (which follows from the fact that 
$V_f$ is irreducible), we infer that the $\cO_{\mathcal X}$-module $H^1(\ZZ[1/S], V_{\mathcal{X}}'(-j-1))$ has no $X$-torsion. On shrinking $\mathcal{X}$ as necessary, we may therefore ensure that it is torsion free, and therefore free (since $\cO_{\mathcal{X}}$ is a PID) as a $\cO_{\mathcal{X}}$-module. As a matter of fact, on shrinking $\cX$, we can ensure that the $L$-value $L(f_x^c,w(x)-j)$ is non-central for any $x\in \mathcal{X}^{\rm cl}(E) \setminus \{x_0\}$. 
For all such $x$, it follows from \cite[Theorem~14.5]{kato04} and the isomorphism 
$V_{x}'(-j-1)\simeq V_{f_x^c}(w(x)-j)$ that  $H^1(\ZZ[1/S], V_{x}'(-j-1))$ is a $1$-dimensional $E$-vector space. This in turn shows that $H^1(\ZZ[1/S], V_{\mathcal{X}}'(-j-1))$ is a free $\cO_{\mathcal X}$-module of rank one.

The specialization map $\gamma \mapsto \chi^{j+2}(\gamma)$ induces an exact sequence 
\begin{equation}
\nonumber
\label{eqn_Prop617_fix_augment_mu_1}
0\lra V_{\mathcal{X}}'(1)\widehat{\otimes}_{\ZZ_p}\LL^{\iota} \xrightarrow{\gamma-\chi^{j+2}(\gamma)} V_{\mathcal{X}}'(1)\widehat{\otimes}_{\ZZ_p}\LL^{\iota} \xrightarrow{} V_{\mathcal X}'(-j-1)\lra 0,
\end{equation}

which gives rise to an exact sequence 
\[
H^1_{\rm Iw}(\ZZ[1/S], V_{\mathcal{X}}'(1))
\xrightarrow{\gamma-\chi^{j+2}(\gamma)}     H^1_{\rm Iw}(\ZZ[1/S], V_{\mathcal{X}}'(1))\lra H^1(\ZZ[1/S], V_{\mathcal{X}}'(-j-1)). 
\]
For each $x\in \mathcal X^{\mathrm{cl}}(E),$ let $\mathfrak m_x\subset O_{\mathcal X}$ denote
the maximal ideal $(X-x).$ We have a commutative diagram 
\[
\xymatrix{
 &&  & 0\ar[d]\\
  &&  & \mathfrak m_x H^1(\ZZ[1/S], V_{\mathcal{X}}'(-j-1))\ar[d]\\
H^1_{\rm Iw}(\ZZ[1/S], V_{\mathcal{X}}'(1)) \ar[rr]^{\gamma-\chi^{j+2}(\gamma)}  \ar[d]^{\mathrm{sp}_x}
&&H^1_{\rm Iw}(\ZZ[1/S], V_{\mathcal{X}}'(1))\ar[r] \ar[d]^{\mathrm{sp}_x} & H^1(\ZZ[1/S], V_{\mathcal{X}}'(-j-1))\ar[d]^{\mathrm{sp}_x}\\ 
H^1_{\rm Iw}(\ZZ[1/S], V_{x}'(1))
\ar[rr]^{\gamma-\chi^{j+2}(\gamma)} &&H^1_{\rm Iw}(\ZZ[1/S], V_{x}'(1))\ar[r]  &H^1(\ZZ[1/S], V_{x}'(-j-1)), 
}
\]
with exact rows and columns, where the vertical maps that are denoted by ${\rm sp}_x$ are induced by the specialization at $x$. Let us denote by 
\[
{}_{c,d}\mathbb{BK}_{N}^{[\mathcal X]}(j,\xi)\vert_{s=j+2}\in H^1(\ZZ[1/S], V_{\mathcal{X}}'(-j-1))
\]
the image of ${}_{c,d}\mathbb{BK}_{N}^{[\mathcal X]}(j,\xi)$ under the indicated cyclotomic specialization,
and by  
\[
{}_{c,d}\mathbb{BK}_{N}^{[\mathcal X]}(j,\xi, x)\in H^1_{\rm Iw}(\ZZ[1/S], V_{x}'(1))
\]
its image under $\mathrm{sp}_x$. Note that by Corollary~\ref{cor_thm_control_thm_interpolation_BK},
${}_{c,d}\mathbb{BK}_{N}^{[\mathcal X ]}(j,\xi, x)$ coincides with the classical Beilinson--Kato element 
${}_{c,d}\mathrm{BK}_{Np,\Iw}(f_x^c,j,\xi).$ Now Kato's integrality results in \cite[\S13.12]{kato04} (combined with  the discussion in Remark~\ref{remark_compare_with_classical_Kato_newforms} when $f_x$ is $p$-old) show  that 
$${}_{c,d}\mathbb{BK}_{N}^{[\mathcal X ,x]}(j,\xi)\in \mu(c,d,j,x) \cdot H^1_{\rm Iw}(\ZZ[1/S], V_{x}'(1))\subset \mu_1(c,j) H^1_{\rm Iw}(\ZZ[1/S], V_{x}'(1))$$ 
for all $c,d$ as before, integers $j\in [0,k_0]$ and $\xi\in {\rm SL}_2(\ZZ)$. 
Since  $\mu_1(c,j)\,\dot{=}\,\gamma-\chi^{j+2}(\gamma)$ by Lemma~\ref{lemma_vanishing_locus_mu}, 
an easy diagram chase shows that 
\[
{}_{c,d}\mathbb{BK}_{N}^{[\mathcal X]}(j,\xi)\vert_{s=j+2}\in \frak{m}_x H^1(\ZZ[1/S], V_{\mathcal{X}}'(-j-1)).
\]
It follows from Lemma~\ref{lemma_tor_free_means_finitely_divisible}  that 
${}_{c,d}\mathbb{BK}_{N}^{[\mathcal X]}(j,\xi)\vert_{s=j+2}=0.$ Therefore 
\[
{}_{c,d}\mathbb{BK}_{N}^{[\mathcal X]}(j,\xi)\in (\gamma-\chi^{j+2}(\gamma)) \cdot H^1_{\Iw} \left (\ZZ [1/S],V^\prime_{\mathcal X} (1) \right ),
\]
and \eqref{eqn:first part normalization prop} is proved. 

\subsubsection{} 
\label{subsubsec_Appendix_A33}
Let  $\frak{X}_{1,\mathcal{X}}^{(j,\xi)}\in H^1_{\Iw} \left (\ZZ [1/S],V^\prime_{\mathcal X} (1) \right )$ denote an element such that 
$
{}_{c,d}\mathbb{BK}_{N}^{[\mathcal X]}(j,\xi)=\mu_1(c,j)\, \frak{X}_{1,\mathcal{X}}^{(j,\xi)}.
$
We shall prove that there exists an affinoid neighborhood $\mathcal W$ of $k_0$ such that 
\begin{equation}
\label{eqn:normalization proposition part 2}
\frak{X}_{1,\mathcal{X}}^{(j,\xi)}\in   \mu_2(d,j,x)\cdot H^1_{\Iw} \left (\ZZ [1/S],V^\prime_{\mathcal X} (1) \right ). 
\end{equation}
The proof of this part is very similar to the verification of \eqref{eqn:first part normalization prop} in \S\ref{subsubsec_Appendix_A32}, after obvious modifications. 
We consider the Galois representation $V_{\mathcal{X}}'(\bbchi_{\rm wt}^{-1}\chi^{j-1})$,
in place of $V_{\mathcal{X}}'(-j-1)$. Note that the specialization of $V_{\mathcal{X}}'(\bbchi_{\rm wt}^{-1}\chi^{j-1})$
at $x$ is 
\[
V_{x}'(\chi^{-w(x)+j-1})\simeq V_{f_x^c}(j).
\]
 Employing the same argument as in the first paragraph of our proof of (\ref{eqn:first part normalization prop}), one observes that the $\cO_{\mathcal X}$-module $H^1(\ZZ[1/S], V_{\mathcal{X}}'(\bbchi_{\rm wt}^{-1}\chi^{j-1}))$ is free of rank one, on shrinking $\mathcal X$ if necessary.
The cyclotomic specialization $\gamma \mapsto \bbchi_{\rm wt}\chi (\gamma)^{2-j}$ gives an exact sequence 
\begin{equation}
\label{eqn_Prop617_fix_augment_mu_2}
0\lra V_{\mathcal{X}}'(1)\widehat{\otimes}_{\ZZ_p}\LL^{\iota} \xrightarrow{\gamma-\bbchi_{\rm wt}\chi^{-j+2}(\gamma)} V_{\mathcal{X}}'(1)\widehat{\otimes}_{\ZZ_p}\LL^{\iota} \xrightarrow{} V_{\mathcal{X}}'(\bbchi_{\rm wt}^{-1}\chi^{j-1})\lra 0\,.
\end{equation}
 Consider the following commutative diagram with exact rows and columns:
\[
\xymatrix{
 &&  & 0\ar[d]\\
  &&  & \mathfrak m_x H^1(\ZZ[1/S], V_{\mathcal{X}}'(\bbchi_{\rm wt}^{-1}\chi^{j-1}))\ar[d]\\
H^1_{\rm Iw}(\ZZ[1/S], V_{\mathcal{X}}'(1)) \ar[rr]^{\gamma-\bbchi_{\rm wt}\chi^{-j+2}(\gamma)}  \ar[d]^{\mathrm{sp}_x}
&&H^1_{\rm Iw}(\ZZ[1/S], V_{\mathcal{X}}'(1))\ar[r] \ar[d]^{\mathrm{sp}_x} & H^1(\ZZ[1/S], V_{\mathcal{X}}'(\bbchi_{\rm wt}^{-1}\chi^{j-1}))\ar[d]^{\mathrm{sp}_x}\\ 
H^1_{\rm Iw}(\ZZ[1/S], V_{x}'(1))
\ar[rr]^{\gamma-\bbchi_{\rm wt}\chi^{-j+2}(\gamma)} &&H^1_{\rm Iw}(\ZZ[1/S], V_{x}'(1))\ar[r]  &H^1(\ZZ[1/S], V_{x}'(\bbchi_{\rm wt}^{-1}\chi^{j-1})). 
}
\]
Let $\frak{X}_{1,\mathcal{X}}^{(j,\xi)}(x)$ denote the specialization of $\frak{X}_{1,\mathcal{X}}^{(j,\xi)}$ at $x\in \mathcal X^{\mathrm{cl}}(E).$ Since 
\[
\mu_1(c,j)\,\frak{X}_{1,\mathcal{X}}^{(j,\xi)}( x)={}_{c,d}\mathbb{BK}_{N}^{[\mathcal X ,x]}(j,\xi)\in H^1(\ZZ[1/S], V_{x}'(1))
\]
coincides with the classical Beilinson--Kato element 
\begin{equation}
\label{eqn_BK_Kato_specialized_y_vs_levelNp_appendix}
{}_{c,d}\mathrm{BK}_{Np,\Iw}(f_y^c,j,\xi)\in H^1_{\rm Iw}(\ZZ[1/S], V_{x}'(1))\,.
\end{equation}
Kato's integrality results in \cite[\S13.12]{kato04} (combined with the discussion in Remark~\ref{remark_compare_with_classical_Kato_newforms} when $f_x$ is $p$-old) show for such $x$ that 
$${}_{c,d}\mathbb{BK}_{N}^{[\mathcal X ]}(j,\xi,  x)\in \mu_0(c,d,j, x) \cdot H^1_{\rm Iw}(\ZZ[1/S], V'_{x}(1)),$$ 
for all $c,d$ as before, integers $j\in [0,k_0]$ and $\xi\in {\rm SL}_2(\ZZ).$
Since  
\[
\mu_0(c,d,j,{x})=\mu_1(c,j) \mu_2(d,j,x),
\]
and $ H^1_{\rm Iw}(\ZZ[1/S], V'_{x}(1))$ is torsion-free (and even free) $\LL [1/p]$-module, 
we deduce that 
\[
\frak{X}_{1,\mathcal{X}}^{(j,\xi)}(x)\in \mu_2(d,j,x) \cdot H^1_{\rm Iw}(\ZZ[1/S], V'_x(1)).
\]
Let $\frak{X}_{1,\mathcal{X}}^{(j,\xi)}\vert_{s=w-j+2}\in H^1(\ZZ[1/S], V_{\mathcal{X}}'(\bbchi_{\rm wt}^{-1}\chi^{j-1}))$ denote the cyclotomic specialization of $\frak{X}_{1,\mathcal{X}}^{(j,\xi)}.$
Recall  that  $\mu_2(d,j,x)\overset{\cdot}{=}\gamma-\bbchi_{\rm wt}\chi^{-j+2}(\gamma).$ 
Now a diagram chase similar to the one employed in \S\ref{subsubsec_Appendix_A32} shows that
\[
\frak{X}_{1,\mathcal{X}}^{(j,\xi)}\vert_{s=w-j+2}\in\mathfrak m_x H^1(\ZZ[1/S], V_{\mathcal{X}}'(\bbchi_{\rm wt}^{-1}\chi^{j-1})).
\]
Since the intersection of any infinite family of ideals $\mathfrak m_x$ is zero,  
we deduce using Lemma~\ref{lemma_tor_free_means_finitely_divisible} that 
\[
\frak{X}_{1,\mathcal{X}}^{(j,\xi)}\vert_{s=w-j+2} \in \bigcap_{x \in \mathcal{X}^{\rm cl}(E)}  \frak{m}_x H^1(\ZZ[1/S],  V_{\mathcal{X}}'(\bbchi_{\rm wt}^{-1}\chi^{j-1}))=\{0\}
\]
This shows that 
\[
\frak{X}_{1,\mathcal{X}}^{(j,\xi)}\in (\gamma-\bbchi_{\rm wt}\chi^{-j+2}(\gamma))
 H^1_{\Iw}(\ZZ[1/S],  V_{\mathcal{X}}'(1)),
\]
and (\ref{eqn:normalization proposition part 2}) is proved.  This also completes the proof of Proposition~\ref{prop_prop_partial_normalization_of_BK_elements_APPENDIX}. \qed

\begin{remark}
One reason why we need the rather tortuous analysis in Proposition~\ref{prop_prop_partial_normalization_of_BK_elements_APPENDIX}  is that we don't know if the following statement holds true: In the notation of \cite[Theorem 6.6]{kato04}, we have
$$\delta(f,j,\xi)^{\pm}\neq 0$$
for some $j\neq \frac{k_0}{2}\pm 1$, where $\pm$ is the sign of $(-1)^{\frac{k_0}{2}}$.
\end{remark}

\bibliographystyle{amsalpha}
\bibliography{references}

\end{document}